\documentclass[11pt]{amsart}

\usepackage{amsmath,amsfonts,amsthm,amssymb,amscd,mathabx}
\usepackage{epsfig}
\usepackage{inputenc}
\usepackage{upref}
\usepackage{bm}
\usepackage{calc}
\usepackage{caption}
\usepackage{subcaption}
\usepackage{hyperref}
\usepackage[usenames,dvipsnames]{xcolor}
\usepackage[normalem]{ulem}
\usepackage{wasysym}
\usepackage{mathrsfs}
\usepackage{enumitem}
\usepackage{setspace}
\usepackage{array}
\usepackage{cancel}
\usepackage{tikz-cd}
\usepackage{colonequals}
\usepackage[titletoc]{appendix}
\usepackage{overpic}
\usepackage{float} 
\usepackage{tikz}

\newtheorem{theorem}{Theorem}[section]
\newtheorem*{theorem*}{Theorem}	
\newtheorem{conjt}[theorem]{Conjecture}
\newtheorem{corollary}[theorem]{Corollary}
\newtheorem{lemma}[theorem]{Lemma}
\newtheorem{proposition}[theorem]{Proposition}

\theoremstyle{definition}

\newtheorem{defx}[theorem]{Definition}
\newenvironment{definition}
{\pushQED{\qed}\defx}
{\popQED\enddefx}

\newtheorem{examplex}[theorem]{Example}
\newenvironment{example}
{\pushQED{\qed}\examplex}
{\popQED\endexamplex}

\newtheorem{remarkx}[theorem]{Remark}
\newenvironment{remark}
{\pushQED{\qed}\remarkx}
{\popQED\endremarkx}

\numberwithin{equation}{section}

\DeclareMathOperator{\im}{Im}

\DeclareMathOperator{\End}{End}

\DeclareMathOperator{\Id}{Id}

\DeclareMathOperator{\Hom}{Hom}

\DeclareMathOperator{\Diff}{Diff}

\DeclareMathOperator{\rank}{rank}

\DeclareMathOperator{\Vol}{Vol}

\DeclareMathOperator{\Spinc}{Spin^c}
\DeclareMathOperator{\Spincr}{Spin^c_R}

\DeclareMathOperator{\Cob}{Cob}

\DeclareMathOperator{\gr}{\mathbf{gr}}

\DeclareMathOperator{\SW}{\underline{SW}}

\DeclareMathOperator{\Ob}{Ob}

  \DeclareMathOperator{\HM}{\it HM}
  \DeclareMathOperator{\LHM}{\it LHM}
\DeclareMathOperator{\KHM}{\it KHM}
\DeclareMathOperator{\SHM}{\it SHM}
\DeclareMathOperator{\SLL}{\it SL}

\DeclareMathOperator{\SFH}{SFH}
\DeclareMathOperator{\HFK}{HFK}

\newcommand{\R}{\mathbb{R}}
\newcommand{\C}{\mathbb{C}}

\newcommand{\Z}{\mathbb{Z}}

\newcommand{\T}{\mathbb{T}}

\newcommand{\BF}{\mathbb{F}}
\newcommand{\x}{\mathbb{X}}
\newcommand{\y}{\mathbb{Y}}
\newcommand{\BW}{\mathbb{W}}

\newcommand{\Step}{\textit{Step }}

\newcommand{\half}{\frac{1}{2}}

\newcommand{\CA}{\mathcal{A}}
\newcommand{\CB}{\mathcal{B}}
\newcommand{\SC}{\mathcal{C}}

\newcommand{\E}{\mathcal{E}}

\newcommand{\CG}{\mathcal{G}}

\newcommand{\K}{\mathcal{K}}

\newcommand{\NR}{\mathcal{R}}

\newcommand{\CL}{\mathcal{L}}
\newcommand{\CX}{\mathcal{X}}

\newcommand{\fa}{\mathfrak{a}}

\newcommand{\fd}{\mathfrak{d}}
\newcommand{\fm}{\mathfrak{m}}

\newcommand{\q}{\mathfrak{q}}

\newcommand{\s}{\mathfrak{s}}
\newcommand{\bs}{\widehat{\mathfrak{s}}}

\newcommand{\sw}{\mathfrak{M}}

\newcommand{\bpartial}{\bar{\partial}}

\newcommand{\hy}{\widehat{Y}}
\newcommand{\hx}{\widehat{X}}

\newcommand{\bomega}{\overline{\omega}}

\newcommand{\AT}{\mathbf{T}}
\newcommand{\AK}{\mathbf{K}}
\newcommand{\AR}{\mathbf{R}}
\newcommand{\CAT}{\textbf{CAT}}
\newcommand{\TSigma}{\mathbf{\Sigma}}

\newcommand{\spinc}{$spin^c\ $}

\newcommand{\tF}{\tilde{F}}

\newcommand{\Mod}{\text{-}\mathrm{Mod}}

\newcommand{\THM}{\widetilde{\HM}}
\newcommand{\tHM}{\widehat{\HM}}
\newcommand{\fHM}{\widecheck{\HM}}
\newcommand{\bHM}{\overline{\HM}}

\title{Monopoles And Landau-Ginzburg Models III:\\ A Gluing Theorem}
\author{Donghao Wang}

\date{\today}

\address{Department of Mathematics, Massachusetts Institute of Technology, Cambridge, MA 02139, USA}
\email{donghaow@mit.edu}

\begin{document}

	\setcounter{section}{1}
		
	\begin{abstract} This is the third paper of this series. In \cite{Wang20}, we defined the monopole Floer homology for any pair $(Y,\omega)$, where $Y$ is a compact oriented 3-manifold with toroidal boundary and $\omega$ is a suitable closed 2-form viewed as a decoration. In this paper, we establish a gluing theorem for this Floer homology when two such 3-manifolds are glued suitably along their common boundary, assuming that $\partial Y$ is disconnected, and $\omega$ is small and yet non-vanishing on $\partial Y$. 
		
		As applications, we construct a monopole Floer 2-functor and the generalized cobordism maps. Using results of Kronheimer-Mrowka and Ni, it is shown that  for any such 3-manifold $Y$ that is irreducible,  this Floer homology detects the Thurston norm on $H_2(Y,\partial Y;\R)$ and the fiberness of $Y$. Finally, we show that our construction recovers the monopole link Floer homology for any link inside a closed 3-manifold.
	\end{abstract}
	
	\maketitle
	\tableofcontents

\part{Introduction}

\subsection{An Overview} The Seiberg-Witten Floer homology of a closed oriented 3-manifold as introduced by Kronheimer-Mrowka \cite{Bible} has greatly influenced the study of 3-manifold
topology since its inception. The underlying idea is an infinite dimensional Morse theory: the monopole equations on $\R_t$ times a closed 3-manifold is interpreted as the downward gradient flow of the Chern-Simons-Dirac functional. 

This idea was further explored by the author in \cite{Wang20} in order to define the monopole Floer homology for any pair $(Y,\omega)$, where $Y$ is a compact oriented 3-manifold with toroidal boundary and $\omega$ is a suitable closed 2-form on $Y$. This Floer homology categorifies the Milnor-Turaev torsion invariant of $Y$ (this follows from the work of Meng-Taubes \cite{MT96} and Turaev \cite{T98}) and can be cast into a functor from a suitable cobordism category; see \cite[Theorem 1.5]{Wang20}. However, the second paper \cite{Wang20} was devoted to the analytic foundation of this Floer theory; very little was explored about its topological properties.

\medskip

The goal of the present paper is to understand the properties of this Floer homology in the special case that
\begin{enumerate}[label=($\star$)]
\item\label{star}  $\partial Y$ is disconnected, and $\omega$ has non-zero small pairing with each component of $\partial Y$.
\end{enumerate}

 Under this assumption, we will prove that the monopole Floer homology of $(Y,\omega)$ is a topological invariant: it depends only on the 3-manifold $Y$, the cohomology class $[\omega]\in H^2(Y; i\R)$ and an additional class in $H^1(\partial Y; i\R)$. Moreover, when $Y$ is irreducible, this invariant detects the Thurston norm on $H_2(Y,\partial Y; \R)$ and the fiberness of $Y$, generalizing the classical results \cite{KM97B,Bible,Ni08} for closed 3-manifolds. 

\medskip

In the context of Heegaard Floer homology, the knot Floer homology $\widehat{\HFK}_*$ and $\HFK_*^-$ were introduced by Oszv\'{a}th-Sz\'{a}bo \cite{KFH} and independently Rasmussen \cite{KFH1}. Motivated by the sutured manifold technique developed by Juh\'{a}sz \cite{J06,J08}, Kronheimer-Mrowka \cite{KS} defined the monopole knot Floer homology $\KHM_*$ as the analogue of the hat-version $\widehat{\HFK}_*$. 

One motivation of this project is to give an alternative definition of $\KHM_*$ so that a (3+1) TQFT property can be verified easily. This goal is accomplished in this paper: for any knot $K$ inside a closed 3-manifold $Z$, we prove that for a suitable choice of $\omega$ the monopole Floer homology of the link complement $Z\setminus N(K\cup m)$ is isomorphic to $\KHM_*(Z,K)$, where $m\subset Z\setminus K$ is a meridian. This confirms a longstanding speculation \cite{M16} that the knot Floer homology is related to the monopole equations on $\R_t$ times the link complement $Z\setminus N(K\cup m)$.

\medskip

Most topological implications of this paper follow immediately from a gluing theorem that computes the monopole Floer homology when two such 3-manifolds $(Y_1,\omega_1)$ and $(Y_2,\omega_2)$ are glued suitably along their common boundary; see Theorem \ref{T2.5} for the precise statement. Under the assumption \ref{star}, the gluing map
\begin{equation}\label{E1.1}
 \alpha: \HM_*(Y_1,\omega_1)\otimes_\NR \HM_*(Y_2,\omega_2)\to \HM_*(Y_1\circ_h Y_2, \omega_1\circ_h\omega_2)
\end{equation}
 is in fact an isomorphism and is functorial when considering both 3-manifold and 4-manifold cobordisms, where $\circ_h$ denotes horizontal composition and $Y_1\circ_h Y_2$ is the 3-manifold obtained after gluing. In fact, the (3+1) TQFT can be upgraded into a (2+1+1) TQFT: there is a monopole Floer 2-functor 
\[
\HM_*: \AT\to \AR,
\]
from a suitable cobordism bi-category $\AT$ to the strict 2-category $\AR$ of finitely generated $\NR$-modules. In this paper, we will always work with the mod 2 Novikov ring $\NR$ to avoid any orientation issues.

\medskip

With that being said, the gluing map \eqref{E1.1} is constructed simply using Floer's excision argument \cite{BD95,KS}, and the technique involved in the proof of the Gluing Theorem \ref{T2.5} is pretty standard. This is not a bad sign: the monopole Floer homology of $(Y,\omega)$ is expected to have intimate relations with the existing theory for closed 3-manifolds and for balanced sutured manifolds. This paper is providing the first few results towards this direction.

\subsection{Summary of Results}\label{Subsec1.2} For the benefit of the readers, we give a more detailed account of results that come out of the Gluing Theorem \ref{T2.5}. Let $Y$ be a connected, compact, oriented 3-manifold whose boundary $\partial Y\cong \Sigma\colonequals \coprod_{1\leq j\leq n} \T^2_j$ is a union of 2-tori.  The monopole Floer homology constructed in \cite{Wang20} relies on some auxiliary data on the surface $\Sigma$. In this paper, we focus on the special case that $\Sigma$ is disconnected and choose the following data in order:
\begin{itemize}
	\item a flat metric $g_\Sigma$ of $\Sigma$;
	\item an imaginary-valued harmonic 1-form $\lambda\in \Omega_h^1(\Sigma; i\R)$ such that  $\lambda|_{\T^2_j}\neq 0,\ 1\leq j\leq n$;
	\item an imaginary-valued harmonic 2-form $\mu\in \Omega_h^2(\Sigma; i\R)$  such that $|\langle \mu, [\T^2_j]\rangle|<2\pi$ and $\neq 0$, $1\leq j\leq n$. 
\end{itemize}

Such a quadruple $\TSigma=(\Sigma, g_\Sigma,\lambda,\mu)$ will be called a $T$-surface in this paper, and the triple $\fd=(g_\Sigma,\lambda,\mu)$ is called a geometric datum on $\Sigma$ in \cite{Wang20}. We choose a closed 2-form $\omega\in \Omega^2(Y; i\R)$ on $Y$ such that
\[
\omega=\mu+ds\wedge\lambda
\]
in a collar neighborhood $(-1,0]_s\times \Sigma\subset Y$. We denote such a pair $(Y,\omega)$ together with other auxiliary data in the construction by a thickened letter $\y$. 

\medskip

In \cite{Wang20}, the monopole Floer homology group $\HM_*(\y)$ of $\y$ is defined by studying the monopole equations on the non-compact 3-manifold $\hy\colonequals\  Y\ \bigcup_{\Sigma}\ [0,+\infty)_s\times\Sigma$ with perturbation given by the 2-form $\omega$. The geometric datum $\fd=(g_\Sigma,\lambda,\mu)$ is used here to specify the geometry along the cylindrical end $[0,+\infty)_s\times\Sigma$, and the class $[*_\Sigma\lambda]\in H^1(\Sigma; i\R)$ must lie in the image of $H^1(Y;i\R)$ in order for the compactness argument to work. By \cite[Theorem 1.4]{Wang20}, $\HM_*(\y)$ is a finitely generated module over $\NR$. However, the results from \cite{Wang20} did not guarantee that $\HM_*(\y)$ is a topological invariant: this group may a priori depend on the geometric datum $\fd$ in a subtle way. In this papar, we establish an invariance result in the special case described above:

\begin{theorem}[Theorem \ref{T5.1}]\label{T1.1} Under above assumptions, the monopole Floer homology group $\HM_*(\y)$ depends at most on the 3-manifold $Y$, the class $[\omega]\in H^2(Y;i\R)$ and $[*_\Sigma\lambda]\in \im (H^1(Y;i\R)\to H^1(\Sigma; i\R))$. This group is independent of the rest of data used in the construction up to canonical isomorphisms. 
\end{theorem}

The group $\HM_*(\y)$ admits a decomposition with respect to relative \spinc structures on $Y$:
\[
\HM_*(\y)=\bigoplus_{\bs\in \Spincr(Y)}\HM_*(\y,\bs). 
\]

For any closed irreducible 3-manifold $Z$, this grading information \cite{KM97B, Bible} determines the Thurston norm $x(\cdot)$ on $H_2(Z;\R)$. The Gluing Theorem \ref{T2.5} then allows us to relate $\HM_*(\y)$ with the monopole Floer homology of the double $Y\cup (-Y)$, so a similar detection result is obtained for any irreducible 3-manifold with disconnected toroidal boundary. However, our statement below is slightly different from the one in \cite{KM97B,Bible}, as the author was unable to verify the adjunction inequality. 

\begin{theorem}[Theorem \ref{T7.1}]\label{T1.2} Consider the set of monopole classes 
	\[
	\sw(\y)\colonequals \{ c_1(\bs): \HM_*(\y,\bs)\neq \{0\}\}\subset H^2(Y,\partial Y;\Z),
	\]
	and define $\varphi_{\y}(\kappa)\colonequals \max_{a\in \sw(\y)} \langle a, \kappa\rangle$, $\kappa\in H_2(Y,\partial Y;\R)$. We set $\varphi_{\y}\equiv -\infty$ if $\sw(\y)=\emptyset$. Then 
	\[
	\half(\varphi_{\y}(\kappa)+\varphi_{\y}(-\kappa))\leq x(\kappa),\ \forall \kappa\in H_2(Y,\partial Y;\R). 
	\]
	If in addition $Y$ is irreducible, then $\sw(\y)$ is non-empty, and the equality holds for any $\kappa$.
\end{theorem}
\begin{remark}\label{R1.3} Ideally, one would expect that $\sw(\y)$ is symmetric about the origin, and so $\varphi_{\y}(\kappa)=\varphi_{\y}(-\kappa)$ whenever $\sw(\y)$ is non-empty; but this symmetry is hard to verify directly due to the presence of the 2-form $\omega$. Nevertheless, the ideal adjunction inequality $\varphi_{\y}(\kappa)\leq x(\kappa)$ can be established in some special cases, e.g., when $\kappa$ is integral and $[\partial \kappa]\in H_1(\partial Y;\Z)$ is proportional to the Poincar\'{e} dual of $[*_\Sigma\lambda]$; see Corollary \ref{C8.4}.
\end{remark}

This Thurston norm detection theorem is accompanied with a fiberness detection result. In the context of Heegaard Floer homology, such a result was first established by Ghiggini \cite{Ghiggini08} for genus-1 fibered knots and by Ni for any knots \cite{Ni07} and for any closed 3-manifolds \cite{Ni09b} in general. In the context of the Seiberg-Witten theory, this was proved by Kronheimer-Mrowka \cite{KS} for the monopole knot Floer homology $\KHM_*$, and by Ni \cite{Ni08} for closed 3-manifolds. Our statement below, however, is slightly weaker than the ideal version that one would hope to prove due to the same reason explained in Remark \ref{R1.3}.

\begin{theorem}[Theorem \ref{T7.2}]\label{T1.4} For any integral class $\kappa\in H_2(Y,\partial Y;\Z)$, consider the subgroup
	\[
	\HM_*(\y|\kappa)\colonequals \bigoplus_{\langle c_1(\bs), \kappa\rangle=\varphi_{\y}(\kappa)}\HM_*(\y,\bs). 
	\]
	If $Y$ is irreducible and $\rank_\NR\HM_*(\y|\kappa)=\rank_\NR\HM_*(\y|-\kappa)=1$, then $Y$ fibers over $S^1$. 
\end{theorem}

The proof of Theorem \ref{T7.2} relies on the fiberness detection result \cite[Theorem 6.1]{KS} for sutured monopole Floer homology $\SHM$. In the original definition of Gabai \cite{Gabai83}, any 3-manifolds with toroidal boundary are examples of sutured manifolds, but they are not balanced in the sense of Juh\'{a}sz \cite[Definition 2.2]{J06}. Thus the sutured monopole Floer homology $\SHM$ is not previously defined for this class of sutured manifolds. One may regard our construction as a natural extension of $\SHM$ and ask if a sutured decomposition theorem, as in \cite[Theorem 1.3]{J08} and \cite[Theorem 6.8]{KS}, continues to hold at this generality. Theorem \ref{T8.1} is a preliminary result in this direction. Theorem \ref{T1.4} then follows immediately from Theorem \ref{T8.1} and the works of Kronheimer-Mrowka \cite{KS} and Ni \cite{Ni09b}. 

\medskip

Theorem \ref{T1.2} combined with Theorem \ref{T1.4} gives a simple characterization of the product manifold $[-1,1]_s\times \T^2$, which was first suggested to the author by Chris Gerig.

\begin{proposition}[Corollary \ref{C9.3}]\label{P0.5} Let $Y$ be any oriented 3-manifold with disconnected toroidal boundary. If $Y$ is connected and irreducible, then $\rank_{\NR}\HM_*(\y)\geq 1$. If the equality holds, then $Y=[-1,1]_s\times\T^2$ is a product. This is true for any choice of $(\omega,\fd)$ satisfying the conditions at the beginning of Section \ref{Subsec1.2}.
\end{proposition}

\subsection{Connected Sum Formulae} Two simple connected sum formulas can be derived from the Gluing Theorem \ref{T2.5} for reducible 3-manifolds. For any $\y_1$ and $\y_2$, take their connected sum to be
\[
\y_1\# \y_2=(Y_1\# Y_2, \omega_1\# \omega_2,\cdots). 
\]
The class $[\omega_1\# \omega_2]\in H^2(Y_1\#Y_2; i\R)$ is canonically determined by $\omega_j\in \Omega^2(Y_j; i\R),\ j=1,2$ and by the relation $\langle [\omega_1\# \omega_2], S^2\rangle=0$, where $S^2\subset Y_1\# Y_2$ is the 2-sphere separating $Y_1$ and $Y_2$. 

\begin{proposition}[Proposition \ref{P10.1}]\label{P1.6} Under above assumptions, we have 
	\[
	\HM_*(\y_1\# \y_2)=\HM_*(\y_1)\otimes_\NR\HM_*(\y_2)\otimes_\NR V,
	\]
	where $V$ is a two dimensional vector space over $\NR$. 
\end{proposition}
\begin{proposition}[Proposition \ref{P10.3}]\label{P1.7} For any closed 3-manifold $Z$, one can form the connected sum $\y\# Z$ in a similar way. Then 
	\[
	\HM_*(\y\# Z)\cong \HM_*(\y)\otimes_\NR\THM_*(Z)
	\]
	where $\THM_*(Z)$ is the sutured monopole Floer homology of $(Z(1),\delta)$, which is the balanced sutured manifold obtained from $Z$ by removing a 3-ball; the unique suture $\delta$ is the equator of $\partial Z(1)$; see Definition \ref{D10.2}. 
\end{proposition}

Proposition \ref{P1.6} and \ref{P1.7} are consistent with the connected sum formulas in \cite[Proposition 9.15]{J06} for sutured Heegaard Floer homology; the analogue for sutured monopole Floer homology was obtained in \cite[Theorem 1.5]{Li18b}.

Proposition \ref{P1.6} should be compared with a vanishing result, which is concerned with the case that the 2-form $\omega$ on $Y_1\# Y_2$ has non-zero pairing with the separating 2-sphere $S^2$.

\begin{proposition} Suppose $Y$ is any 3-manifold with toroidal boundary that contains an embedded 2-sphere $S^2\subset Y$. If in addition $\langle c_1(\bs),[S^2]\rangle =0$ and $\langle \omega, [S^2]\rangle \neq 0$, then the group $\HM_*(\y,\bs)$ vanishes. 
\end{proposition}

\begin{proof} This follows from the neck-stretching argument in \cite[Proposition 40.1.3]{Bible} and the energy equation in \cite[Proposition 5.4]{Wang20}.
\end{proof}

\subsection{Gerig's program in contact topology} This paper is motivated in part by Gerig's broader program to attack the two-or-infinite conjecture (see for instance \cite{HWZ98, HWZ03, HT09, GH16}) by mimicking Taubes' solution \cite{Taubes07} to the Weinstein conjecture in dimension 3. What he needed from the Seiberg-Witten theory is a much stronger result, which we state now as a conjecture.

	\begin{conjt}[Gerig]\label{conjecture} For any connected 3-manifold $Y$ with toroidal boundary, one can choose a closed 2-form $\omega$ and a geometric datum $\fd=(g_\Sigma,\lambda,\mu)$ on $\partial Y$ such that $\rank_\NR\HM_*(\y)\geq 1$. If $\rank=1$, then $Y$ is the connected sum of $[-1,1]_s\times \T^2$ with some monopole $L$-space. 
	\end{conjt}

In light of Proposition \ref{P0.5}, \ref{P1.6} and \ref{P1.7}, this conjecture holds if $Y$ is the connected sum $Y_1\#\cdots \# Y_m$ of prime 3-manifolds, with each $Y_j$ either closed or boundary-disconnected.

Gerig's approach towards the two-or-infinite conjecture can be described as follows. Suppose that $Z_0$ is any compact oriented 3-manifold, equipped with a contact 1-form  $\alpha_0$  with only finitely many Reeb orbits. Denote by $Y_0$ the complement of these orbits. If one can define an embedded contact homology for the pair $(Y_0,\alpha_0)$ and prove an ECH=SW type theorem, then $\rank_\NR\HM_*(\y)=\rank_\NR \text{ECH}(Y_0,\alpha_0)=1$. Conjecture  \ref{conjecture} then implies that $\partial Y_0$ has only two components. 

An appropriate candidate for this ECH theory has already appeared in \cite{CGH10} and \cite{HS05} assuming that all Reeb orbits of $\alpha_0$ are non-degenerate and elliptic. In this case, the Reeb flow of $\alpha_0$ foliates the boundary of $Y_0$ with irrational slopes, which is a technical condition needed for their construction. The main difficulty in Gerig's program is to find the correct 2-form $\omega_0$ on $Y_0$ and to establish this ECH=SW type theorem using the deformation $ \omega_0+td\alpha_0$ as $t\to +\infty$. The 2-form $\omega_0$ must be adapted to $\alpha_0$ in a certain sense so that Taubes' compactness argument goes through as usual. (Of course, Taubes' solution to the Weinstein conjecture in dimension 3 requires only a weaker version of ``ECH=SW", but it is conceptually easier to think of Gerig's strategy this way).

\subsection{Relations with Link Floer Homology} Our monopole Floer theory can be also used to recover the monopole link Floer homology $\LHM_*$. The next theorem will be proved using the Decomposition Theorem \ref{T8.1} in Section \ref{Sec8}.

\begin{theorem}[Corollary \ref{C8.3}]\label{T1.6} For any link $L=\{L_i\}_{i=1}^n$ in a closed oriented 3-manifold $Z$, we pick a meridian $m_i$ for each component $L_i$, and consider the link complement 
	\[
	Y(Z, L)=Z\setminus N(L\cup m_1\cup \cdots\cup m_n). 
	\]
	Then for a suitable 2-form $\omega$ on $Y=Y(Z,L)$, we have $
	\HM_*(\y)\cong \LHM_*(Z, L). $
\end{theorem}

By the work of Osv\'{a}th-Sz\'{a}bo\cite{OS08} and Ni \cite{Ni09}, the link Floer homology detects the Thurston norm of the link complement for any links in rational homology spheres. The assumption on homology is later removed by Juh\'{a}sz \cite[Remark 8.5]{J08}. The same detection result using monopole link Floer homology $\LHM_*$ was obtained by Ghosh-Li \cite[Theorem 1.17]{GL19}. By Theorem \ref{T1.6}, Theorem \ref{T1.2} recovers the previous detection results in the case that the link complement is irreducible. This constraint can be further removed by our connected sum formulae. 

On the other hand, any 3-manifold $Y$ with toroidal boundary can be viewed as the link complement of some link $L$ inside a closed 3-manifold $Z$. Then the link Floer homology of $(Z,L)$ detects the Thurston norm of $Y$. However, the statements in \cite{OS08, Ni09b,J08} and \cite{GL19} must rely on the choice of $(Z,L)$ and therefore have a radically different nature from the one in Theorem \ref{T1.2}.

\subsection{Generalized (3+1) TQFT} In \cite{Wang20}, the cobordism maps were constructed only for a very restrictive class of cobordisms called \textit{strict}. For any 3-manifolds $(Y_i,\Sigma_i)$ with toroidal boundary, $i=1,2$, a cobordism 
\[
(X,W): (Y_1,\Sigma_1)\to (Y_2,\Sigma_2)
\]
is a 4-manifold with corners, with $W$ a cobordism from $\Sigma_1$ to $\Sigma_2$. It is called strict, if $\Sigma_1=\Sigma_2$ and $W=[-1,1]_t\times\Sigma_1$ is a product. This constraint is removed by constructing generalized cobordism maps in this paper. Again we shall work under the assumptions stated in the beginning of Section \ref{Subsec1.2}.
\begin{theorem}[Corollary \ref{C4.4}]\label{T1.10} We define a category $\AT_2$ as follows: each object of $\AT_2$ is a 3-manifold $\y=(Y,\omega,\cdots)$ with toroidal boundary, equipped with a closed 2-form $\omega$ satisfying the conditions in Section \ref{Subsec1.2}. Each morphism is a triple $(\x,\BW, a)$ where
	\begin{itemize}
\item  $(X,W): (Y_1,\Sigma_1)\to (Y_2,\Sigma_2)$ is a 4-manifold with corners;
\item $a\in \HM_*(\BW)$ is an element in the monopole Floer homology of $\BW=(W,\omega_{W},\cdots)$.
	\end{itemize}
Then there is a functor $\HM_*: \AT_2\to \NR\Mod$ from $\AT_2$ to the category of finitely generated $\NR$-modules, which assigns to each $\y$ its monopole Floer homology $\HM_*(\y)$. 
\end{theorem}

\begin{remark} To better package the information on 2-forms, the actual construction of $\AT_2$ in Corollary \ref{C4.4} is slightly different. The gluing map \eqref{E1.1} has to be used in order to compose morphisms in the category $\AT_2$. This generalized (3+1) TQFT structure contains strictly less information than the monopole Floer 2-functor in Theorem \ref{T2.5}, since the functoriality of the element $a\in \HM_*(\BW)$ is completely ignored here.
\end{remark}

\subsection{A Remark on the Gluing Theorem} Most results of this paper are based on the monopole Floer 2-functor in the Gluing Theorem \ref{T2.5}, and we give a short conceptual reason here to see why it has a particularly simple form and how this gluing result fits into the broader context of Landau-Ginzburg models.

In the first paper \cite{Wang202}, the author introduced an infinite dimensional gauged Landau-Ginzburg model for any $T$-surface $\TSigma=(\Sigma,g_\Sigma,\lambda,\mu)$:
\[
(M(\Sigma), W_\lambda, \CG(\Sigma)). 
\]
It is more appropriate to think of $\HM_*(\y)$ as an invariant of $Y$ relative to this Landau-Ginzburg model. To develop a gluing theory in general, one has to assign an $A_\infty$-category $\CA$ to $(M(\Sigma), W_\lambda, \CG(\Sigma))$ and upgrade the Floer homology of $\y$ as an $A_\infty$-module over $\CA$. By the work of Seidel \cite[Corollary 18.27]{S08},  we wish to construct a spectral sequence abutting to $\HM_*(\y_1\circ_h\y_2)$ with $E_1$ page equal to 
\[
\HM_*(\y_1)\otimes_\NR\HM_*(\y_2). 
\]

In a finite dimensional framework, this spectral sequence has been reproved using the complex gradient flow equation in \cite{Wang22}, and a generalization to suitable infinite dimensional situations may come soon afterwards. However, in the special case \ref{star} that we have discussed so far, the $A_\infty$-category $\CA$ is expected to be trivial: the thimbles of different critical points of $W_\lambda$ are not supposed to intersect at all. This suggests that the spectral sequence collapses at $E_1$-page and the gluing formula is simply a tensor product. On the technical level, this allows us to take a shortcut in this paper by using Floer's excision argument, which bypasses the complicated general framework of \cite{Wang22}. 

\medskip

The next interesting case that we hope to investigate in the future is when $\Sigma$ is connected and $\mu=0$ in the geometric datum $\fd=(g_\Sigma,\lambda,\mu)$, which is also important for Gerig's program in contact topology. This may lead eventually to an analytic construction of the knot Floer homology $\HFK_*^-$ from our story. In this case, the gluing map \eqref{E1.1} constructed in this paper is identically zero, and the general framework of \cite{Wang22} will be indispensable. Readers are referred to \cite[Section 1.6]{Wang20} for more heuristics on this direction.

\begin{remark} We also note that the Gluing Theorem \ref{T2.5} is indeed subject to certain limitations. When $(Y_1,\omega_1)$ and $(Y_2,\omega_2)$ are glued along their common boundaries, the 2-forms $\omega_1$ and $\omega_2$ must match within a neighborhood of the boundary. This imposes a non-trivial homological constraint on the way they are glued. 
\end{remark}

\subsection{Organization} The rest of this paper is organized as follows:

\medskip

In Section \ref{Sec2}, we review the basic definition of bi-categories and state the Gluing Theorem \ref{T2.5}. Section \ref{Sec3} is devoted to its proof. We will follow closely the proof of Floer's excision theorem in \cite{BD95, KS}. In Section \ref{Sec4}, we construct the generalized cobordism maps and define the functor in Theorem \ref{T1.10}. In Section \ref{Sec5}, we prove the Invariance Theorem \ref{T1.1}. 

\medskip

In Section \ref{Sec6}, we review the theory for closed 3-manifolds following the book \cite{Bible} by Kronheimer-Mrowka and adapt their non-vanishing results to the case of non-exact perturbations. Section \ref{Sec7} is devoted to the Thurston norm detection result: Theorem \ref{T1.2}. After a digression into sutured Floer homology in Section \ref{Sec8}, the proof of the fiberness detection result, Theorem \ref{T1.4}, is supplied in Section \ref{Sec9}. The connected sum formulae are derived in Section \ref{Sec10}. 

\subsection*{Acknowledgments} The author is extremely grateful to his advisor, Tom Mrowka, for his patient help and constant encouragement throughout this project.  The author would like to thank Chris Gerig for suggesting Proposition \ref{P0.5}. The author would also like to thank Zhenkun Li for many helpful discussions. This material is based upon work supported by the National Science Foundation under Grant No.1808794.

\part{The Gluing Theorem}

The primary goal of this part is to construct the monopole Floer 2-functor
\[
\HM_*: \AT\to \AR
\]
where $\AT$ is the toroidal bi-category constructed in Section \ref{Sec2} and $\AR$ is the strict 2-category of finitely generated $\NR$-modules. Throughout this paper, the base ring $\NR$ is always the mod $2$ Novikov ring. Our main result is the Gluing Theorem \ref{T2.5}. 


\section{The Toroidal Bi-category}\label{Sec2}

\subsection{2-categories}
In this section, we review the definition of 2-categories following \cite{B67}. The $\Hom$-sets of a 2-category form categories themselves. 

\begin{definition} A strict 2-category $\AK$ consists of the following data:
	\begin{enumerate}[label=(Y\arabic*)]
		\item\label{Y1} a collections of objects $A,B,C\cdots$. They are also called $0$-cells;
		\item\label{Y2} for any objects $A,B\in \Ob \AK$, there is a category $\AK(A,B)$, whose objects are called 1-cells and morphisms are called 2-cells. The identity morphism of an 1-cell $f$ is denoted by $\Id_f:f\to f$. The compositions of 2-cells in $\AK(A,B)$ are called vertical compositions, denoted by $\circ_v$.
		\item\label{Y3} for any objects $A,B,C\in \Ob\AK$, there is a functor
		\[
		\circ_h: \AK(A,B)\times \AK(B,C)\to \AK(A,B),
		\]
		called the horizontal composition. 
		
		\item\label{Y4} The horizontal composition $\circ_h$ is associative, i.e. for any four 0-cells $A,B,C,D$, the two different ways of  composing $\circ_h$:
		\[
		-\ \circ_h\ (-\ \circ_h\ -) \text{ and }(-\ \circ_h\ -)\ \circ_h \ -
		\]
		give rise to the same functor from $\AK(A,B)\times\AK(B,C)\times \AK(C,D)$ to $\AK(A,D)$.
		\item\label{Y5} let $\mathbf{1}$ be the category with one object and one morphism; for any object $A\in \Ob \AK$, there is a functor $1_A: \mathbf{1}\to \AK(A,A)$ picking out the identity 1-cell $e_A: A\to A$ and its identity 2-cell $\Id_{e_A}:e_A\to e_A$. The functor $1_A$ is the unit of the horizontal composition $\circ_h$. \qedhere
	\end{enumerate}
\end{definition}	

\begin{example} Let $\NR$ be the Novikov field defined over $\BF_2$, the field of 2-elements:
	\[
	\NR\colonequals \BF_2[\R]^-=\{\sum_{i\geq 0} a_i q^{n_i}: a_i\in \BF_2, n_i\in \R, \lim_{i\to\infty} n_i=-\infty\}.
	\]
	We define $\AR$ to be the strict 2-category with a single object $\star$ such that $\AR(\star)\colonequals \AR(\star, \star)$ is the category of finitely generated $\NR$-modules. The horizontal composition is defined using the tensor product of $\NR$-modules. The identity 1-cell $e\in \AR(\star)$ is $\NR$ itself.
\end{example}

\begin{example} Let $\CAT$ be the strict 2-category consisting of all categories as 0-cells. For any categories $A,B$, 1-cells in $\CAT(A,B)$ are functors from $A$ to $B$, while 2-cells are given by natural transformations. 
\end{example}
	For technical reasons, the toroidal bi-category $\AT$ defined in the next subsection is not strictly associative. However, $\AT$ is still unital, and the associativity of $\circ_h$ still holds up to 2-isomorphisms; so it is an example of bi-categories:
	\begin{definition} A (unital) bi-category $\AK$ satisfies \ref{Y1}\ref{Y2}\ref{Y3}\ref{Y5} and 
		\begin{enumerate}[label=(Y4')]
			\item\label{Y4'} for any 0-cells $A,B,C,D$, there is a natural isomorphism $a(A,B,C,D)$ between the two functors below:
			\[
			\begin{tikzcd}
		\AK(A,B)\times\AK(B,C)\times \AK(C,D)\arrow[r, bend left, "-\ \circ_h\ (-\ \circ_h\ -)"] \arrow[r, bend right, "(-\ \circ_h\ -)\ \circ_h\ -"'] &\AK(A,D),
			\end{tikzcd}
			\] 
			which satisfies an associativity coherence condition \cite[P.5]{B67}. For any triple $(f,g,k)$, we use $a(f,g,k)$ to denote the 2-cell isomorphism:
			\[
			a(f,g,k): f\circ_h (g\circ_h k)\to (f\circ_h g)\circ_h k. \qedhere
			\]
					\end{enumerate}
			\end{definition} 
\subsection{The Toroidal Bi-category}\label{Subsec2.2} The primary goal of this subsection is to define the toroidal bi-category $\AT$, over which the monopole Floer 2-functor $\HM_*$ is defined. Each object of $\AT$ is a quadruple
\[
\TSigma=(\Sigma, g_\Sigma, \lambda, \mu),
\]
called a $T$-surface, where 
\begin{enumerate}[label=(T\arabic*)]
\item $\Sigma=\coprod_{i=1}^n \Sigma^{(i)}\neq \emptyset$ is an oriented surface consisting of finitely many 2-tori; we insist here that the surface $\Sigma$ is non-empty;
\item $g_\Sigma$ is a flat metric of $\Sigma$;
\item\label{T3} $\lambda\in \Omega^1_h(\Sigma, i\R)$ is a harmonic 1-form and $\mu\in \Omega^2_h(\Sigma, i\R)$ is a harmonic 2-form; when restricted to each connected component $\Sigma^{(i)}$, $\lambda$ and $\mu$ are both non-zero;
\item\label{T1} for any $1\leq i\leq n$, $|\langle \mu, [\Sigma^{(i)}]\rangle|<2\pi$.
\end{enumerate}


For any $T$-surfaces $\TSigma_1$ and $\TSigma_2$, $\AT(\TSigma_1, \TSigma_2)$ is a full subcategory of the strict cobordism category $\Cob_s$ defined in \cite[Section 3]{Wang20}. We recall the definition below for the sake of completeness. An object of $\AT(\TSigma_1, \TSigma_2)$ is a quintuple $\y=(Y, \psi, g_Y, \omega, \{\q\})$ satisfying the following properties:
\begin{enumerate}[label=(P\arabic*)]
	\item\label{P1} $Y$ is a compact oriented 3-manifold with toroidal boundary and $\psi : \partial Y\to (-\Sigma_1)\cup 
	\Sigma_2$ is an orientation preserving diffeomorphism. $\Sigma_1$ and $\Sigma_2$ are regarded as the incoming and outgoing boundaries of $Y$ respectively. When it is clear from the context, the identification map $\psi$ might be dropped from our notations.
	
	\item \label{P1.2} Each component of $Y$ intersects non-trivially with both $\TSigma_1$ and $\TSigma_2$.
	\item\label{P1.5} The metric $g_Y$ of $Y$ is cylindrical, i.e. $g_Y$ is the product metric 
	\[
	ds^2+\psi^*(g_{\Sigma_1}, g_{\Sigma_2})
	\]
	within a collar neighborhood $[-1,1)_s\times\Sigma_1\ \bigcup\ (-1,1]_s\times\Sigma_2$ of $\partial Y$.  
	
	\item\label{P2} $\omega\in \Omega^2(Y, i\R)$ is an imaginary valued \textbf{closed} 2-form on $Y$ such that 
	\begin{align*}
	\omega=\mu_1+ds\wedge\lambda_1& \text{ on } [-1,1)_s\times\Sigma_1,\\
	\omega=\mu_2+ds\wedge \lambda_2& \text{ on } (-1,1]_s\times\Sigma_2.
	\end{align*} 
	In particular, $(\mu_1,\mu_2)$ lies in the image $\im (H^2(Y; i\R)\to H^2(\Sigma; i\R))$.
	\item\label{P3} The cohomology class of $(*_1\lambda_1, *_2\lambda_2)$ lies in the image $
	\im (H^1(Y;i\R)\to H^1(\partial Y; i\R)).$  In particular, there exists a co-closed 2 form $\omega_{h}$ such that 
	\begin{align*}
\omega_h=ds\wedge\lambda_1& \text{ on } [-1,1)_s\times\Sigma_1;\\
\omega_h=ds\wedge \lambda_2& \text{ on } (-1,1]_s\times\Sigma_2;
\end{align*} 
	\item\label{P8} $\{\q\}$ is a collection of admissible perturbations (in the sense of \cite[Definition 13.3]{Wang20}) of the Chern-Simons-Dirac functional $\CL_\omega$, one for each relative \spinc structure of $Y$.
\end{enumerate}
\begin{remark} The reason to include \ref{P1.2} is to rule out closed components of $Y$ when considering horizontal compositions. This is not a serious problem; see Subsection \ref{Subsec2.4}.
\end{remark}

Having defined the objects (1-cells) of $\AT(\TSigma_1,\TSigma_2)$, let us now describe the set of morphisms (2-cells). Since each 3-manifold with toroidal boundary is now decorated by a closed 2-form $\omega$, morphisms will take these forms into account. Given two objects $\y_i=(Y_i, \psi_i, g_i, \omega_i,\{\q_i\}), i=1,2$ in $\AT(\TSigma_1,\TSigma_2)$, a morphism 
\[
\x: \y_1\to \y_2
\]
is a triple $\x=(X,[\psi_X],[\omega_X]_{cpt})$ satisfying the following properties.

\begin{figure}[H]
	\begin{tikzpicture}
	\draw (0,0) -- (3,0) -- (3,1.2) -- (0,1.2) -- (0,0);
	\draw  [ultra thick] (0,0) -- (3,0);
	\draw  [ultra thick] (0,1.2) -- (3,1.2);
	\node at (1.5,0.6) {$X$};
	\node [left]at (0,0.6) {\small $[-1,1]_t\times\Sigma_1$};
		\node [right] at (3,0.6) {\small$[-1,1]_t\times\Sigma_2$};
		\node [above] at (1.5,1.2) {$Y_1$};
				\node [below] at (1.5,0) {$Y_2$};
	\node [above] at (3,1.2) {$\Sigma_2$};
		\node [above] at (0,1.2) {$\Sigma_1$};
				\node [below] at (0,0) {$\Sigma_1$};
					\node [below] at (3,0) {$\Sigma_2$};
	\node at (9,0.6) {$\begin{tikzcd}[column sep=2cm]
	\TSigma_1\arrow[r,"\y_1"]\arrow[rd,"\x",phantom] \arrow[d,equal]& \TSigma_2\arrow[d,equal]\\
	\TSigma_1\arrow[r, "\y_2"'] & \TSigma_2. 
	\end{tikzcd}$}; 
\node at (6,0.6) {$\rightsquigarrow$};
	\end{tikzpicture}
	\caption{A 2-cell morphism $\x$.}
	\label{Pic0}
\end{figure}
\begin{enumerate}[label=(Q\arabic*)]
	\item \label{Q1}$X$ is a 4-manifold with corners, i.e. $X$ is a space stratified by manifolds 
	\[
	X\supset X_{-1}\supset X_{-2}\supset X_{-3}=\emptyset
	\]
	such that the co-dimensional 1 stratum $X_{-1}$ consists of three parts
	\[
	X_{-1}= (-Y_1)\cup (Y_2)\cup W_X.
	\]
	where $W_X$ is an oriented 3-manifold with boundary $\partial W_X=\partial Y_1\cap\partial Y_2$.  Moreover, $\partial Y_i=Y_i\cap W_X$ and $X_{-2}=\partial Y_1\cup \partial Y_2$.  
	
	\item\label{Q3} $\psi_X: W_X\to [-1,1]_t\times ((-\Sigma_1)\cup\Sigma_2)$ is an orientation preserving diffeomorphism compatible with $\psi_1$ and $\psi_2$. More precisely, we require that 
	\begin{align*}
	\psi_X |_{\partial Y_1}&=\psi_1: \partial Y_1\to \{-1\}\times ((-\Sigma_1)\cup\Sigma_2), \\
	\psi_X |_{\partial Y_2}&=\psi_2: \partial Y_2\to \{1\}\times ((-\Sigma_1)\cup\Sigma_2),
	\end{align*}
	which hold also in a collar neighborhood of $\partial W_X$.  When no chance of confusion is possible, $\psi_X$ might be dropped from our notations. Such a pair $(X, \psi_X)$ is called \textbf{a strict cobordism} from $(Y_1,\psi_1)$ to $(Y_2,\psi_2)$. $[\psi_X]$ denotes the isotopy class of such a diffeomorphism. 
		\item\label{Q6} There exists a closed 2-form $\omega_X\in \Omega^2(X, i\R)$ on $X$ satisfying the following properties. 
	\begin{itemize}
		\item $\omega_X=\omega_i $ (see \ref{P2}) within a collar neighborhood of $Y_i\subset X_{-1}$ for $i=1,2$;
		\item within a collar neighborhood of $W_X\subset X_{-1}$, 
		\begin{align*}
	\omega_X&=\mu_1+ds\wedge \lambda_1 \text{ on } [-1,1]_t\times [-1,1)_s\times\Sigma_1,\\
		\omega_X&=\mu_2+ds\wedge \lambda_2\text{ on } [-1,1]_t\times (-1,1]_s\times\Sigma_2.
		\end{align*}
	\end{itemize}
	The existence of such a 2-form $\omega_X$ is guaranteed by a cohomological condition on $[\omega_X]\in H^2(X; i\R)$; see \cite[Section 3 (Q4)]{Wang20}. Any two such forms $\omega_X,\omega_X'$ are called relative cohomologous if $\omega_X-\omega_X'=da$ for a compactly supported smooth 1-form $a\in \Omega^1_c(X,i\R)$. We fix the relative cohomology class of $\omega_X$, denoted by $[\omega_X]_{cpt}$. 
\end{enumerate}

\begin{remark} It is necessary to record the isotopy class of $\psi_X$ here, because the diffeomorphism group $\Diff_+(\T^2)$ is not simply connected. By \cite[Theorem 1(b)]{EE67}, $\Diff_+(\T^2)$ has the same homotopy type of its linear subgroup $S^1\times S^1\times \SLL(2,\Z)$, so $\pi_1(\Diff_+(\T^2))\cong \Z\oplus\Z$. 
\end{remark}

The vertical composition of $\AT(\TSigma_1, \TSigma_2)$ is defined by composing strict cobordisms. Since we have recorded the relative cohomology class $[\omega_X]_{cpt}$ (instead of just $[\omega_X]\in H^2(X;i\R)$), these classes can be concatenated accordingly on the composed manifold. For any 1-cell $\y\in \AT(\TSigma_1, \TSigma_2)$, the identity 2-cell is given by the product strict cobordism $([-1,1]_t\times Y, \Id_{[-1,1]_t} \times \psi)$, with $[\omega]_{cpt}$ being the class of the pull-back 2-form $\omega$. A metric of $X$ is \textbf{not} encoded in the definition of $\x$. This category is topological, although auxiliary data are specified for its objects (1-cells).

The horizontal composition $\AT(\TSigma_1, \TSigma_2)\times \AT(\TSigma_2,\TSigma_3)\to \AT(\TSigma_1,\TSigma_3)$ is defined using the diffeomorphisms $\psi$. On the level of 1-cells, given any pair $(\y_{12},\y_{23})\in \AT(\TSigma_1, \TSigma_2)\times \AT(\TSigma_2,\TSigma_3)$, the underlying 3-manifold of $\y_{12}\circ_h \y_{23}$ is formed by gluing
\[
Y_{12}\setminus [0,1]_s\times \psi_{12}^{-1}(\Sigma_2) \text{ and } Y_{23}\setminus [-1,0]_s\times \psi_{23}^{-1}(\Sigma_2) 
\] 
along the common boundary $\{0\}\times \Sigma_2$, using the composition
\[
\psi_{12}^{-1} (\Sigma_2)\xrightarrow{\psi_{12}} \Sigma_2 \xrightarrow{\psi_{23}^{-1}}\psi_{23}^{-1}(\Sigma_2),
\]
which is orientation reversing. The condition \ref{P3} is certified by concatenating the co-closed 2-form $(\omega_h)_{12}$ with $(\omega_h)_{23}$. We shall formally write:
\[
\y_{12}\circ_h \y_{23}=(Y_{12}\circ_h Y_{23},\ \omega_{12}\circ_h\omega_{23},\ \cdots). 
\]

The problem arises from \ref{P8}: it is not guaranteed that two admissible perturbations on $Y_{12}$ and $Y_{23}$ can be concatenated in a canonical way to form an admissible perturbation on $Y_{12}\circ_h Y_{23}$. Instead, we will pick a random collection of admissible peturbations for $\y_{12}\circ_h \y_{23}$ making $\circ_h$ not strictly associative. 

As for 2-cells, their horizontal compositions are formed similarly using $\psi_X$ instead. 

\medskip

Let us now construct the natural isomorphism $a(\TSigma_1, \TSigma_2,\TSigma_3,\TSigma_4)$ in \ref{Y4'}: for any triple $(\y_{12},\y_{23},\y_{34})$, define
\[
a(\y_{12},\y_{23},\y_{34}): \y_{12}\circ_h(\y_{23}\circ_h\y_{34})\to (\y_{12}\circ_h\y_{23})\circ_h\y_{34}.
\]
to be the \textit{identity} 2-cell. Indeed, as 1-cells in $\AT(\TSigma_1,\TSigma_4)$, they have the same underlying metrics and closed 2-forms; only the admissible perturbations may differ from one another. 

\smallskip

Finally, for any 0-cell $\TSigma$, we define its  identity 1-cell $e_{\TSigma}$ as
\[
(Y=[-1,1]_s\times \Sigma,\ \psi=\Id,\ g_Y=d^2s+g_{\Sigma},\  \omega=\mu+ds\wedge\lambda,\ \{\q=0\}).
\]
For any 1-cell $\y_{12}\in \AT(\TSigma_1, \TSigma_2)$, one may set $\y_{12}\circ_h e_{\Sigma_2}$ and $e_{\Sigma_1}\circ_h \y_{12}$ to be just $\y_{12}$, as they already have the same underlying metrics and closed 2-forms by our conventions of horizontal compositions. In this way, the toroidal bi-category $\AT$ becomes strictly unital. 

\subsection{The Monopole Floer 2-Functor} \label{Subsec2.3}The primary goal of this paper is to define the monopole Floer 2-functor:
\[
\HM: \AT\to \AR.
\]

We expand on the requirement for a 2-functor in the theorem below:
\begin{theorem}[The Gluing Theorem]\label{T2.5} There exists a 2-functor $\HM$ from the toroidal bi-category $\AT$ to the strict 2-category $\AR$ of finitely generated $\NR$-modules satisfying the following properties:
	\begin{enumerate}[label=(G\arabic*)]
\item\label{G1} for any T-surfaces $\TSigma_1, \TSigma_2$, the functor \[
\HM_*: \AT(\TSigma_1,\TSigma_2)\to \AR(\star)
\]
is defined as in \cite[Theorem 1.5\ \&\ Remark 1.7]{Wang20}, which assigns each 1-cell $\y$ to its monopole Floer homology group $\HM_*(\y)$;
\item\label{G2} for any $\TSigma_1, \TSigma_2,\TSigma_3$, there is a natural isomorphism $\alpha(\TSigma_1, \TSigma_2,\TSigma_3)$ between the two compositions in the digram below:
\[
\begin{tikzcd}[column sep=2cm]
\AT(\TSigma_1,\TSigma_2)\times \AT(\TSigma_2,\TSigma_3)\arrow[rd,"\circ_h"']\arrow[r,"\HM_*\times \HM_*"] & \AR(\star)\times\AR(\star) \arrow[r,"\circ_h"]& \AR(\star)\\
 & \AT(\TSigma_1,\TSigma_3) \arrow[ru,"\HM_*"'].&
\end{tikzcd}
\]
In other words, for any composing pair $(\y_{12},\y_{23})$, there is an isomorphism of $\NR$-modules:
\[
\alpha:  \HM_*(\y_{12})\otimes_\NR\HM_*(\y_{23})\to \HM_*(\y_{12}\circ_h \y_{23}),
\]
that is natural with respect to the 2-cell morphisms in $\AT(\TSigma_1,\TSigma_2)$ and $ \AT(\TSigma_2,\TSigma_3)$;
\item\label{G3} $\alpha$ is associative meaning that the digram 
\begin{equation}\label{E2.1}
\begin{tikzcd}
\HM_*(\y_{12})\otimes \HM_*(\y_{23})\otimes \HM_*(\y_{34})\arrow[r,"\Id\otimes \alpha"] \arrow[d,"\alpha\otimes \Id"]&  \HM_*(\y_{12})\otimes \HM_*(\y_{23}\circ_h\y_{34})\arrow[d,"\alpha"]\\
\HM_*(\y_{12}\circ \y_{23})\otimes \HM_*(\y_{34})\arrow[dr,"\alpha"] &\HM_*(\y_{12}\circ_h(\y_{23}\circ_h \y_{34})) \arrow[d, "{\HM_*(a(\y_{12},\y_{23},\y_{34}))}","\cong"']\\
& \HM_*((\y_{12}\circ_h\y_{23})\circ_h \y_{34}) 
\end{tikzcd}
\end{equation}
is commutative for any triples $(\y_{12},\y_{23},\y_{34})\in \AT(\TSigma_1,\TSigma_2)\times \AT(\TSigma_2,\TSigma_3)\times \AT(\TSigma_3,\TSigma_4).$
\item\label{G4} for any T-surface $\TSigma$ and its identity $1$-cell $e_{\TSigma}$, there is a canonical isomorphism 
\[
\iota_{\TSigma}: \HM_*(e_{\TSigma})\to \NR,
\]
such that the gluing map
\[
\alpha: \HM_*(\y_{12})\otimes \HM_*(e_{\TSigma_2})\to \HM_*(\y_{12}\circ_h e_{\TSigma_2})=\HM_*(\y_{12}),
\]
is simply $\Id \otimes \iota_{\TSigma}$. A similar property holds also for $e_{\TSigma_1}\circ_h \y_{12}$.
\end{enumerate}
\end{theorem}

\subsection{A convention for $\AT(\emptyset,\TSigma)$}\label{Subsec2.4} If we allow the empty surface $\emptyset$ to be a 0-cell of the toroidal bi-category $\AT$, then we must allow the underlying 3-manifold of an 1-cell $\y=(Y,\omega,\cdots)$ to have closed components. This is not a serious problem, as long as on each closed component, the 2-form $\omega$ is never balanced or negatively monotone with respect to any \spinc structures, in the sense of \cite[Definition 29.1.1]{Bible}. This will allow us to apply the adjunction inequality in Proposition \ref{P4.2}.  Instead of setting up the theory at this generality, we will simply define
\[
\AT(\emptyset,\TSigma), \AT(\TSigma, \emptyset),
\]
as categories in their own rights, and do not regard them as part of the bi-category $\AT$. We can still define the horizontal composition
\[
\circ_h: \AT(\emptyset,\TSigma_1)\times \AT(\TSigma_1,\TSigma_2)\to \AT(\emptyset,\TSigma_2)
\]
in the usual way and make the assignment
\[
\TSigma\mapsto \AT(\emptyset,\TSigma)
\]
into a 2-functor (in a suitable sense) from $\AT$ to $\CAT$. However, the latter point of view is not needed for the purpose of this paper.

\medskip

The category $\AT(\emptyset,\emptyset)$ consists of closed 3-manifolds equipped with imaginary valued closed 2-forms that are never balanced or negatively monotone on each component.

\section{Proof of the Gluing Theorem}\label{Sec3}
In this section, we present the proof of the Gluing Theorem \ref{T2.5}. The construction of the gluing map $\alpha$ in Theorem \ref{T2.5} \ref{G2} is based upon Floer's Excision Theorem \cite{BD95}, which has been adapted to the monopole Floer homology by Kronheimer-Mrowka \cite{KS}. We will follow the setup of \cite[Theorem 3.1\ \&\ 3.2]{KS} closely.

This section starts with an overview of the monopole Floer homology defined in \cite{Wang20}, which yields the monopole Floer functor $\HM_*$ in Theorem \ref{T2.5} \ref{G1}. The gluing map $\alpha$ is then constructed in Subsection \ref{Subsec3.3}.

\subsection{Review} Recall that a \spinc structure $\s_X$ on an oriented Riemannian 4-manifold $X$ is a pair $(S_X,\rho_4)$ where $S_X=S^+\oplus S^-$ is the spin bundle, and the bundle map $\rho_4: T^*X\to \End(S_X)$ defines the Clifford multiplication. A configuration $\gamma=(A,\Phi)\in \SC(X,\s)$ consists of a smooth \spinc connection $A$ and a smooth section $\Phi$ of $S^+$. Use $A^t$ to denote the induced connection of $A$ on $\bigwedge^2 S^+$. Let $\omega_X\in \Omega^2(X, i\R)$ be a closed 2-form on $X$ and $\omega^+_X$ denote its self-dual part. The Seiberg-Witten equations perturbed by $\omega_X$ are defined on the configuration space $\SC(X,\s)$ by the formula:
\begin{equation}\label{SWEQ}
\left\{
\begin{array}{rl}
\half \rho_4(F_{A^t}^+)-(\Phi\Phi^*)_0&=\rho_4(\omega^+_X),\\
D_A^+\Phi&=0,
\end{array}
\right.
\end{equation}
where $D_A^+: \Gamma(S^+)\to \Gamma(S^-)$ is the Dirac operator and $(\Phi\Phi^*)_0=\Phi\Phi^*-\half |\Phi|^2\otimes\Id_{S^+}$ is the traceless part of the endomorphism $\Phi\Phi^*:S^+\to S^+$. 

\smallskip

When it comes to an oriented Riemannian 3-manifold $Y$, the dimensional reduction of \eqref{SWEQ} is obtained by looking at \eqref{SWEQ} on the product manifold $\R_t\times Y$ and by asking the configuration $(A,\Phi)$ to be $\R_t$-invariant. A \spinc structure $\s$ on $Y$ is again a pair $(S,\rho_3)$ where the spin bundle $S=S^+$ has complex rank 2 and $\rho_3: T^*Y\to \End(S)$ defines the Clifford multiplication. The 3-dimensional Seiberg-Witten equations now read: 
\begin{equation}\label{3DDSWEQ}
\left\{
\begin{array}{rl}
\half \rho_3(F_{B^t})-(\Psi\Psi^*)_0&=\rho_3(\omega),\\
D_B\Psi&=0.
\end{array}
\right.
\end{equation}
where $B$ is a \spinc connection and $\Psi\in \Gamma(Y,S)$. Here $\omega\in \Omega^2(Y, i\R)$ is a closed 2-form and $D_B:\Gamma(Y,S)\to \Gamma(Y,S)$ denotes the Dirac operator on $Y$. 
\subsection{Results from the Second Paper}\label{Subsec3.2} In this subsection, we review the construction of the monopole Floer homology from \cite{Wang20}, which defines the functor in Theorem \ref{T2.5} \ref{G1}. For any $T$-surface $\TSigma$, define $-\TSigma\colonequals (-\Sigma, g_\Sigma, -\lambda,\mu)$ to be the orientation reversal of $\TSigma$. Since the category $\AT(\TSigma_1, \TSigma_2)$ is more or less equivalent to $\AT(\emptyset, (-\TSigma_1)\cup \TSigma_2)$ (only the property \ref{P1.2} may be different), we focus on the case when $\TSigma_1=\emptyset$. 

Given any 1-cell $\y=(Y,\psi, g_Y,\omega, \{\q\})\in \AT(\emptyset, \TSigma)$, we first attach a cylindrical end $[1,\infty)_s\times\Sigma$ to $Y$ to obtain a complete Riemannian manifold:
\[
\hy=Y\cup_{\psi} [1,\infty)_s\times\Sigma.
\]
The metric on the end is given by $d^2s+g_\Sigma$. The closed 2-form $\omega$ is extended constantly as $\mu+ds\wedge\lambda$ on $[1,\infty)_s\times\Sigma$, denoted also by $\omega\in \Omega^2(\hy, i\R)$. 
 Let $\s_{std}=(S_{std},\rho_{std,3})$ be the standard \spinc structure on $\R_s\times\Sigma$ with $c_1(\s_{std})=0\in H^2(\Sigma, i\R)$. The spin bundle $S_{std}$ can be constructed explicitly as 
\[
S_{std}=\C\oplus \Lambda^{0,1} \Sigma.
\]
See \cite[Section 2]{Wang20} for more details. A relative \spinc structure $\bs$ on $Y$ is a pair $(\s, \varphi)$ where $\s=(S,\rho_3)$ is a \spinc structure on $Y$ and 
\[
\varphi: (S,\rho_3)|_{\partial Y}\to \psi^*\s_{std}|_{\partial Y}
\]
is an isomorphism of \spinc structures near the boundary, compatible with $\psi:\partial Y\to \Sigma$. The set of isomorphism classes of relative \spinc structures on $Y$ 
\[
\Spincr(Y)
\]
is a torsor over $H^2(Y,\partial Y; \Z)$. There is a natural forgetful map from $\Spincr(Y)$ to the set of isomorphism classes of \spinc structures:
\[
\Spincr(Y) \to \Spinc(Y),\ \bs=(\s,\varphi)\mapsto \s,
\]
whose fiber is acted on freely and transitively by  $H^1(\Sigma, \Z)/\im (H^1(Y,\Z))$ reflecting the change of boundary trivializations. Any $\bs\in \Spincr(Y)$ extends to a relative \spinc structure on $\hy$, denoted also by $\bs$.

\medskip

 The key observation is that on $\R_s\times\Sigma$, the 3-dimensional Seiberg-Witten equations \eqref{3DDSWEQ} perturbed by $\omega=\mu+ds\wedge\lambda$ have a canonical $\R_s$-invariant solution, denoted by
 \[
 (B_*,\Psi_*),
 \]
  which is also the unique finite energy solution on $\R_s\times\Sigma$, up to gauge, by our assumptions on $(\lambda,\mu)$. This result is due to Taubes \cite[Proposition 4.4\ \&\ 4.7]{Taubes01}. See \cite[Theorem 2.6]{Wang20} for the precise statement that we use here. 
  
  When it comes to $\hy$, each configuration is required to approximate this special solution $(B_*,\Psi_*)$ as $s\to\infty$. Take $(B_0,\Psi_0)$ to be a smooth configuration on $\hy$ which agrees with $(B_*,\Phi_*)$ on the cylindrical end $[1,\infty)_s\times \Sigma$ and consider the configuration space for any $k> \half$:
\begin{align*}
\SC_k(\hy,\bs)=\{(B,\Psi): (b,\psi)=(B,\Psi)-(B_0,\Psi_0)\in L^2_k (\hy, iT^*\hy\oplus S)
\}.
\end{align*}
which is acted on freely by the gauge group
\begin{align*}
\CG_{k+1}(\hy)=\{u: \hy\to S^1\subset \C: u-1\in L^2_{k+1} (\hy, \C)\},
\end{align*}
using the formula:
\[
u(B,\Psi)=(B-u^{-1}du, u\Psi). 
\]

The perturbed Chern-Simons-Dirac functional on $\SC_k(\hy, \bs)$ is then defined as 
	\begin{equation}\label{E3.3}
	\CL_\omega (B,\Psi)=-\frac{1}{8}\int_{\hy} (B^t-B_0^t)\wedge (F_{B^t}+F_{B_0^t})+\half \int_{\hy}\langle D_B\Psi, \Psi\rangle+\half \int_{\hy}(B^t-B_0^t)\wedge \omega. \qedhere
	\end{equation}

For any 1-cell $\y\in \AT(\emptyset, \AT)$ and any relative \spinc structure $\bs\in \Spincr(Y)$, the monopole Floer homology $\HM_*(\y,\bs)$ is defined as the Morse homology of $\CL_\omega$ on the quotient configuration space $\SC_k(\hy,\bs)/\CG_{k+1}(\hy)$. However, it is not guaranteed that $\CL_\omega$ descends to a Morse function on $\SC_k(\hy,\bs)/\CG_{k+1}(\hy)$, so an admissible perturbation $\q$ of $\CL_\omega$, which is encoded already in \ref{P8}, is needed to make critical points Morse and moduli spaces of flowlines regular. The main result of \cite{Wang20} says that the monopole Floer homology 
\[
\HM_*(\y,\bs)
\]
is well-defined as a finitely generated module over the mod $2$ Novikov field $\NR$. Since we have assumed in \ref{T3} that for any $T$-surface $\TSigma$, both $\lambda$ and $\mu$ are non-vanishing on any component $\Sigma^{(i)}\subset \Sigma$, we have a stronger statement:

\begin{theorem}[{\cite[Theorem 1.4]{Wang20}}] For any 1-cell $\y\in \AT(\TSigma_1,\TSigma_2)$, the direct sum
	\[
	\HM_*(\y)\colonequals \bigoplus_{\bs\in \Spincr(Y)}\HM_*(\y,\bs),
	\]
	is finitely generated over $\NR$. In particular, the group $\HM_*(\y,\bs)$ is non-trivial for only finitely many relative \spinc structures. 
\end{theorem}

As explained in \cite[Section 17]{Wang20}, this Floer homology is further enhanced into a functor:
\begin{theorem}[{\cite[Theorem 1.6 \& Remark 1.7]{Wang20}}]\label{T3.2} For any $T$-surfaces $\TSigma_1,\TSigma_2$, there is a functor from $\AT(\TSigma_1,\TSigma_2)$ to the category of finitely generated $\NR$-modules $ \AR(\star)$:
\[
\HM_*: \AT(\TSigma_1,\TSigma_2)\to \AR(\star),
\]
which assigns to each 1-cell $\y$ to its monopole Floer homology group $\HM_*(\y)$. 
\end{theorem}

\begin{remark} In the second paper \cite{Wang20}, we focused on connected 3-manifolds with toroidal boundary, but the results generalize to the disconnected case with no difficulty.
\end{remark}

For any 2-cell morphism $\x:\y_1\to \y_2$, the cobordism map 
\[
\HM_*(\x): \HM_*(\y_1)\to \HM_*(\y_2)
\]
is constructed as follows. We focus on the case when $\TSigma_1=\emptyset$. For the underlying strict cobordism $X: Y_1\to Y_2$, pick a diffeomorphism $\psi_X: W_X\to [-1,1]_t\times \Sigma$ and a closed 2-form $\omega_X\in \Omega^2(X, i\R)$ belonging to the class $[\psi_X]$ and $[\omega_X]_{cpt}$ respectively, as in \ref{Q3}\ref{Q6}. Choose a metric $g_X$ of $X$ compatible with its corner structures. We attach an end in the spatial direction to obtain a cobordism from $\hy_1$ to $\hy_2$:
\[
\hx\colonequals X\ \bigcup_{\psi_X}\ [-1,1]_t\times [0,\infty)_s\times\Sigma, 
\]
and attach cylindrical ends in the time direction to form a complete Riemannian manifold:
\[
\CX\colonequals (-\infty, -1]_t\times \hy_1\ \bigcup\ \hx\ \bigcup\ [1,+\infty)_t\times \hy_2. 
\]
The closed 2-form $\omega_X$ extends to a 2-form on $\CX$ by setting 
\[
\omega_X=\left\{\begin{array}{ll}
\omega_1 & \text{ on } (-\infty, -1]_t\times \hy_1,\\
\omega_2 & \text{ on }[1,+\infty)_t\times \hy_2,\\
\mu+ds\wedge\lambda& \text{ on } \R_t\times [0,\infty)_s\times \Sigma. 
\end{array}
\right.
\]

The cobordism map $\HM_*(\x)$ is then defined by counting 0-dimensional solutions (modulo gauge) to the Seiberg-Witten equations \eqref{SWEQ} with $\omega_X$ defined above. Some additional perturbations are required here to make moduli spaces regular; see \cite[Section 16]{Wang20} for more details. The cobordism map $\HM_*(\x)$ depends only on the isotopy class of $\psi_X$ and the relative cohomology class of $\omega_X$, and is independent of the planar metric $g_X$ of $X$. 

\subsection{The Canonical Grading} By \cite[Lemma 28.1.1]{Bible}, the standard spinor $\Psi_*$ on $\R_s\times\Sigma$ determines a canonical oriented 2-plane field $\xi_*$ that is $\R_s$-invariant.

 For any 3-manifold $Y$ with $\partial Y\cong \Sigma$, an oriented 2-plane field $\xi$ is called relative if $\xi=\xi_*$ near the boundary. Any homotopy of oriented relative 2-plane fields is supposed to preserve $\xi_*$ near the boundary $\partial Y$. 

Inspired by the construction in \cite[Section 28]{Bible}, the author introduced a canonical grading on the group $\HM_*(\y,\bs)$ in \cite[Section 18]{Wang20}. The grading set 
\[
\Xi^{\pi}(\y,\bs)
\]
is identified with the homotopy classes of oriented relative 2-plane fields that belongs to the relative \spinc structure $\bs$.

\subsection{Euler Characteristics} The 3-manifold $Y$ is homology oriented, if  we pick an orientation of $\bigoplus_{i=0}^3 H_i(Y;\R)$. Any homology orientation of $Y$ induces a canonical mod 2 grading on $\HM_*(\y,\bs)$ (cf. \cite{MT96}\cite[Subsection 18.2]{Wang20}). Then the graded Euler Characteristics of $\HM_*(\y,\bs)$ is well-defined and recovers a classical algebraic invariant for 3-manifolds with toroidal boundary. For future reference, we record the statement below:
\begin{theorem}[\cite{MT96,T98,Taubes01}]\label{T3.4} The graded Euler Characteristic of $\HM_*(\y,\bs)$:
	\begin{align*}
\SW(Y): \Spincr(Y)&\to \Z\\
	\bs &\mapsto  \chi(\HM_*(\y,\bs)),
	\end{align*}
	is independent of the auxiliary data $(g_Y,g_\Sigma; \omega,\lambda,\mu)$ and is equal to the Minor-Turaev torsion $T(Y)$ up to an overall sign ambiguity. Moreover, the function $\SW(Y)$ is invariant under the conjugacy symmetry: $\bs\leftrightarrow \bs^*$.
\end{theorem}

\subsection{Identity 1-cells}\label{Subsec3.2.5} Before we proceed to the construction of the gluing map, let us first define the canonical isomorphism 
\[
\iota_{\TSigma}: \HM_*(e_{\TSigma})\to \NR
\]
in Theorem \ref{T2.5}\ \ref{G4}. By the definition of the identity 1-cell $e_{\TSigma}$, $\HM_*(e_{\TSigma})$ is computed using the product manifold $\R_s\times \Sigma$ with $\omega=\mu+ds\wedge\lambda$. As noted earlier, the 3-dimensional Seiberg-Witten equations \eqref{3DDSWEQ} have a unique finite energy solution $(B_*,\Psi_*)$ up to gauge for the standard relative \spinc structure $\bs_{std}$. Moreover, $(B_*,\Psi_*)$ is irreducible and non-degenerate as the critical point of $\CL_\omega$ in the quotient configuration space; see the proof of  \cite[Proposition 12.1]{Wang20}. By \cite[Theorem 2.4]{Wang20}, any downward gradient flowline of $\CL_\omega$ connecting $(B_*,\Psi_*)$ to itself is necessary a constant path. As a result, the monopole Floer chain complex of $(e_\TSigma,\bs_{std})$ is generated by this special solution with trivial differential. The canonical isomorphism $\iota_{\TSigma}$ is then defined by sending this generator to $1\in \NR$. When $\bs\neq \bs_{std}$, the 3-dimensional equations \eqref{3DDSWEQ} have no solutions at all, so $\HM_*(e_{\TSigma},\bs)=\{0\}$. 
\subsection{The Gluing Map}\label{Subsec3.3} Having defined the functor $\HM_*$ in Theorem \ref{T2.5} \ref{G1} and the canonical isomorphism $\iota_{\TSigma}$ in \ref{G4}, let us construct the gluing map $\alpha$ in \ref{G2} in this subsection. The idea is borrowed from the proof of Floer's excision theorem \cite{BD95} and \cite[Theorem 3.2]{KS}. We focus on the special case when $\TSigma_1=\TSigma_3=\emptyset$ and construct the map
\[
\alpha:  \HM_*(\y_1)\otimes_\NR\HM_*(\y_2)\to \HM_*(\y_1\circ_h \y_2),
\]
for any 1-cells $\y_1\in \AT(\emptyset, \TSigma)$ and $\y_2\in \AT(\TSigma, \emptyset)$. The general case is not really different. 

Let $Y_i$ be the underlying 3-manifold of $\y_i$ and $Y_1\circ_h Y_2$ be the closed 3-manifold obtained by gluing $Y_1$ and $Y_2$ along $\{0\}\times \Sigma$. In what follows, we also work with the truncated 3-manifolds
\[
\begin{array}{lcl}
Y_{1,-}\colonequals& \{s\leq -1\}&\subset \hy_1,\\
Y_{2,-}\colonequals&\{s\geq 1\}&\subset \hy_2.
\end{array}
\]

The gluing map $\alpha$ is induced from an explicit strict cobordism between
\[
X: Y_1\ \coprod\ Y_2\to [-1,1]_s\times\Sigma\ \coprod\ Y_1\circ_h Y_2,
\]
as we describe now. Let $\Omega$ be an octagon with a prescribed metric such that the boundary of $\Omega$ consists of geodesic segments of length 2, and the internal angles are always $\pi/2$. Moreover, this metric is hyperbolic somewhere in the interior and flat near the boundary. 
\medskip
\begin{figure}[H]
	\centering
	\begin{overpic}[scale=.12]{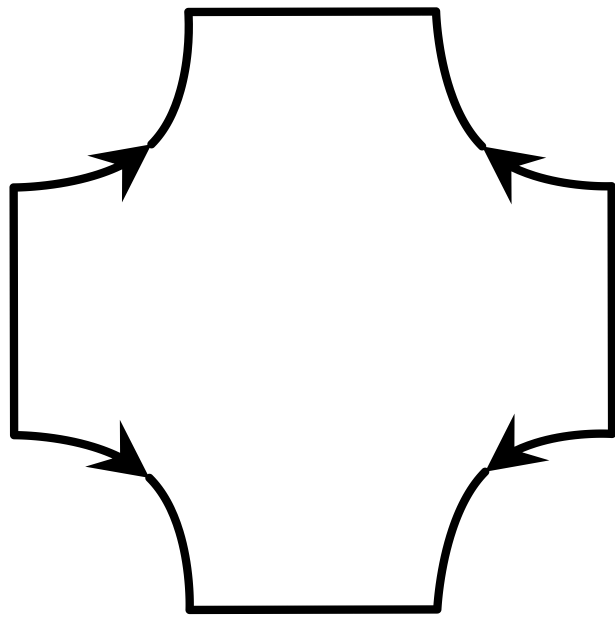}
		\put(5,10){\small$\gamma_1$}
		\put(85,10){\small$\gamma_2$}
		\put(85,85){\small$\gamma_3$}
				\put(5,85){\small$\gamma_4$}
				\put(47,47){\small$\Omega$}
					\put(-20,25){\small$-1$}
					\put(107,25){\small$1$}
					\put(107,65){\small$-1$}
					\put(-35,65){\small$h=1$}
					\put(20,-12){\small$-1$}
					\put(67,-12){\small 1}
				\put(27,103){\small 1}
	\put(63,103){\small $-1$}
	\end{overpic}	
\medskip
\caption{The surface $\Omega$ with corners}\label{Pic1}
\end{figure}
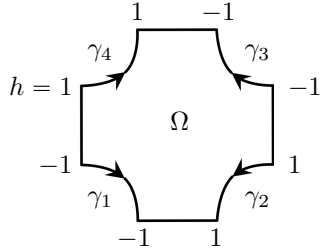

The product $\Omega\times\Sigma$ is now a 4-manifold with corners. The desired strict cobordism $X$ is obtained then by attaching $[-1,1]_t\times Y_{1,-}$ to $\gamma_1\times \Sigma$ and $[-1,1]_t\times Y_{2,-}$ to $\gamma_2\times \Sigma$. Arrows in Figure \ref{Pic1} indicate the positive direction of the time coordinate $t$. 

\smallskip

To define the closed 2-form $\omega_X$, let $h: \Omega\to \R$ be a function that equals to $1$ on $\gamma_2\cup \gamma_4$ and to $-1$ on $\gamma_1\cup \gamma_3$. Also, $h$ is required to be the linear function on the other four boundary segments of $\Omega$. Set
\begin{equation}\label{E3.1}
\omega_X=\mu+dh\wedge \lambda
\end{equation}
on $\Omega\times\Sigma$. The 4-manifold cobordism
\[
\hx:\hy_1\ \coprod\ \hy_2\to (\R_s\times\Sigma)\ \coprod\ (Y_1\circ_h Y_2)
\]
can be now schematically shown as follows: 
\begin{figure}[H]
	\centering
	\begin{overpic}[scale=.10]{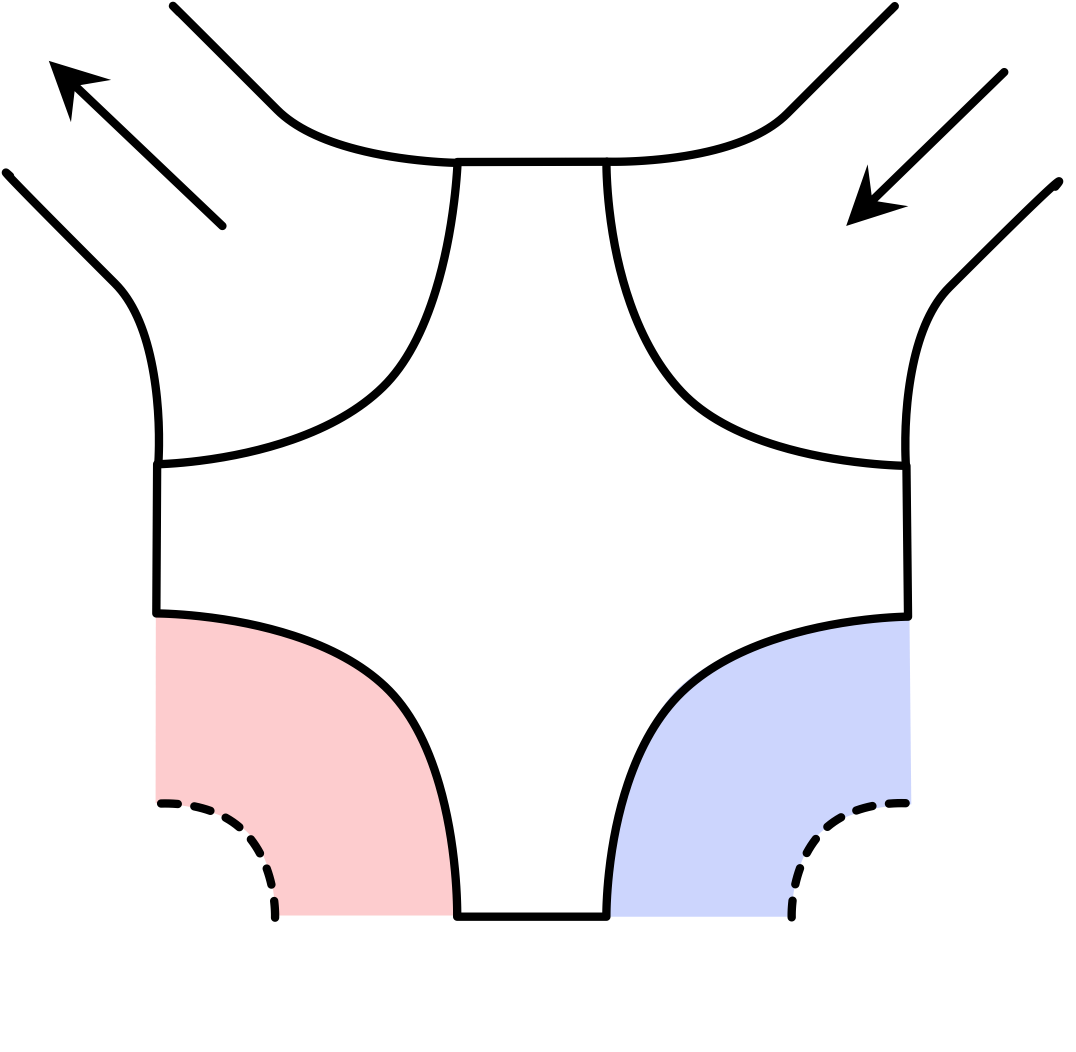}
	\put(35,0){$Y_1\circ _h Y_2$}
	\put(-20,40){$\hy_1\Rightarrow$}
	\put(95,40){$\Leftarrow \hy_2$}
		\put(35,90){$\R_s\times\Sigma$}
			\put(-70,10){$[-1,1]_t\times Y_1\rightsquigarrow$}
						\put(95,10){$\leftsquigarrow [-1,1]_t\times Y_2$}
							\put(105,80){$\leftsquigarrow [-1,1]_t\times (-\infty,-1]_s\times \Sigma$}
								\put(-120,80){$[-1,1]_t\times [1,\infty)_s\times\Sigma\rightsquigarrow$}
	\end{overpic}	
	\caption{}
\label{Pic2}
\end{figure}
The arrows $``\to"$ in Figure \ref{Pic2} indicate the positive direction of the spatial coordinate $s$. As a result, the complete Riemannian manifold $\CX$ obtained by attaching cylindrical ends to $\hx$ has two planar ends: $\R_t\times [1,+\infty)_s\times\Sigma$ and $\R_t\times (-\infty,-1]_s\times\Sigma$.

\begin{remark} In order to make $\omega_X$ into a smooth form on $\CX$, the function $h:\Omega\times\Sigma\to \R$ must extend to a smooth function on $\CX$ such that $h\equiv s$ on any planar end. One may think of $h$ as an extension of the spatial coordinate $s$ over the region $\Omega\times\Sigma$.
\end{remark}

On the other hand, one can draw the cobordism $\hx$ vertically and regard $\Omega$ as the part of the pair-of-pants cobordism that contains the saddle point:
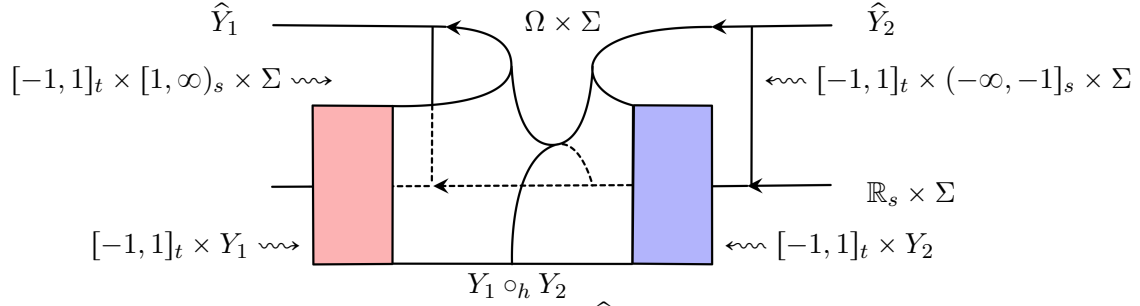
\begin{figure}[H]
	\centering
	\begin{overpic}[scale=.10]{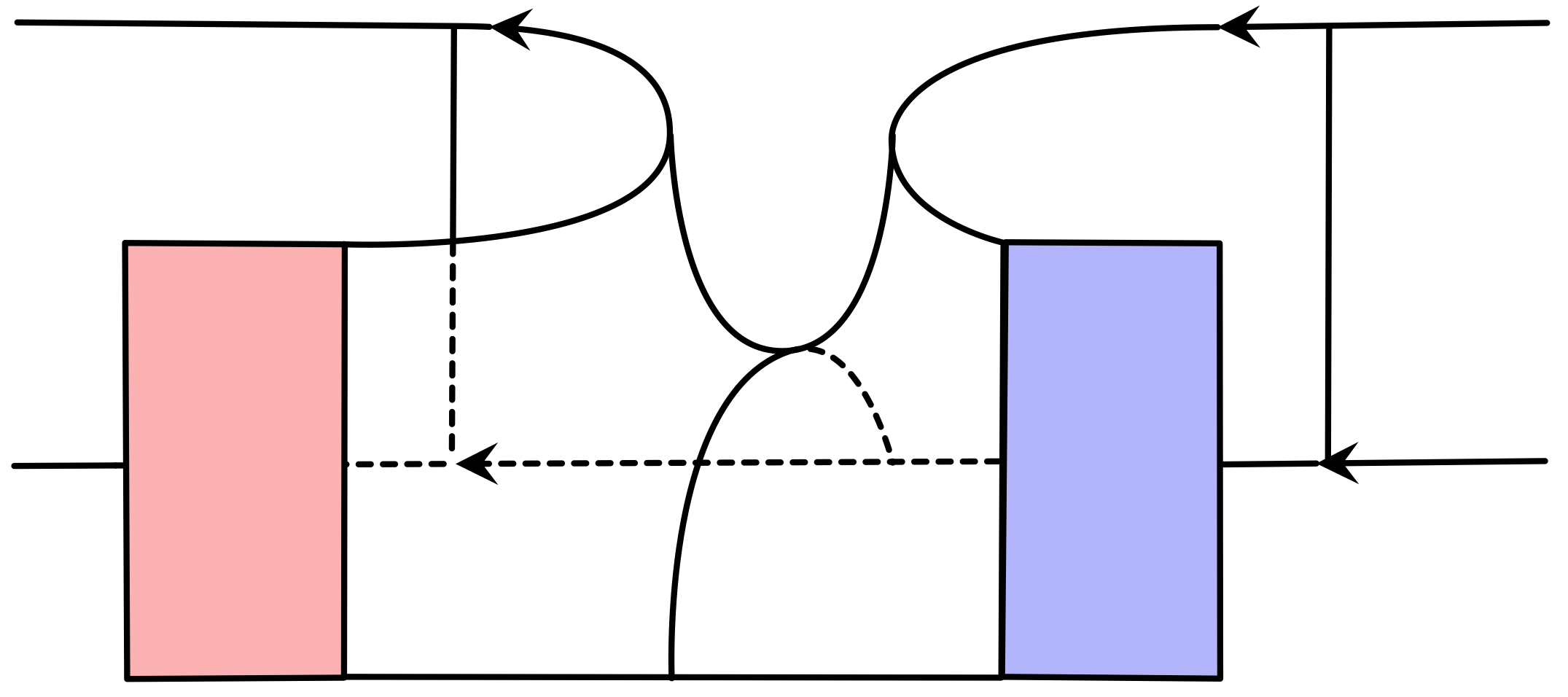}
		\put(35,-4){$Y_1\circ _h Y_2$}
		\put(-10,42){$\hy_1$}
		\put(105,42){$\hy_2$}
		\put(105,12){$\R_s\times\Sigma$}
			\put(45,42){$\Omega \times\Sigma$}
		\put(-31,3){$[-1,1]_t\times Y_1\rightsquigarrow$}
		\put(80,3){$\leftsquigarrow [-1,1]_t\times Y_2$}
		\put(87,32){$\leftsquigarrow [-1,1]_t\times (-\infty,-1]_s\times \Sigma$}
		\put(-45,32){$[-1,1]_t\times [1,\infty)_s\times\Sigma\rightsquigarrow$}
	\end{overpic}	
	\caption{Draw $\hx$ vertically.}
	\label{Pic3} 
\end{figure}

Since the base ring $\NR$ is also a field, we use the K\"{u}nneth formula to identify
\[
\HM_*(\y_1\coprod \y_2)\cong\HM_*(\y_1)\otimes_\NR\HM_*(\y_2). 
\]

The gluing map $\alpha$ is then defined as the cobordism map induced from $\x=(X,\omega_X)$ and multiplied by a normalizing constant $\eta\in \NR$:
\[
\begin{array}{rcc}
\alpha: \HM_*(\y_1)\otimes_\NR \HM_*(\y_2)&\xrightarrow{\HM_*(\x)}& \HM_*(e_{\TSigma})\otimes_\NR \HM_*(\y_1\circ_h \y_2)\\
&\xrightarrow{\iota_{\TSigma}\otimes \Id}&\NR\otimes_\NR \HM_*(\y_1\circ_h \y_2)\\
&\xrightarrow{\eta\times}&\HM_*(\y_1\circ_h \y_2),
\end{array}
\]
where 
\[
\eta=\prod_{i=1}^n\frac{1}{t_i-t_i^{-1}}\in \NR\text{ with }t_i=q^{|2\pi i\langle \mu, [\Sigma^{(i)}]\rangle|},
\]
and $\Sigma^{(i)}$ is the $i$-th component of $\Sigma=\coprod_{i=1}^n\Sigma^{(i)}$. The inverse $\beta$ of $\alpha$ is induced from the opposite cobordism of $\x$, denoted by $\x'=(X',\omega_X')$, and is normalized by the same constant $\eta$:
\[
\beta: \HM_*(\y_1\circ_h \y_2)\xrightarrow{\eta\ \otimes\ -\ } \HM_*(e_{\TSigma})\otimes_\NR \HM_*(\y_1\circ_h \y_2)\xrightarrow{\HM_*(\x')}\HM_*(\y_1)\otimes_\NR \HM_*( \y_2)
\]
\begin{figure}[H]
		\medskip
	\centering
	\begin{overpic}[scale=.10]{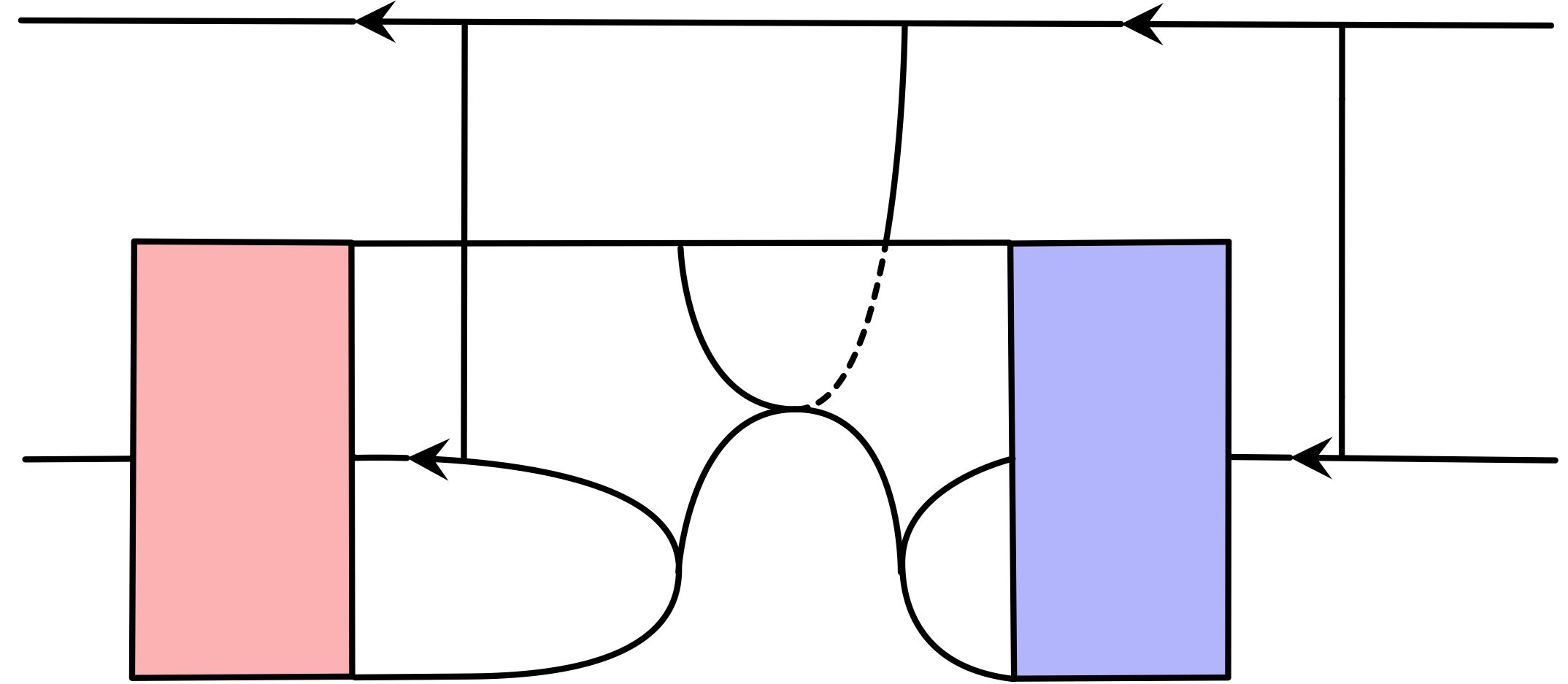}
		\put(35,31){$Y_1\circ _h Y_2$}
		\put(0,0){$\hy_1$}
		\put(82,0){$\hy_2$}
		\put(50,45){$\R_s\times\Sigma$}
	\end{overpic}	
	\caption{The opposite cobordism $\hx'$.}
	\label{Pic4} 
\end{figure}
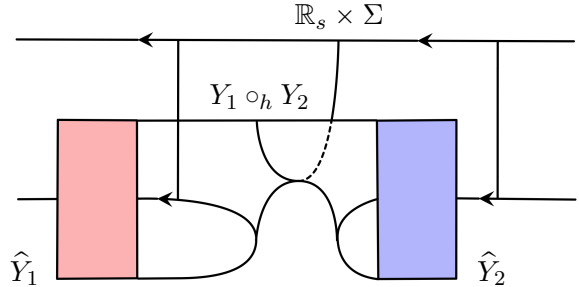

The choice of the normalizing constant $\eta$ is justified by the following theorem, which says that the gluing map $\alpha$ is indeed an isomorphism with inverse $\beta$. 

\begin{theorem}[Floer's Excision Theorem]\label{T3.3} The gluing map $\alpha$ and $\beta$ constructed above are mutual inverses to each other, i.e.,
	\[
	\alpha\circ \beta=\Id_{\HM_*(\y_1\circ_h \y_2)},\ \beta\circ \alpha=\Id_{\HM_*(\y_1)\otimes\HM_*(\y_2)}.
	\]
\end{theorem}

The gluing map $\alpha$ preserves the canonical grading on $\HM_*(\y_i)$. There are natural concatenation maps:
\begin{align*}
-\ \circ_h\ -\ &:\Spincr(Y_1)\times \Spincr(Y_2)\to \Spincr(Y_1\circ_h Y_2),\\
-\ \circ_h\ -\ &:\Xi^{\pi}(\y_1,\bs_1)\times\Xi^{\pi}(\y_2,\bs_2)\to \Xi^{\pi}(\y_1\circ_h \y_2, \bs_1\circ_h\bs_2). 
\end{align*}
Indeed, any two relative \spinc structures $\bs_i=(\s_i,\varphi_i), i=1,2$ can be composed using the map
\[
(S_1,\rho_{3,1})|_{\partial Y_1} \xrightarrow{\varphi_1} \s_{std}|_\Sigma\xrightarrow{\varphi_2^{-1}} (S_2,\rho_{3,2})|_{\partial Y_2},
\]
to produce a \spinc structure on $Y_1\circ_h Y_2$. Meanwhile, any oriented relative 2-plane fields $(\xi_1,\xi_2)$ can be composed, since they agree with the canonical 2-plane field $\xi_*$ near $\Sigma$.

\smallskip

 In the special case that we have considered so far, $Y_1\circ_h Y_2$ is a closed 3-manifold, so $\Spincr(Y_1\circ_h Y_2)=\Spinc(Y_1\circ_h Y_2)$. $\Xi^{\pi}(\y_1\circ_h \y_2, \s)$ is the subset of $\pi_0(\Xi(Y_1\circ_h Y_2))$, the homotopy classes of oriented 2-plane fields on $Y_1\circ_h Y_2$, that belongs to $\s$; see \cite[P. 585]{Bible} for the precise definition of $\pi_0(\Xi(Y_1\circ_h Y_2))$.

\begin{theorem}\label{T3.5} The gluing map $\alpha:  \HM_*(\y_1)\otimes_\NR\HM_*(\y_2)\to \HM_*(\y_1\circ_h \y_2)$ preserves the relative \spinc grading and the canonical grading by the homotopy classes of oriented relative 2-plane fields, meaning that $\alpha$ restricts to a map 
	\[
	\alpha: \bigoplus_{ \bs_1\circ_h \bs_2=\s}\HM_*(\y_1,\bs_1)\otimes \HM_*(\y_2,\bs_2) \to \HM_*(\y_1\circ_h \y_2,\bs),
	\]
which is an isomorphism by Theorem \ref{T3.3}. Moreover, if an element $(x_1, x_2) $ belongs to the grading $([\xi_1],[\xi_2])\in \Xi^{\pi}(\y_1,\bs_1)\times\Xi^{\pi}(\y_2,\bs_2)$, then $\alpha(x_1\otimes x_2)$ is in the grading $[\xi_1]\circ_h[\xi_2]$. 
\end{theorem}

The rest of Subsection \ref{Subsec3.3} is devoted to the proof of Theorem \ref{T3.3} and \ref{T3.5}.
\begin{proof}[Proof of Theorem \ref{T3.3}] The argument in \cite[Theorem 3.2]{KS} carries over to our case with little change. We focus on the second identity $\beta\circ \alpha=\Id_{\HM_*(\y_1)\otimes\HM_*(\y_2)}$ to explain the choice of the normalizing constant $\eta$. The map $\beta\circ \alpha$ is identical to the one induced by $\x'\circ _v \x$:
	\begin{figure}[H]
		\centering
		\begin{overpic}[scale=.08]{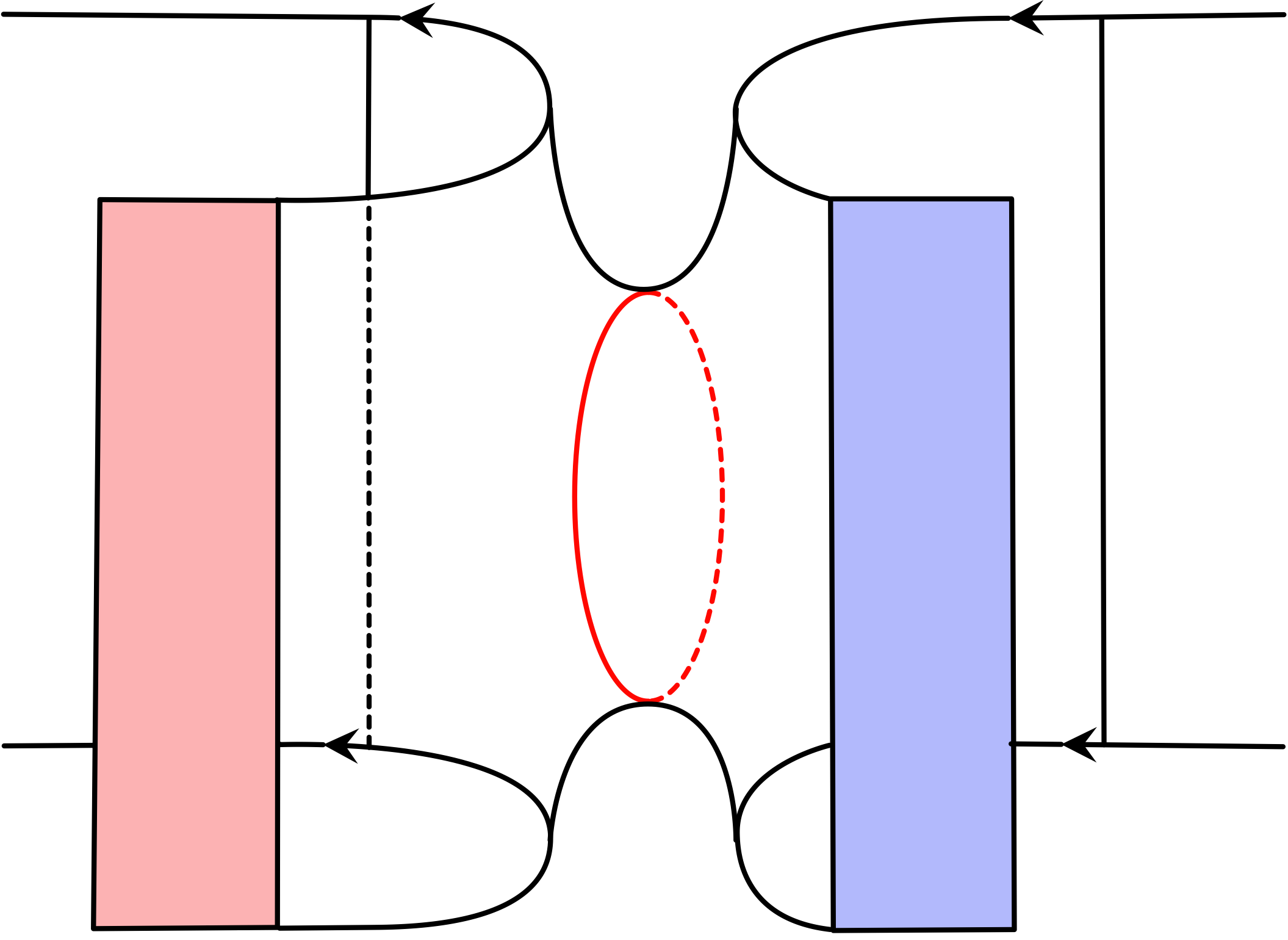}
			\put(-5,0){$\hy_1$}
			\put(85,0){$\hy_2$}
			\put(30,45){\small $k\times \Sigma$}
		\end{overpic}	
		\caption{}
		\label{Pic5} 
	\end{figure}
To compare $\x'\circ_v \x$ with the product cobordism from $\y_1\coprod \y_2$ to itself, we stretch the neck along a union of 3-tori $k\times \Sigma$, where $k$ is the red circle in Figure \ref{Pic5}. To specify the closed 2-form $\omega_4$ in the Seiberg-Witten equations \eqref{SWEQ} as we vary the metrics, we regard $k$ as a curve in $\Omega'\circ_v \Omega$:
	\begin{figure}[H]
	\centering
	\begin{overpic}[scale=.12]{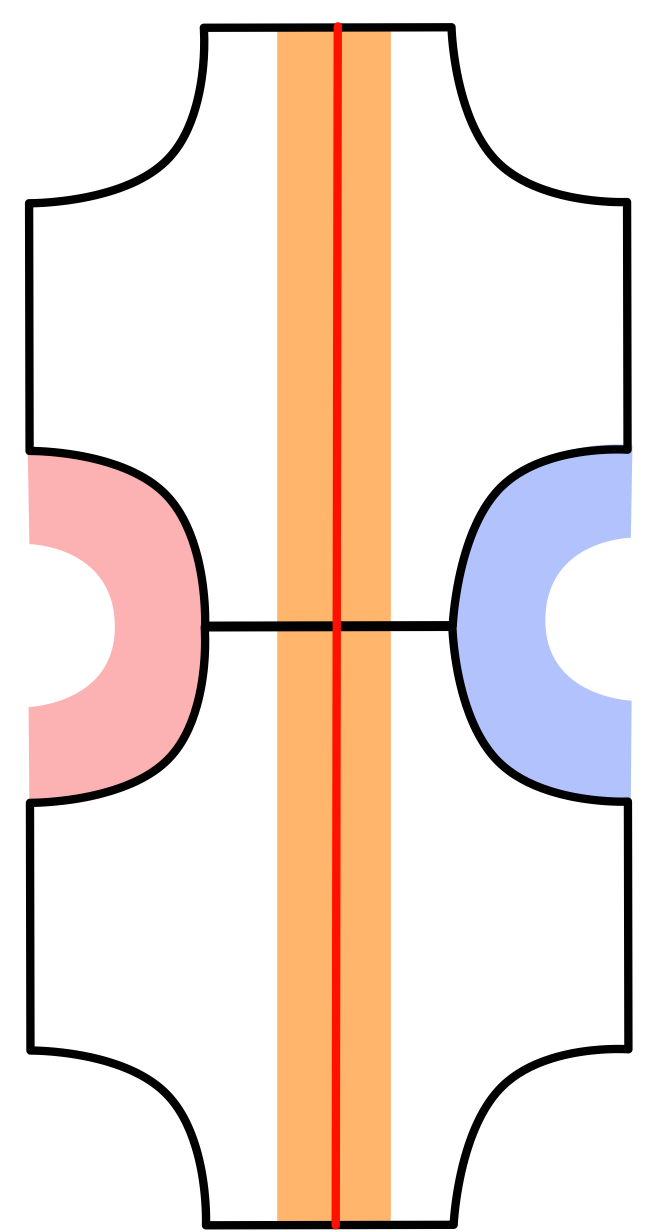}
		\put(-20,70){$\hy_1\Rightarrow$}
		\put(-20,20){$\hy_1\Leftarrow$}
		\put(25,101){$k$}
				\put(55,20){$\Rightarrow\hy_2$}
		\put(55,70){$\Leftarrow\hy_2$}
		\put(5,92){$1$}
			\put(40,92){$-1$}
					\put(5,3){$1$}
			\put(40,3){$-1$}
					\put(5,65){$-1$}
			\put(40,65){$1$}
					\put(10,20){$\Omega'$}
	\end{overpic}	
	\caption{}
	\label{Pic6} 
\end{figure}
Here $\Omega'$ is the opposite cobordism of $\Omega$, regarded as part of a pair-of-pants cobordism. In Figure \ref{Pic6}, the top horizontal edge is identified with the bottom edge. On $\Omega'\circ_v\Omega$, one may homotope the function $h: \Omega'\circ_v \Omega\to \R$ rel boundary such that $h\equiv 1/2$ on $[-1/2,1/2]_s\times k$, the tubular neighborhood of $k$ colored orange in Figure \ref{Pic6}. As we stretch the neck $[-1/2,1/2]_s\times k\times \Sigma$, the 2-form $\omega_4$ is set to be $\mu+dh\wedge\lambda$ on $(\Omega'\circ_v\Omega)\times\Sigma$. In particular, $\omega\equiv \mu$ on the neck $[-1/2,1/2]_s\times k\times \Sigma$. 

This 2-form $\omega_4$ is relatively cohomologous to the concatenation $\omega_X'\circ_v\omega_X$. Indeed, their difference is $d(f\wedge\lambda)$ for a compactly supported smooth function $f: \Omega'\circ_v \Omega\to \R$. Thus we can use $\omega_4$ to compute the cobordism map of $\x'\circ_v \x$. 

As the underlying 4-manifold of $\x'\circ_v \x$ is completely stretched along $k\times\Sigma$ in Figure \ref{Pic5}, we need the following result concerning the monopole Floer homology of the 3-torus $\T^3$:
\begin{lemma}\label{L3.4} Let $\T^2$ be the 2-torus and $\T^3=\T^2\times S^1$. Let $d\in H^2(\T^3,\Z)$ be the Poincar\'{e} dual of $\{pt\}\times S^1$ and set $[\omega]=i\delta \cdot d\in H^2(\T^3, i\R)$ for some $\delta\in \R$. Following the notations from \cite[Section 30]{Bible}, we write $\HM_*(\T^3, \s, c;\NR_\omega)$ for the monopole Floer homology of $\T^3$ associated to the period class $c=-2\pi i[\omega]=2\pi \delta\cdot d$ and the \spinc structure $\s$, which is defined using the Seiberg-Witten equations \eqref{3DDSWEQ} for some imaginary valued 2-form $\omega$ in the class $ [\omega]$. If in addition $\delta\neq 0$ and $|\delta|<2\pi$, then this group can be computed as follows:
\[
\HM_*(\T^3, \s, c;\NR_\omega)=\left\{\begin{array}{cl}
\NR & \text{ if } c_1(\s)=0,\\
\{0\} & \text{ otherwise.}
\end{array}
\right.
\]
\end{lemma}
\begin{proof}[Proof of Lemma \ref{L3.4}] Pick a flat metric of $\T^2$ and equip $\T^3$ with the product metric. Take $\omega$ to be a multiple of the volume form $dvol_{\T^2}$. In this case, the 3-dimensional Seiberg-Witten equations \eqref{3DDSWEQ} can be solved explicitly. If $|\delta|<2\pi$ and $c_1(\s)\neq 0$, \eqref{3DDSWEQ} has no solutions at all. If $\delta\neq 0$ and $c_1(\s)=0$, \eqref{3DDSWEQ} has a unique solution $\gamma_*$, which is irreducible; see \cite[Lemma 3.1]{Taubes01}. Since we have worked with a non-balanced perturbation, the perturbed Chern-Simons-Dirac functional $\CL_\omega$ is not full gauge invariant. By \cite[Proposition 4.4]{Taubes01}, the moduli space of down-ward gradient flowlines of $\CL_\omega$ connecting $\gamma_*$ to itself is not empty, although the formal dimension predicted by the index theory is always zero. 
	
This issue can be circumvented using an admissible perturbation $\q$ of $\CL_\omega$ supported away from $\gamma_*$, as explained in \cite[Section 15]{Bible}. Moreover, $\q$ can be made small so that $\gamma_*$ is still the unique critical point of the perturbed functional, giving rise to the unique generator of $\HM_*(\T^3,\s,c; \NR_\omega)$. 
\end{proof}

To complete the proof of Theorem \ref{T3.3}, we need another result regarding the monopole invariants of $M\colonequals D^2\times \T^2$, which we recall below. Let $\s_{std}$ be the standard \spinc structure on $\T^3$ with $c_1(\s_{std})=0$. A relative \spinc structure $\bs_M$ on $M$ is a pair $(\s_M,\varphi)$ where $\s_M$ is a \spinc structure and $\varphi:\s_M|_{\partial M}\to \s_{std}$ is a fixed isomorphism. In particular, one may define its relative Chern class $c_1(\bs_M)\in H^2(M, \partial M;\Z)\cong H^2(D^2,S^1;\Z)$. We shall work with a flat metric of $\T^2$ and make $D^2$ into a surface with a cylindrical end:
	\begin{figure}[H]
	\centering
	\begin{overpic}[scale=.12]{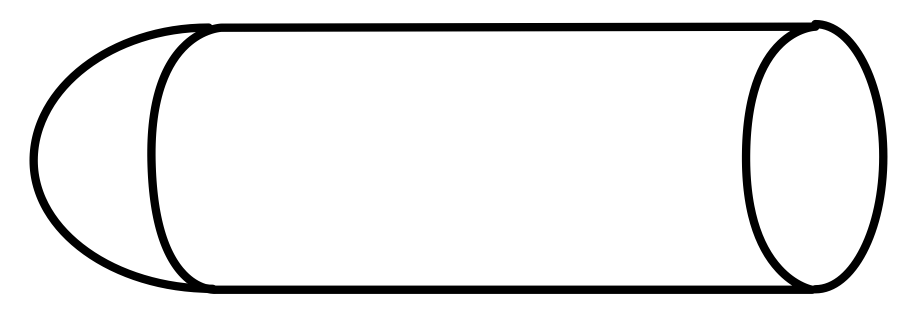}
		\put(110, 16){$\cdots$}
		\put(100,36){$s\to $}
	\end{overpic}	
	\caption{The disk with a cylindrical end.}
	\label{Pic7} 
\end{figure}

Let $d\in H^2(M, \Z)$ be the dual of $\{pt\}\times (D^2,S^1)$ and $[\omega_M]\colonequals i\delta \cdot d$ for some $\delta\in \R$. The monopole invariant of $M$ is defined as a generating function 
\[
\fm(M,[\omega_M])\colonequals \sum_{\bs_M} \fm(M, \bs_M,[\omega_M])\cdot q^{-\E_{top}^{\omega_M}(\bs_M)}\in\NR 
\]
with $\fm(M, \bs_M,[\omega_M])\in \BF_2$ and 
\begin{equation}\label{E3.5}
\E_{top}^{\omega_M}(\bs_M)\colonequals \frac{1}{4}\int_M (F_{A^t_0}-2\omega_M)\wedge(F_{A^t_0}-2\omega_M)=-2\pi \delta\cdot (c_1(\bs)\cup d)[M,\partial M].
\end{equation}
Here $A_0$ is a reference \spinc connection on $(M,\bs_M)$. The coefficient $\fm(M, \bs_M,[\omega_M])$ is defined by counting finite energy solutions to the Seiberg-Witten equations \eqref{SWEQ} with $\omega_M=i\delta\cdot dvol_{\T^2}/ \Vol(\T^2)$ for the relative \spinc manifold $(M,\bs_M)$. In practice, one has to perturb $\omega_M$ by a compactly supported closed 2-form (see \cite{Taubes01}) or add a tame perturbation to $\CL_\omega$ as in the proof of Lemma \ref{L3.4}, to ensure that the moduli space is transversely cut out. More invariantly, one should regard $\fm(M,[\omega_M])$ as an element in $\HM_*(\T^3,\s_{std}, c; \NR_\omega)$.
\begin{lemma}\label{L3.5} Using the canonical identification $\HM_*(\T^3,\s_{std},c;\NR_\omega)\cong \NR$ in the proof of Lemma \ref{L3.4}, the monopole invariant $\fm(M,[\omega_M])$ can be computed as 
	\[
	(t-t^{-1})^{-1}=t^{-1}+t^{-3}+t^{-5}+\cdots \in \NR \text{ with }t=q^{2\pi |\delta|}.
	\]
\end{lemma} 
\begin{proof}[Proof of Lemma \ref{L3.5}] Although we have used non-exact perturbations, this computation is not different from the one in \cite[Section 38.2]{Bible}, in which case exact perturbations and a non-trivial local coefficient system are used. This formula can be found in \cite[P.719]{Bible}. 
\end{proof}

Back to the proof of Theorem \ref{T3.3}. Once the neck is completely stretched along $k\times\Sigma$ in Figure \ref{Pic5}, we glue two copies of $D^2\times \Sigma$ to obtain the product cobordism from $\hy_1 \coprod \hy_2$ to itself, which induces the identity map on monopole Floer homology groups. Since our surface $\Sigma=\coprod_{i=1}^n\Sigma^{(i)}$ is disconnected, we have 
\[
\fm(D^2\times \Sigma, [\mu])=\prod_{i=1}^n \fm(D^2\times \Sigma^{(i)}, [\mu_i])=\prod_{i=1}^n (t_i-t_i^{-1})^{-1}=\eta, 
\]
with $\mu_i\colonequals \mu|_{\Sigma^{(i)}}$ and $t_i=q^{|2\pi i\langle \mu,[\Sigma^{(i)}]\rangle|}$. As a result,
\[
\beta\circ \alpha=\eta^2 \HM_*(\x'\circ_v \x)=\Id_{\HM_*(\y_1)\otimes \HM_*(\y_2)}.
\]
The computation for $\alpha\circ \beta$ is similar and is omitted here. This completes the proof of Theorem \ref{T3.3}. 
\end{proof}
\begin{proof}[Proof of Theorem \ref{T3.5}] Fix relative \spinc structures $\bs_i\in \Spincr(Y_i)$ for $i=1,2$. The statement about relative \spinc gradings is obvious. Indeed, the 4-manifold cobordism in Figure \ref{Pic2} can be upgraded into a relative \spinc cobordism:
	\[
	(\hx,\bs_X): (\hy_1,\bs_1)\ \coprod\ (\hy_2,\bs_2)\to (Y_1\circ_h Y_2,\bs)\ \coprod\ (\R_s\times\Sigma, \bs_{std})
	\]
	only if $\bs=\bs_1\circ_h\bs_2$. To deal with the canonical grading, we make use of an equivalent definition of $\Xi^{\pi}(\hy_i,\bs_i)$ following \cite[(18.2)]{Wang20}. Instead of oriented relative 2-plane fields, we investigate the space of unit-length relative spinors on $\hy$, denoted by $\Xi(\y,\bs)$. A spinor $\Psi\in \Gamma(\hy,S)$ is called relative, if $\Psi=\Psi_*/|\Psi_*|$ on the cylindrical end $[0,\infty)_s\times\Sigma$; see \cite[Definition 18.2]{Wang20}. Finally, we have 
	\[
	\Xi^{\pi}(\y,\bs)\colonequals \pi_0(\Xi(\y,\bs))/H^2(Y,\partial Y;\Z)
	\]
	where $H^2(Y,\partial Y;\Z)= \pi_0(\CG(\hy))$ is the component group of the gauge group $\CG(\hy)$.
	
	Let $\fa_i\in \SC(\hy_i,\bs_i)$ be a critical point of the perturbed Chern-Simons-Dirac functional on $\hy_i$ for $i=1,2$. In fact, before passing to the quotient configuration space, we can assign an element 
	\[
	\gr(\fa_i)\in \pi_0(\Xi(\hy,\bs))
	\]
	whose image in $\Xi^{\pi}(\y,\bs)$ is the chain level grading of $[\fa_i]$. Let $\Psi_i\in \Gamma(\hy_i,S_i)$ be a unit-length relative spinor representing $\gr(\fa_i)$. 
	
	Since the cobordism map $\HM_*(\x)$ is defined by counting 0-dimensional moduli spaces on the 4-manifold $\CX$, the chain level map exists between 
	\[
	(\fa_1,\fa_2) \text{ and } (\fa_3, \fa_*)
	\]
	only if an index condition holds, where $\fa_*=(B_*,\Psi_*)$ is the canonical solution on $\R_s\times\Sigma$ and $\fa_3$ is a critical point on $Y_1\circ_h Y_2$. The relative spinor representing the grading of $\fa_*$ is $ \Psi_*/|\Psi_*|$, by the Normalization Axiom \cite[Section 18]{Wang20}. Let $\Psi_3$ be the one for $\fa_3$.
	
	By the Index Axiom from \cite[Section 18]{Wang20}, this index condition can be stated in terms of the quadruple $(\Psi_1,\Psi_2; \Psi_3, \Psi_*/|\Psi_*|)$. For a fixed relative \spinc cobordism $\bs_X$, we construct a non-vanishing spinor $\Phi_X$ on $\CX\setminus X$ as follows:
	\begin{itemize}
\item $\Phi_X\equiv \Psi_i$ on $(-\infty,1]_t\times \hy_i$ for $i=1,2$;
\item $\Phi_X\equiv \Psi_3$ on $[1,\infty)_t\times (Y_1\circ_h Y_2)$;
\item $\Phi_X\equiv \Psi_*/|\Psi_*|$ on $[1,\infty)_t\times\R_s\times\Sigma$ and 
\item $\Phi_X\equiv \Psi_*/|\Psi_*|$ also on $\R_t\times [1,\infty)_s\times\Sigma$ and $\R_t\times (-\infty,-1]_s\times \Sigma$.
	\end{itemize}

The index condition is then equivalent to saying that the relative Euler class of $\Phi_X$ vanishes
\begin{equation}\label{E3.2}
e(S^+_X; \Phi_X)[X,\partial X]=0,
\end{equation}
which determines the class $[\Psi_3]\in \pi_0(\Xi(\y_1\circ_h \y_2,\bs_1\circ_h\bs_2))$ in terms of $[\Psi_1]$ and $[\Psi_2]$.

\medskip

There is an obvious \spinc cobordism $\bs_X$ on $X$ (see Figure \ref{Pic2}): we pick the product relative \spinc structures on $[-1,1]_t\times Y_{1,-}$ and $[-1,1]_t\times Y_{2,-}$ respectively, and choose any relative \spinc structure on $\Omega\times\Sigma$. In this case, the characteristic condition \eqref{E3.2} holds trivially, if we take $\Psi_3$ to be the concatenation of $\Psi_1$ and $\Psi_2$. 

In general, any other relative \spinc cobordism $\bs_X'$ differs from $\bs_X$ by taking the tensor product with a relative line bundle $L\in H^2(X,\partial X;\Z)$, an action that leaves the relative Euler number unaffected. Thus the image of $[\Psi_3]$ in 
\[
\pi_0(\Xi(\y_1\circ_h\y_2,\bs_1\circ_h\bs_2))/H^2(Y_1\circ_h Y_2;\Z)
\]
is independent of $\bs_X'$. This completes the proof of Theorem \ref{T3.5}.
\end{proof}

\subsection{Proof of \ref{G2}} The proof of Theorem \ref{T2.5} will dominate the rest of Section \ref{Sec3}. In this subsection, we focus on \ref{G2}. For any 1-cells $\y_{12}\in\AT(\TSigma_1,\TSigma_2)$ and $\y_{23}\in \AT(\TSigma_2,\TSigma_3)$, the isomorphism 
\[
\alpha:  \HM_*(\y_{12})\otimes_\NR\HM_*(\y_{23})\to \HM_*(\y_{12}\circ_h \y_{23}),
\]
is constructed in the same way as in the special case discussed in Subsection \ref{Subsec3.3}. It remains to verify that $\alpha$ is a functor. Let $\x_{12}: \y_{12}\to \y_{12}'$ and $\x_{23}:\y_{23}\to\y_{23}'$ be 2-cell morphisms in $\AT(\TSigma_1,\TSigma_2)$ and $\AT(\TSigma_2,\TSigma_3)$ respectively. We have to show that  
\[
\HM_*(\x_{12}\circ_h \x_{23})\circ \alpha=\alpha\circ (\HM_*(\x_{12})\otimes \HM_*(\x_{23}))
\]
as maps from  $\HM_*(\y_{12})\otimes_\NR\HM_*(\y_{23})$ to $ \HM_*(\y_{12}'\circ_h \y_{23}')$. Indeed, both of them agree with the map induced by the 4-manifold cobordism in Figure \ref{Pic8}. This completes the proof of \ref{G2}. \qed
	\begin{figure}[H]
	\centering
	\begin{overpic}[scale=.10]{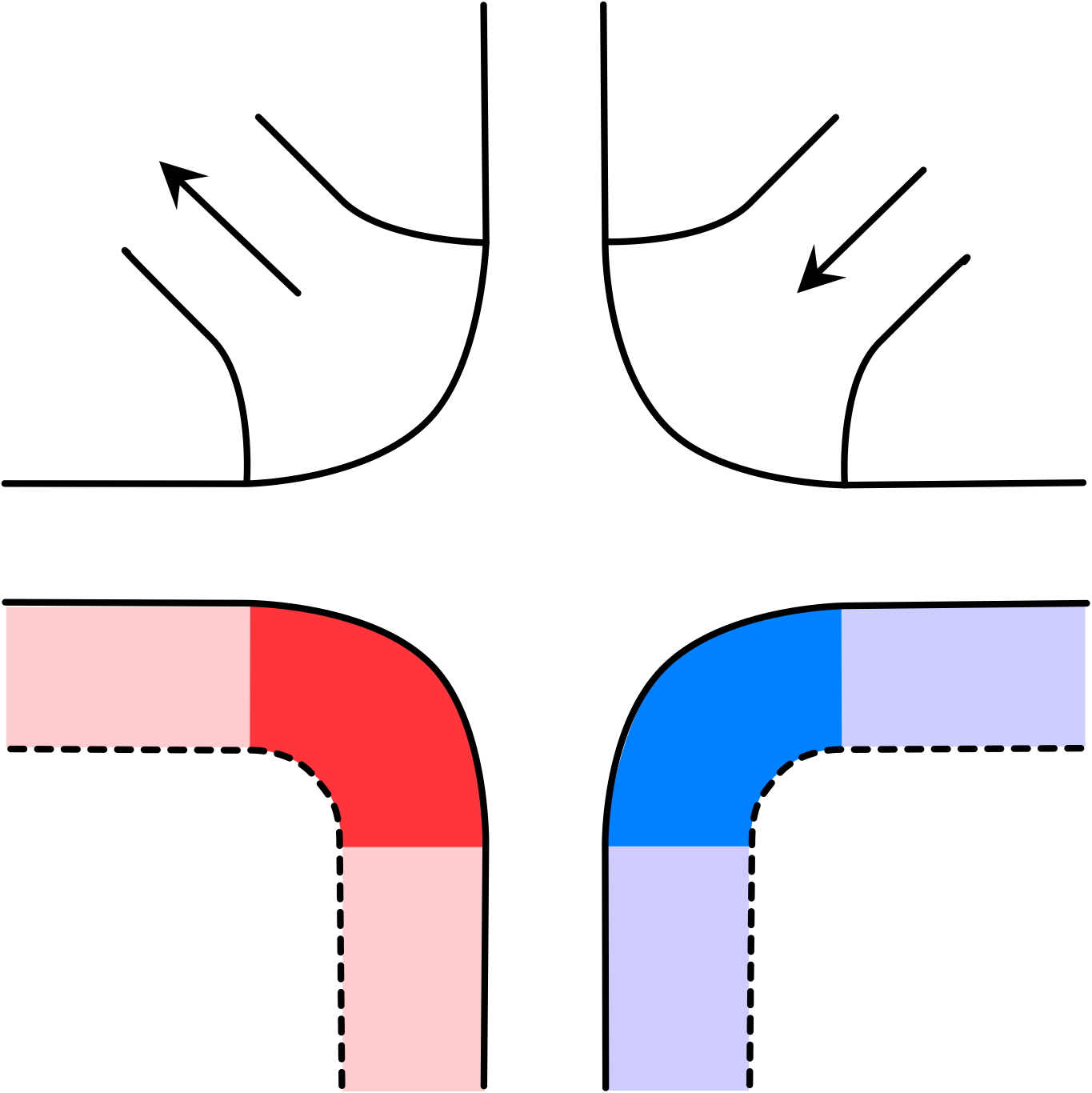}
		\put(-15,50){$\hy_{12}$}
		\put(105,50){$\hy_{23}$}
\put(29,31){\small $X_{12}$}
\put(59,31){\small $X_{23}$}
\put(-63,35){\small $(-\infty,-1]_t\times Y_{12}\to $}
\put(-25,10){\small $[1,\infty)_t\times Y_{12}'\to $}
\put(72,10){\small $\leftarrow [1,\infty)_t\times Y_{23}'$}
\put(102,35){\small $\leftarrow (-\infty, 1]_t\times Y_{23}$}
\put(36,-7){\small $Y_{12}'\circ_h Y_{23}'$}
	\end{overpic}	
\medskip
	\caption{}
	\label{Pic8} 
\end{figure}
\subsection{Proof of \ref{G3}} For any triple $(\y_{12},\y_{23},\y_{34})\in \AT(\TSigma_1,\TSigma_2)\times \AT(\TSigma_2,\TSigma_3)\times \AT(\TSigma_3,\TSigma_4)$, the composition $\alpha\circ (\alpha\otimes\Id)$ in the diagram \eqref{E2.1} is induced from the cobordism in Figure \ref{Pic9} below, with the red dot line indicating a copy of the completion $\widehat{Y_{12}\circ_h Y_{23}}$.
	\begin{figure}[H]
	\centering
	\begin{overpic}[scale=.06]{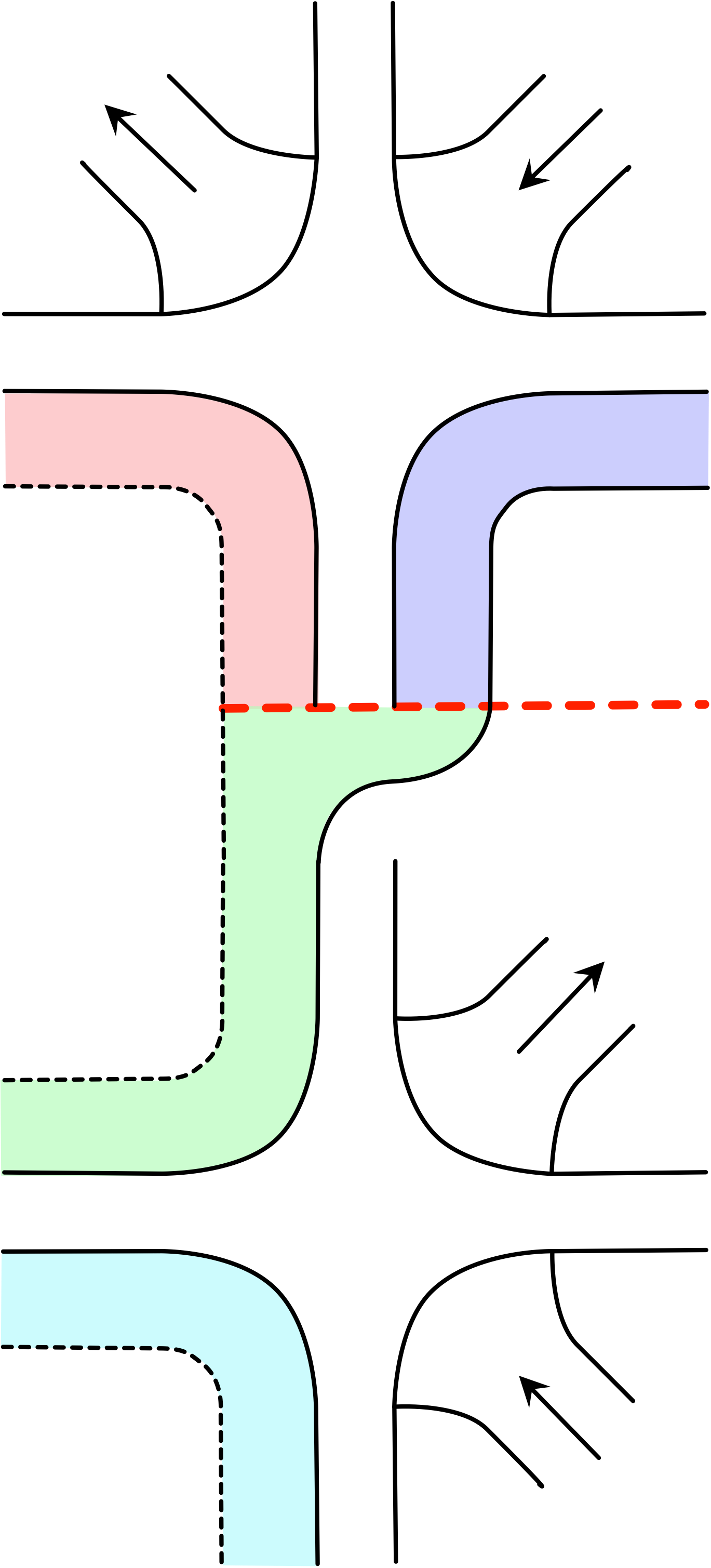}
		\put(-17,70){\small $Y_{12}\Rightarrow$}
		\put(-55,20){\small $(Y_{12}\circ_h Y_{23})\circ_h Y_{34}\Leftarrow$}
		\put(47,70){$\Leftarrow Y_{23}$}
\put(15,-7){$Y_{34}$}
\put(50,53){$\widehat{\ Y_{12}\circ_hY_{23}\ }$}
	\end{overpic}	
\medskip
	\caption{}
	\label{Pic9} 
\end{figure}
On the other hand, the map $\alpha\circ (\Id\otimes \alpha)$ is identical to the one induced from Figure \ref{Pic10}, which can be obtained from Figure \ref{Pic9} by continuously varying the metric and the closed 2-form, which implies that the diagram \eqref{E2.1} is commutative and completes the proof of \ref{G3}. \qed
	\begin{figure}[H]
	\centering
	\begin{overpic}[scale=.08]{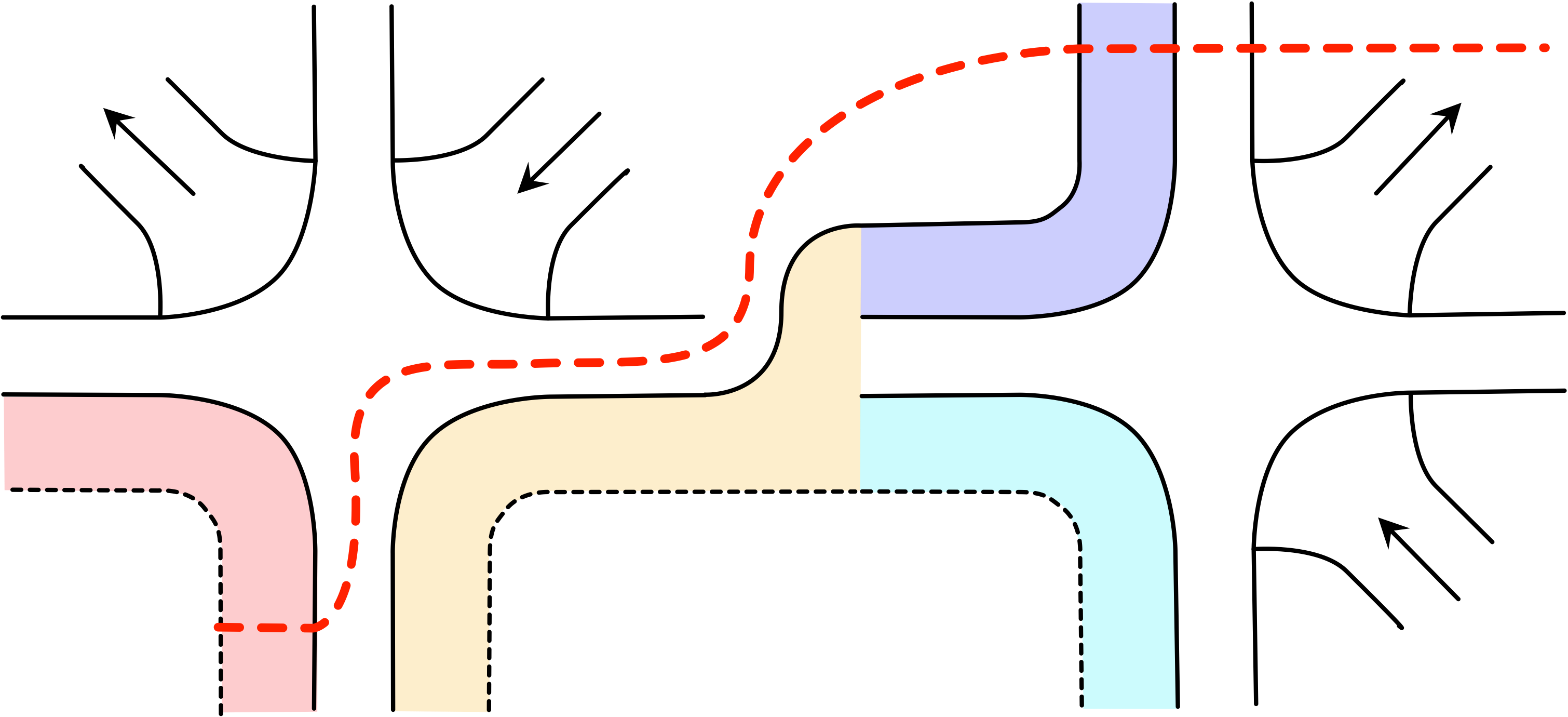}
		\put(-12,15){\small $Y_{12}\Rightarrow$}
		\put(5,-5){\small $Y_{12}\circ_h (Y_{23}\circ_h Y_{34})$}
		\put(70,48){$Y_{23}$}
		\put(70,-5){$Y_{34}$}
		\put(100,40){$\widehat{\ Y_{12}\circ_hY_{23}\ }$}
	\end{overpic}	
	\medskip
	\caption{}
	\label{Pic10} 
\end{figure}
\subsection{Proof of \ref{G4}} We have to show that the gluing map 
$$\alpha: \HM_*(\y_{12})\otimes\HM_*(e_{\TSigma_2})\to\HM_*(\y_{12}\circ_h e_{\TSigma_2})=\HM_*(\y_{12})$$
agrees with $\Id\otimes_\NR \iota_{\TSigma_2}$. Equivalently, we prove that 
\[
\tilde{\alpha}:\HM_*(\y_{12})\xrightarrow{\Id\otimes \iota_{\TSigma_2}^{-1}(1)} \HM_*(\y_{12})\otimes\HM_*(e_{\TSigma_2})\xrightarrow{\alpha}\HM_*(\y_{12}).
\]
is the identity map. We start with a few reductions:

\Step 1. $\tilde{\alpha}$ is an isomorphism. This is by Theorem \ref{T3.3}.

\Step 2. It suffices to verify the special case when $\TSigma_1=\TSigma_2$ and $\y_{12}=e_{\TSigma_2}$. 

Indeed, if the statement holds for this special case, consider the diagram:
\[
\begin{tikzcd}[column sep=2.3cm]
\HM_*(\y_{12})\arrow[r,"\Id\otimes\iota_{\TSigma_2}^{-1}(1)\otimes \iota_{\TSigma_2}^{-1}(1)"]& \HM_*(\y_{12})\otimes \HM_*(e_{\TSigma_2})\otimes \HM_*(e_{\TSigma_2})\arrow[r,bend left,"\alpha (\Id\otimes\alpha)"']\arrow[r,bend right,"\alpha(\alpha\otimes\Id)"] & \HM_*(\y_{12}).
\end{tikzcd}
\]
Applying Theorem \ref{T2.5} \ref{G3} to the triple $(\y_{12},e_{\TSigma_2},e_{\TSigma_2})$, we obtain that $\tilde{\alpha}^2=\tilde{\alpha}$; so $\tilde{\alpha}=\Id$.

\medskip

\Step 3. In the case when $\TSigma_1=\TSigma_2$ and $\y_{12}=e_{\TSigma_2}$, the group $\HM_*(e_{\TSigma_2})$ has rank $1$; so $\tilde{\alpha}:\HM_*(e_{\TSigma_2})\to \HM_*(e_{\TSigma_2})$ is a multiplication map. Let $\x$ be the cobordism inducing the gluing map $\alpha$. In this special case, the opposite cobordism $\x'$ of $\x$ is identical to $\x$. Theorem \ref{T3.3} then implies that $\tilde{\alpha}^2=\Id$; so $\tilde{\alpha}=\Id$. This completes the proof of Theorem \ref{T2.5} \ref{G4}.\qed

\section{Generalized Cobordism Maps}\label{Sec4}

2-cell morphisms in the category $\AT(\emptyset, \TSigma)$ are given by strict cobordisms: the induced cobordisms between boundaries are necessarily standard products, as required by Property \ref{Q3}. The primary goal of this section is to remove this constraint and define the generalized cobordism maps. 


For any 1-cells $\y_1\in \AT(\emptyset,\TSigma_1)$ and $\y_2\in \AT(\emptyset,\TSigma_2)$, a general cobordism $(X,W): (Y_1,\Sigma_1)\to (Y_2,\Sigma_2)$ is a 4-manifold with corners. To better package the data of closed 2-forms, however, we shall adopt a different point of view and introduce a new category $\AT_1$.
\begin{definition} Consider the category $\AT_1$ with objects 
	\[
	\Ob \AT_1=\coprod_{\TSigma} \Ob \AT(\emptyset,\TSigma).
	\] 
	For any $\y_1\in \AT(\emptyset,\TSigma_1)$ and $\y_2\in \AT(\emptyset,\TSigma_2)$, a morphism in $\AT_1(\y_1,\y_2)$ is a pair $(\x_{12},\BW_{12})$ where
	\[
	\BW_{12}\in \AT(\TSigma_1,\TSigma_2),\  \x_{12}\in \Hom_{\AT(\emptyset, \TSigma_2)}(\y_1\circ_h \BW_{12},\y_2). 
	\]
	For morphisms $(\x_{12},\BW_{12})\in \AT_1(\y_1,\y_2)$ and $(\x_{23},\BW_{23})\in \AT_1(\y_2,\y_3)$, their composition $(\x_{13},\BW_{13})\in \AT_1(\y_1,\y_3)$ is defined as
	\[
		\BW_{13}=\BW_{12}\circ_h\BW_{23},\  \x_{13}=(\x_{12}\circ_h \Id_{\BW_{23}})\circ \x_{23},
	\]
	The associativity can be easily checked using the digram that represents $\x_{13}$:
	\begin{equation}\label{E5.2}
	\begin{tikzcd}
\emptyset \arrow[r,"\y_1"]  \arrow[d,equal]& \TSigma_1 \arrow[r,"\BW_{12}"] \arrow[rd, phantom, "\x_{12}"] & \TSigma_2 \arrow[r,dashed, "\BW_{23}"] \arrow[d,equal] \arrow[rd, phantom, "\Id_{\BW_{23}}"]& \TSigma_3 \arrow[d,equal] \\
\emptyset \arrow[rr,"\y_2"]  \arrow[d,equal]&   & \TSigma_2 \arrow[r,"\BW_{23}"] \arrow[rd, phantom, "\x_{23}"]& \TSigma_3 \arrow[d,equal]\\
\emptyset\arrow[rrr,"\y_3"] & & {}&\TSigma_3.
\end{tikzcd}\qedhere
	\end{equation}
\end{definition}
\begin{corollary} We define a fake functor $\HM_*: \AT_1\to \AR(\star)$ as follows. For any object $\y_i\in \AT(\emptyset,\TSigma_i)$, we assign its monopole Floer homology group $\HM_*(\y_i)$. For any morphism $(\x_{12},\BW_{12}):\y_1\to \y_2$, we assign the map $	\HM_*(\x_{12},\BW_{12})\colonequals \HM_*(\x_{12})\circ \alpha: $
	\[
\HM_*(\y_1)\otimes \HM_*(\BW_{12})\xrightarrow{\alpha} \HM_*(\y_1\circ_h\BW_{12})\xrightarrow{\HM_*(\x_{12})} \HM_*(\y_2).
	\]
	
	This assignment $\HM_*$ fails to be a functor, since the ordinary composition law is violated. The replacement is a commutative diagram relating $\HM_*(\x_{12},\BW_{12})$, $\HM_*(\x_{23},\BW_{23})$ and the map of their composition $(\x_{13},\BW_{13})=(\x_{12},\BW_{12})\circ (\x_{23},\BW_{23})$:
		\begin{equation}\label{E5.1}
	\begin{tikzcd}
	\HM_*(\y_1)\otimes \HM_*(\BW_{12})\otimes \HM_*(\BW_{23})\arrow[r,"\Id\otimes\alpha"] \arrow[d,"{\HM(\x_{12},\BW_{12})\otimes \Id}"] & \HM_*(\y_1)\otimes \HM_*(\BW_{13}) \arrow[d, "{\HM_*(\x_{13},\BW_{13})}"]\\
	\HM_*(\y_2)\otimes \HM_*(\BW_{23})\arrow[r,"{\HM_*(\x_{23},\BW_{23})}"] & \HM_*(\y_3). 
	\end{tikzcd}
	\end{equation}
\end{corollary}
\begin{proof} The commutativity of $\eqref{E5.1}$ is obtained by applying the monopole Floer 2-functor $\HM_*$ in Theorem \ref{T2.5} to the digram \eqref{E5.2}.
\end{proof}

In order to obtain a genuine functor, we have to enlarge the category $\AT_1$ to incorporate an element of $\HM_*(\BW_{12})$ for each morphism $(\x_{12},\BW_{12})$. 
\begin{definition} The category $\AT_2$ has the set of same objects as $\AT_1$. For any objects $\y_1\in \AT(\emptyset,\TSigma_1)$ and $\y_2\in \AT(\emptyset,\TSigma_2)$, a morphism in $\AT_2(\y_1,\y_2)$ is now a triple $(\x_{12},\BW_{12}, a_{12})$ where
	\[
	(\x_{12},\BW_{12})\in \AT_1(\y_1,\y_2) \text{ and } a_{12}\in \HM_*(\BW_{12}). 
	\]
		For $(\x_{12},\BW_{12},a_{12})\in \AT_2(\y_1,\y_2)$ and $(\x_{23},\BW_{23},a_{23})\in \AT_2(\y_2,\y_3)$, their composition $$(\x_{13},\BW_{13},a_{13})\in \AT_2(\y_1,\y_3)$$ is defined as
	\[
	\BW_{13}=\BW_{12}\circ_h\BW_{23},\  \x_{13}=(\x_{12}\circ_h \Id_{\BW_{23}})\circ \x_{23},\ a_{23}=\alpha(a_{12}\otimes a_{23}). \qedhere
	\]
\end{definition}

\begin{corollary}\label{C4.4} We define a functor $\HM_*: \AT_2\to \AR(\star)$ as follows. For any object $\y_i\in \AT(\emptyset,\TSigma_i)$, we assign its monopole Floer homology group $\HM_*(\y_i)$. For any morphism $(\x_{12},\BW_{12},a_{12}):\y_1\to \y_2$, define the map $\HM_*(\x_{12},\BW_{12},a_{12})$ to be the composition
	\[
	\HM_*(\y_1)\xrightarrow{\alpha(\cdot, a_{12})} \HM_*(\y_1\circ_h\BW_{12})\xrightarrow{\HM_*(\x_{12})} \HM_*(\y_2).
	\]
	Then $\HM_*: \AT_2\to \AR(\star)$ is a functor in the classical sense.
\end{corollary}
\begin{proof} The functoriality of $\HM_*$ follows from the commutative digram \eqref{E5.1}.
\end{proof}

\section{Invariance of Boundary Metrics}\label{Sec5} 

For any 1-cell $\y\in \AT(\emptyset, \TSigma)$, we defined its monopole Floer homology group in \cite{Wang20}; but its invariance is only proved in a weak sense:

\begin{theorem}[{\cite[Remark 1.8\ \&\ Corollary 19.10]{Wang20}}]\label{T5.2}For any $T$-surface $\TSigma$ and any 1-cell $\y\in \AT(\emptyset, \TSigma)$, the monopole Floer functor from Theorem \ref{T3.2} implies the invariance of $\HM_*(\y,\bs)$ when
	\begin{itemize}
		\item we change the cylindrical metric $g_Y$ in \ref{P1.5}; 
		\item we replace $\omega$ by $\omega+d_{\hy}b$ for a compactly supported 1-form $b\in \Omega^1_c(Y,i\R)$ in \ref{P2};
		\item we apply an isotopy to the diffeomorphism $\psi:\partial Y\to \Sigma$. 
	\end{itemize}
\end{theorem}

 In this section, we strengthen this invariance result by showing that $\HM_*(\y)$ is independent of the flat metric of $\Sigma$ and depends only on the class $[\omega]\in H^2(Y;i\R)$ and $[*_\Sigma\lambda]\in H^1(\Sigma, i\R)$. In particular, the Floer homology $\HM_*(\y)$ is a topological invariant for the triple $(Y,[\omega],[*_\Sigma\lambda])$. The main result is as follows:

\begin{theorem}\label{T5.1} Suppose that $\TSigma_i=(\Sigma, g_i, \lambda_i,\mu_i), i=1,2$ have the same underlying oriented surface $\Sigma$ and 
	\begin{align*}
	[*_1\lambda_1]&=[*_2\lambda_2]\in H^1(\Sigma, i\R),& [\mu_1]&=[\mu_2] \in H^2(\Sigma, i\R).
	\end{align*}
If 1-cells $\y_i\in \AT(\emptyset,\TSigma_i),i=1,2$ have the same underlying 3-manifold $Y_1\cong Y_2$ (this diffeomorphism is compatible with the identification maps $\psi_i:\partial Y_i\to \Sigma, i=1,2$) and 
\[
[\omega_1]=[\omega_2]\in H^2(Y, i\R),
\]
then $\HM_*(\y_1, \bs)\cong \HM_*(\y_2,\bs)$ for any relative \spinc structure $\bs\in \Spincr(Y)$. 
\end{theorem}

The proof of  Theorem \ref{T5.1} is based on the Gluing Theorem \ref{T2.5} and a special property concerning the product manifold $[-3,3]_s\times \Sigma$:

\begin{lemma}\label{L5.2} Under the assumptions of Theorem \ref{T5.1}, consider the 1-cell $\y_0\in \AT(\TSigma_1, \TSigma_2)$ with $Y_0=[-3,3]_s\times\Sigma$ and $\psi=\Id: \partial Y_0\to (-\Sigma)\cup \Sigma$. Then for any cylindrical metric of $Y_0$ and any compatible 2-form $\omega_0\in \Omega^2(Y_0, i\R)$, we have 
	\[
	\HM_*(\y_0,\bs)\cong \left\{ \begin{array}{cl}
	\NR & \text{ if }\ \bs=\bs_{std},\\
	\{0\} & \text{ otherwise. }
	\end{array}
	\right.
	\] 
\end{lemma}
\begin{proof}[Proof of Theorem \ref{T5.1}] By applying the gluing functor from Theorem \ref{T2.5}
	\[
	\alpha: \AT(\emptyset, \TSigma_1)\times\AT(\TSigma_1,\TSigma_2)\to \AT(\emptyset, \TSigma_2). 
	\]
	to the pair $(\y_1, \y_0)$, we deduce from Theorem \ref{T3.5} and Lemma \ref{L5.2} that
	\[
	\HM_*(\y_1,\bs)\cong \HM_*(\y_1\circ_h \y_0,\bs).
	\] 
	for any $\bs\in \Spincr(Y)$. Now $\y_1\circ_h \y_0$ and $\y_2$ are 1-cells in the same strict cobordism category $\AT(\emptyset, \TSigma_2)$. The difference $\omega_1\circ_h\omega_0-\omega_2$ determines a class $\beta\in H^2(Y, \partial Y; i\R)$. Our goal here is to choose $\omega_0$ properly so that $\beta=0$. This can be always done, because $[\omega_2]=[\omega_1]=[\omega_1\circ_h\omega_0]$; so $\beta$ lies in the image of $H^1(\partial Y; i\R)$:
	\[
	\begin{array}{rcl}
\cdots	\to H^1(\partial Y; i\R) \to &H^2(Y,\partial Y;i\R)&\to H^2(Y;i\R)\to\cdots \\
	&\beta&\mapsto 0.
	\end{array}
	\]
	When $\beta=0$, we can then identify $\HM_*(\y_1\circ_h \y_0,\bs)$ with $\HM_*(\y_2,\bs)$ using Theorem \ref{T5.2}. This completes the proof of Theorem \ref{T5.1}. 
\end{proof}

The proof of Lemma \ref{L5.2} relies on a computation for the 3-torus $\T^3$, which generalizes Lemma \ref{L3.4}.

\begin{lemma}\label{L5.4} For any closed 2-form $\omega\in \Omega^2(\T^3, i\R)$, suppose that the period class $c\colonequals -2\pi i[\omega]$ is neither negative monotone nor balanced for any \spinc structures on $\T^3$, then 
	\[
	\HM_*(\T^3, \s, c; \NR_\omega)=\left\{\begin{array}{cl}
	\NR & \text{ if } c_1(\s)=0,\\
	\{0\} & \text{ otherwise.}
	\end{array}
	\right.
	\]
\end{lemma}
\begin{proof}[Proof of Lemma \ref{L5.4}]The meaning of Lemma \ref{L5.4} will become more transparent when we review the theory of closed 3-manifolds in Section \ref{Sec6}. When $c_1(\s)\neq 0$, the statement can be verified as in the proof of Lemma \ref{L3.4}, by working with a product metric and a harmonic 2-form $\omega$. When $c_1(\s)=0$, the statement follows from \cite[Lemma 3.1]{Taubes01}. Alternatively, we may apply Proposition \ref{P4.2} to reduce the problem to the case of exact perturbations.
\end{proof}

\begin{proof}[Proof of Lemma \ref{L5.2}] With loss of generality, we assume that $\Sigma$ is connected. Since the roles of $\TSigma_1$ and $\TSigma_2$ are symmetric, consider a similar 1-cell $\y_0'\in \AT(\TSigma_2,\TSigma_1)$. Now regard $\y_0$ and $\y_0'$ as 1-cells in $\AT(\emptyset, \TSigma_2\cup (-\TSigma_1))$ and $\AT(\TSigma_2\cup (-\TSigma_1),\emptyset)$ respectively. We apply Theorem \ref{T3.3} to obtain that 
	\begin{equation}\label{E5.3}
	\HM_*(\y_0)\otimes_\NR \HM_*(\y_0')\cong 	\HM_*(\T^3, \s, c; \NR_\omega)
	\end{equation}
	where $\omega=\omega_0\circ_h\omega_0'$ is a closed 2-form on the 3-torus $\T^3=S^1\times\Sigma$. The condition of Lemma \ref{L5.4} can be verified, since
	\[
	|\langle \omega, [\Sigma]\rangle|<2\pi \text{ and }\neq 0 \text{ by properties }\ref{T3} \text{ and }\ref{T1}.
	\]
	
We conclude from \eqref{E5.3} and Lemma \ref{L5.4} that 
\[
\rank_{\NR} \HM_*(\y_0)=\rank_{\NR}\HM_*(\y_0')=1. 
\]
In particular, the group $\HM_*(\y_0,\bs)$ vanishes except for one particular relative \spinc structure. By Theorem \ref{T3.4}, the Euler characteristic $\chi(\HM_*(\y_0,\bs))$ is independent of the metric and the 2-form $\omega_0$. We conclude from the computation of $\HM_*(e_{\TSigma_1},\bs_{std})$ that 
\[
\chi (\HM_*(\y_0,\bs_{std}))=1.
\]
Thus $\HM_*(\y_0,\bs_{std})\neq \{0\}$. This completes the proof of Lemma \ref{L5.2}.
\end{proof}

\part{Monopoles and Thurston Norms}

Let $F=\bigcup_{i=1}^n F_i$ be a compact oriented surface with $F_i$ connected. Recall that the norm of $F$ is defined to be 
\[
x(F)\colonequals -\sum_{i} \min\{\chi(F_i),0\}.
\] 
For any 3-manifold $Y$ with toroidal boundary, Thurston \cite{T86} introduced a semi-norm $x(\cdot)$ on $H_2(Y,\partial Y;\R)$ such that for any integral class $\kappa\in H_2(Y,\partial Y;\Z)$, we have 
\[
x(\kappa)\colonequals \min\{ x(F): (F,\partial F)\subset (Y,\partial Y) \text{ properly embedded and }[F]=\kappa\}. 
\]

\medskip

In this part, we show that the monopole Floer homology $\HM_*(\y)$ defined in \cite{Wang20} detects the Thurston norm and fiberness of $Y$, if the underlying 3-manifold $Y$ is connected and irreducible, generalizing the previous results for closed 3-manifolds \cite{Bible,Ni08,Ni09b}. The same detection theorems for link Floer homology have been obtained in \cite{OS08, Ni09, J08,GL19} and \cite{Ni07}. The main results are Theorem \ref{T7.1} and \ref{T7.2}.

For the proof of Theorem \ref{T7.1}, we use the Gluing Theorem \ref{T2.5} to reduce the problem to the double of $Y$, and apply the non-vanishing result \cite[Corollary 41.4.3]{Bible} for closed 3-manifold. However, we have to adapt this corollary first to the case of non-exact perturbations, which is done in Section \ref{Sec6}.

The proof of Theorem \ref{T7.2} is accomplished in Section \ref{Sec9}, which exploits the relation of $\HM_*(\y)$ with sutured Floer homology, as discussed in Section \ref{Sec8}. Any 3-manifold with toroidal boundary is a sutured manifold in the sense of Gabai \cite{Gabai83}, but it is not an example of balanced sutured manifolds in the sense of \cite[Definition 2.2]{J06}. One may think of our construction as a natural extension of the sutured Floer homology. Although the author was not able to prove a general sutured manifold decomposition theorem, a preliminary result, Theorem \ref{T8.1}, will be supplied in Section \ref{Sec8} to justify this heuristic.

\section{Closed 3-Manifolds Revisited}\label{Sec6}

Throughout this section, we will take $Y$ to be a closed connected oriented 3-manifold. In \cite{Bible}, Kronheimer and Mrowka introduced three flavors of monopole Floer homology groups for any \spinc structure $\s$ on $Y$, which fit into a long exact sequence:
\begin{equation}\label{E6.1}
\cdots\xrightarrow{i_*}  \fHM_*(Y,\s; \Gamma)\xrightarrow{j_*}\tHM_*(Y,\s; \Gamma)\xrightarrow{p_* }\bHM_*(Y,\s;\Gamma)\xrightarrow{i_*}\cdots
\end{equation}
Here $\Gamma$ is any local coefficient system on the blown-up configuration space $\CB^{\sigma}(Y,\s)$. For instance, one may take $\Gamma$ to be the trivial system with $\Z$ coefficient. For any real 1-cycle $\xi$, we can also define the local coefficient system $\Gamma_\xi$ as in \cite[Section 3.7]{Bible}, whose fiber is always $\R$. We write $\HM_*(Y,\s;\Gamma)\colonequals \im j_*$ for the reduced Floer homology. The third group $\bHM_*(Y,\s;\Gamma)$ in \eqref{E6.1} is trivial if 
\begin{itemize}
\item $c_1(\s)\in H^2(Y;\Z) $ is non-torsion, or 
\item $\Gamma=\Gamma_\xi$ and $[\xi]\neq 0\in H_1(Y;\R)$; see \cite[Proposition 3.9.1]{Bible}.
\end{itemize}
 In either case, the map $j_*$ is an isomorphism, so $\HM_*(Y,\s;\Gamma)\cong \fHM_*(Y,\s;\Gamma)$. 

Monopole Floer homology detects the Thurston norm $x(\cdot)$ on $H_2(Y;\R)$. The next theorem expand on what lies behind this slogan: 
\begin{theorem}[{\cite[Corollary 40.1.2 \& 41.4.3]{Bible}}]\label{T4.1} Let $Y$ be a closed oriented 3-manifold with $b_1(Y)>0$ and $\kappa\in H_2(Y,\Z)$ be any integral class.
\begin{enumerate}
\item $($The adjunction inequalities$)$ For any local coefficient system $\Gamma$ and any \spinc structure $\s$ with $|\langle c_1(\s),\kappa\rangle|>\|\kappa\|_{Th}$, the monopole Floer homology group $\HM_*(Y,\s;\Gamma)$ is trivial. 
\item\label{2} If in addition $Y$ is irreducible, then there is a \spinc structure $\s$ on $Y$ such that 
\[
\langle c_1(\s), \kappa\rangle=\|\kappa\|_{Th} \text{ and }\HM_*(Y,\s; \Gamma_\xi)\neq \{0\} 
\]
for any 1-cycle $\xi$ with $[\xi]\neq 0\in H_1(Y;\R)$. 
\end{enumerate}
\end{theorem}

In particular, when $[\xi]\neq 0\in H_1(Y,\R)$, the subset 
\[
\{c_1(\s): \HM_*(Y,\s;\Gamma_\xi)\neq \{0\}\}\subset H^2(Y,\R)
\]
determines the Thurston norm on $H_2(Y,\R)$ in the same way that the Newton polytope of the Alexander polynomial determines the Alexander norm; see \cite{M02}. 

\medskip

However, Theorem \ref{T4.1} is stated only for the monopole Floer homology defined using exact perturbations. In \cite[Section 30]{Bible}, these groups are extended for non-exact perturbations:
\begin{equation}\label{E4.1}
\cdots\xrightarrow{i_*}  \fHM_*(Y,\s,c; \Gamma)\xrightarrow{j_*}\tHM_*(Y,\s,c; \Gamma)\xrightarrow{p_* }\bHM_*(Y,\s, c;\Gamma)\xrightarrow{i_*}\cdots
\end{equation}
where $c\in H^2(Y;\R)$ is the period class and $\Gamma$ is any $c$-complete local coefficient system in the sense of \cite[Definition 30.2.2]{Bible}. These groups are defined using the Seiberg-Witten equations \eqref{3DDSWEQ} on $Y$ with the closed 2-form $\omega$ belonging to the class $[\omega]=\frac{i}{2\pi}\cdot c\in H^2(Y,i\R)$. 
\smallskip

The purpose of this section is to understand the extent to which Theorem \ref{T4.1} generalizes to these groups defined using non-exact perturbations. With that said, the results of this section follow almost trivially from the general theory in \cite{Bible}; no originality is claimed here.

\subsection{Statements} Fix a \spinc structure $\s$ on $Y$. Recall from \cite[Definition 29.1.1]{Bible} that the non-exact perturbation associated to a closed 2-form $\omega\in \Omega^2(Y, i\R)$ is called monotone if 
\[
2\pi^2c_1(\s)+c=2\pi^2c_1(\s)\cdot t
\]
for some $t\in \R$, where $c\colonequals -2\pi i[\omega]$ is the period class of this perturbation. Furthermore, it is called balanced, positively or negatively monotone if $t=0$, $t>0$ or $<0$ respectively. We write
\[
\HM_*(Y,\s,c;\Gamma)\colonequals \im j_*\subset \tHM_*(Y,\s,c;\Gamma)
\] 
for the reduced monopole Floer homology. The third group $\bHM_*(Y,\s,c;\Gamma)$ in \eqref{E4.1} is always trivial unless the period class $c=c_b\colonequals -2\pi^2c_1(\s)$ is balanced. Thus $\HM_*(Y,\s,c;\Gamma)= \tHM_*(Y,\s,c;\Gamma)$ if $c\neq c_b$. 

\smallskip

We focus on the local coefficient systems $\NR_\omega$ on $\CB^\sigma(Y,\s)$, whose fiber at every point is always the mod 2 Novikov ring $\NR$. The fundamental group of $\CB^\sigma(Y,\s)$ is $H^1(Y,\Z)=\pi_0(\CG(Y))$. The monodromy of $\NR_\omega$ is then defined by sending $z\in H^1(Y,\Z)$ to $q^{f(z)}$ with
\begin{align}\label{E4.6}
f(z)\colonequals \langle (2\pi^2c_1(\s)+c)\cup z, [Y]\rangle\in \R.
\end{align}

For the adjunction inequality, it suffices to consider \spinc structures with non-torsion $c_1(\s)$. The problem already occurs for the 3-torus $\T^3=\T^2\times S^1$ in Lemma \ref{L3.4}. Let $d\in H^2(\T^3, \Z)$ be the Poincar\'{e} dual of $ \{pt\}\times S^1$. Suppose $\omega=i\delta\cdot d$ and consider the \spinc structure $\s$ on $\T^3$ with $c_1(\s)=2\cdot d$. If $\delta<-2\pi $, then this perturbation is negatively monotone. In this case, the moduli space of the Seiberg-Witten equations \eqref{3DDSWEQ} is non-empty, and is diffeomorphic to $\T^2\times\{pt\}$. As one may verify, for either $\Gamma=\Z$ or $\NR_\omega$, the monopole Floer homology group $\HM_*(\T^3,\s,c;\Gamma)$ is non-trivial, and so the adjunction inequality in Theorem \ref{T4.1} is violated in this case. 

\smallskip

However, by Proposition \ref{P4.2}, negatively monotone perturbations are the only exceptions.

\begin{proposition}\label{P4.2} Let $Y$ be any closed oriented 3-manifold with $b_1(Y)\geq 2$ and $\s$ be any non-torsion \spinc structure. Suppose that $[\omega]\in H^2(Y,i\R)$ is neither negatively monotone nor balanced with respect to $\s$, then there is an isomorphism 
	\[
	\HM_*(Y,\s,c;\NR_\omega)\cong \HM_*(Y,\s;\NR_\omega)
	\]
	where the second group is defined using an exact perturbation with local system $\NR_\omega$. In particular, the adjunction inequality from Theorem \ref{T4.1} holds also for the group $\HM_*(Y,\s,c;\NR_\omega)$: it is trivial, whenever $\langle c_1(\s),\kappa\rangle >\|\kappa\|_{Th}$ for some integral class $\kappa\in H_2(Y;\Z)$.  
\end{proposition} 


The non-vanishing result is more robust: it suffices to rule out exact perturbations:

\begin{proposition}\label{P4.5} Let $Y$ be any irreducible closed oriented 3-manifold with $b_1(Y)\geq 1$. If the period class $c=-2\pi i[\omega]\in H^2(Y; \R)$ is non-exact, then for any integral class $\kappa\in H^2(Y;\Z)$, there is a \spinc structure $\s$ such that 
	\[
\langle c_1(\s), \kappa\rangle=\|\kappa\|_{Th} \text{ and }	\HM_*(Y,\s,c;\NR_\omega)\neq \{0\}.
	\] 
\end{proposition}

To ease our notation, we introduce the following shorthand: 

\begin{definition}\label{D6.4} For any closed 2-form $\omega\in \Omega^2(Y; i\R)$ and any integral class $\kappa\in H_2(Y;\Z)$, we write $\HM(Y,[\omega])$ for the direct sum 
	\[
	 \bigoplus_{\s} \HM_*(Y,\s,c;\NR_\omega)
	\]
	and $\HM(Y,[\omega]|\kappa)$ for the subgroup
	\[
 \bigoplus_{\langle c_1(\s),\kappa\rangle=x(\kappa)} \HM_*(Y,\s,c;\NR_\omega). 
	\]
For any connected oriented subsurface $F\subset Y$, define similarly
		\[
\HM(Y,[\omega]|F)=	\bigoplus_{\langle c_1(\s),\kappa\rangle=x(F)} \HM_*(Y,\s,c;\NR_\omega).\qedhere
	\]
\end{definition}

\begin{remark}\label{R4.5} When $[\omega]=0\in H^2(Y; i\R)$ and $c_1(\s)\neq 0$, the local system $\NR_\omega$ is not trivial. Nevertheless, by \cite[P.288 (16.5)]{Bible}, we still have an isomorphism:
	\[
	\HM_*(Y,\s,c;\NR_\omega)\cong \HM_*(Y,\s; \NR).
	\]
	The latter group is defined using the trivial system with coefficients $\NR$. Set
	\[
	\HM_*(Y|\kappa)\colonequals\bigoplus_{\langle c_1(\s),\kappa\rangle=x(\kappa)} \HM_*(Y,\s;\NR). 
	\]
	Then $	\HM_*(Y|\kappa)\cong 	\HM_*(Y,[0]|\kappa)$. We define similarly 	the group $\HM_*(Y|F)$. 
\end{remark}

\begin{corollary}\label{C6.5} If $Y$ is irreducible with $b_1(Y)>0$ and $[\omega]\neq 0\in H^2(Y; i\R)$, then $\HM(Y,[\omega]|\kappa)\neq \{0\}$ for any integral class $\kappa\in H^2(Y;\Z)$. 
\end{corollary}

The proof of Proposition \ref{P4.2} and \ref{P4.5} will dominate the rest of this section. 

\subsection{Completions of Chain Complexes} The proof of Proposition \ref{P4.2} is achieved in two steps: the first step is to relate the group $\HM_*(Y,\s,c;\NR_\omega)$ with the Floer homology associated to a balanced class $c_b=-2\pi^2c_1(s)$. For monotone perturbations, this is already done in \cite[Theorem 31.1.1\ \&\ 31.5.1]{Bible}. We shall describe a slightly different setup that allows generalization for any non-balanced perturbations. 
 
Consider the monopole Floer chain complexes of $(Y,\s, c_b)$ associated to the local system $\NR_\omega$:
\[
\widecheck{C}_*(Y,\s, c_b; \NR_\omega),\ \widehat{C}_*(Y,\s, c_b; \NR_\omega),\ \overline{C}_*(Y,\s, c_b; \NR_\omega). 
\]
As modules, they are free over the Novikov ring $\NR$, but none of them is finitely generated. Since the period class $c_b$ is balanced, we have to blow up the configuration space in order to define the Floer homology. Suppose an admissible perturbation of the Chern-Simons-Dirac functional $\CL_\omega$ is chosen, then each reducible solution $[\alpha]$ of the perturbed 3-dimensional Seiberg-Witten equations will contribute infinitely many generators to each of $\widecheck{C}_*,\widehat{C}_*,\overline{C}_*$, corresponding to the eigenvectors of the Dirac operator at $[\alpha]$.

As noted in \cite[Section 30.1]{Bible}, we can form a chain level completion using the filtration of eigenvalues. Label the reducible solutions in the quotient configuration space $\CB(Y,\s)$ as $[\alpha^1],\cdots,[\alpha^p]$ and label the corresponding critical points in the blown-up space as $[\alpha^r_i]$ with $i\in \Z$ and $1\leq r\leq p$, so that $[\alpha^r_i]$ corresponds to the eigenvalue $\lambda^r_i$ of the perturbed Dirac operator at $[\alpha^r]$ and $$\cdots<\lambda^r_i<\lambda^r_{i+1}<\cdots,$$ with $\lambda^r_0$ being the first positive one. For any $m\geq 1$, let
\begin{equation}\label{E4.3}
 \widehat{C}_*(Y,\s, c_b;  \NR_\omega)_m\subset  \widehat{C}_*(Y,\s, c_b; \NR_\omega),
\end{equation}
be the subgroup generated by $[\alpha^r_i]$ with $i\leq -m$. This defines a filtration on $\widehat{C}_*$, called the $\lambda$-filtration. We form the completion 
\[
 \widehat{C}_\bullet(Y,\s, c_b; \NR_\omega)\supset\widehat{C}_*(Y,\s, c_b; \NR_\omega).  
\]
The same construction applies also to the bar-version; so we obtain 
\[
\overline{C}_\bullet(Y,\s, c_b; \NR_\omega)\supset \overline{C}_*(Y,\s, c_b; \NR_\omega).
\]

On the other hand, $\widehat{C}_*$ carries an additional filtration arising from the base ring $\NR$, called the $\NR$-filtration. One may formally write 
\[
 \widehat{C}_*(Y,\s, c_b; \NR_\omega)=\bigoplus_{n\geq 0} \NR e_n,
\]
by identifying the fibers of $\NR_\omega$ at different critical points using a collection of paths in the blown-down space $\CB(Y,\s)$. We may start with the group ring $\BF_2[\R]$ and obtain $\NR$ by taking the completion in the negative direction. This filtration on $\NR=\BF_2[\R]^-$ then induces the $\NR$-filtration on the free module $ \widehat{C}_*(Y,\s, c_b; \NR_\omega)$; let 
\[
\widehat{C}_\diamond(Y,\s, c_b; \NR_\omega)\subset \widehat{C}_\bullet (Y,\s,c_b; \NR_\omega),
\]
be the resulting completion. Any element $\sum_{n\geq 0} a_n\cdot e_n\in \widehat{C}_\diamond(Y,\s, c_b; \NR_\omega)$ may have infinitely many non-zero coefficients $a_n\in \NR$, but under the topology of $\NR$, we must have 
\begin{equation}\label{E4.5}
\lim_{n\to\infty} a_n=0.
\end{equation}

The bar-version analogue is constructed in a slightly different way. First, take the $\NR$-completion of its $m$-th filtered subgroup $\overline{C}_*(Y,\s, c_b;  \NR_\omega)_m\subset  \overline{C}_*(Y,\s, c_b; \NR_\omega)$ for each $m\in \Z$, denoted by
\[
\overline{C}_\diamond(Y,\s, c_b;  \NR_\omega)_m.
\]
Next, we form the union:
\[
\overline{C}_\diamond(Y,\s, c_b;  \NR_\omega)\colonequals\bigcup_{m\in \Z} \overline{C}_\diamond(Y,\s, c_b;  \NR_\omega)_m\subset \overline{C}_\bullet(Y,\s, c_b;  \NR_\omega).
\]
An element in $\overline{C}_\diamond(Y,\s, c_b;  \NR_\omega)$ has only finitely many non-zero coefficients for critical points with $\lambda^r_i>0$, while a infinite sum may occur in the negative direction satisfying the convergence condition \eqref{E4.5}. The upshot is that the homology groups of 
\[
\widehat{C}_*(Y,\s, c_b; \NR_\omega),\widehat{C}_\diamond(Y,\s, c_b; \NR_\omega), \overline{C}_\diamond(Y,\s, c_b; \NR_\omega)
\]
form a long exact sequence (by the proof of \cite[Proposition 22.2.1]{Bible}), which fits into a diagram below: 
\begin{equation}\label{E4.2}
\begin{tikzcd}[column sep=1.7em]
\cdots\arrow[r,"i_*"] &\fHM_*(Y,\s,c_b; \NR_\omega)\arrow[r,"j_*"] \arrow[d,equal]&\tHM_*(Y,\s,c_b; \NR_\omega)\arrow[r,"p_*"] \arrow[d,"x_*"]&\bHM_*(Y,\s,c_b;\NR_\omega)\arrow[r,"i_*"]\arrow[d]&\cdots\\
\cdots\arrow[r,"i_\diamond"] &\fHM_*(Y,\s,c_b; \NR_\omega)\arrow[r,"j_\diamond"]\arrow[d,equal] &\tHM_\diamond(Y,\s,c_b; \NR_\omega)\arrow[r,"p_\diamond"]\arrow[d,"y_\diamond"] &\bHM_\diamond(Y,\s,c_b;\NR_\omega)\arrow[r,"i_\diamond"]\arrow[d]&\cdots\\
\cdots\arrow[r,"i_\bullet"] &\fHM_*(Y,\s,c_b; \NR_\omega)\arrow[r,"j_\bullet"] &\tHM_\bullet(Y,\s,c_b; \NR_\omega)\arrow[r,"p_\bullet"] &\bHM_\bullet(Y,\s,c_b;\NR_\omega)\arrow[r,"i_\bullet"]&\cdots
\end{tikzcd}
\end{equation}

The next proposition says that $\HM_*(Y,\s,c;\NR_\omega)$ can be computed in terms of the group $\tHM_\diamond$ associated the balanced period class $c_b\colonequals -2\pi^2c_1(\s)$. 
\begin{proposition}\label{P4.3} For any $(Y,\s)$ and any non-balanced perturbation with period class $c=-2\pi i[\omega]$, we have an isomorphism
	\begin{align*}
		\tHM_*(Y,\s, c;\NR_\omega)&\cong \tHM_\diamond(Y,\s,c_b; \NR_\omega),
	\end{align*}
\end{proposition}
\begin{proof} This result follows from the proof of \cite[Theorem 31.1.1]{Bible}. See \cite[Section 31.2]{Bible}. Since we have used a $c$-complete local coefficient system $\NR_\omega$ here, there is no need to estimate the topological energy $\E_{1}^{top}$ in \cite[(31.5)\  P.611]{Bible}. The argument is even simpler. 
\end{proof}

The second step in the proof of Proposition \ref{P4.2} is to relate $\tHM_\diamond(Y,\s,c_b; \NR_\omega)$ with the bullet-version $\tHM_\bullet (Y,\s,c_b;\NR_\omega)$ using the vertical map $y_\diamond$ in \eqref{E4.2}. In fact, more is true when the class $[\omega]$ is not monotone with respect to $c_1(\s)$:
\begin{proposition}\label{P4.4}  Let $Y$ be any closed oriented 3-manifold with $b_1(Y)\geq 2$ and $\s$ be any non-torsion \spinc structure. Suppose that the period class $c=-2\pi i[\omega]$ is not monotone, then the vertical maps $x_*$ and $y_\diamond$ in digram \eqref{E4.2} are isomorphisms 
	\[
	\tHM_*(Y,\s,c_b; \NR_\omega)\xrightarrow[\cong]{x_*} \tHM_\diamond(Y,\s,c_b; \NR_\omega)\xrightarrow[\cong]{y_\diamond}
	\tHM_\bullet(Y,\s,c_b; \NR_\omega). 
	\]
\end{proposition}

Now we ready to prove Proposition \ref{P4.2}.

\begin{proof}[Proof of Proposition \ref{P4.2}] When $c$ is not monotone, we can combine Proposition \ref{P4.3}\ \&\ \ref{P4.4} with the following isomorphism 
	\begin{equation}\label{E4.8}
		\tHM_\bullet(Y,\s,c_b; \NR_\omega)\cong \HM_*(Y,\s; \NR_\omega)
	\end{equation}
	from \cite[Theorem 31.1.1]{Bible} to conclude. The case when $c$ is positively monotone is already addressed in \cite[Theorem 31.1.2]{Bible}.
\end{proof}

\begin{proof}[Proof of Proposition \ref{P4.4}] Since the left vertical maps in digram \eqref{E4.2} are identity maps, the statement then follows from the fact that for any $\circ\in \{*, \diamond,\bullet\}$, 
	\[
	\bHM_\circ (Y,\s,c_b; \NR_\omega)=\{0\}.
	\]
The case when $\circ=\bullet$ is addressed already in \cite[Theorem 31.1.1]{Bible}; in fact, the group $\bHM_\bullet(Y,\s,c_b; \Gamma)$ always vanishes, no matter which local system $\Gamma$ we use. 

In general, for any $\circ\in \{*, \diamond,\bullet\}$, the group $\bHM_\circ (Y,\s,c_b; \NR_\omega)$ is an instance of coupled Morse homology (see \cite[Section 34]{Bible}) associated to the Picard torus of $Y$ 
$$\T^b\colonequals H^1(Y; \R)/ H^1(Y;\Z),$$ 
with $b\colonequals b_1(Y)\geq 2$. In particular, we may use \cite[Theorem 35.1.6]{Bible} to compute $\bHM_\circ (Y,\s,c_b; \NR_\omega)$ using a standard chain complex, which we recall below. Consider the local coefficient system $\Gamma_\omega$ on $\T^b$ with fiber $\NR[T,T^{-1}]$ and with monodromy given by 
\begin{align}\label{E4.7}
\pi_1(\T^b)\cong H^1(Y,\Z)&\to \NR[T,T^{-1}]^\times\\
z&\mapsto q^{f(z)}T^{g(z)}\nonumber 
\end{align}
where $f(z)$ is as in \eqref{E4.6} and $g(z)=\langle \half c_1(\s)\cup z, [Y]\rangle$. Then the group $\bHM_*(Y,\s,c_b; \NR_\omega)$ is isomorphic to the homology of $(C, \partial=\partial_1+\partial_3)$, where 
\[
C=C(\T^b, h; \Gamma_\omega)\colonequals \bigoplus_{n\geq 0}^N \NR[T,T^{-1}] x_n
\]
is the Morse complex of $\T^b$ with local coefficients $\Gamma_\omega$, defined using a suitable Morse function $h: \T^b\to\R$ and $\partial_1$ is the Morse differential. As a result, $C$ is a finite-rank free module over the ring $\NR[T,T^{-1}]$, generated by critical points $\{x_n\}$ of $h$. The Morse index then induces an additional grading on $C$ such that $\deg\partial_i=-i$ for $i=1,3$. The $m$-th $\lambda$-filtered subgroup of $C$ is given by
\[
C_m\colonequals T^{-m}\bigoplus_{n\geq 0}^N \NR[T^{-1}] x_n, m\in \Z.
\] 

The upshot is that te differential $\partial$ is $\NR[T,T^{-1}]$-linear. The bar-version chain complex $$(\overline{C}_*(Y,\s,c_b;\NR_\omega),\bpartial)$$
which defines $\bHM_*(Y,\bs,c;\NR_\omega)$ admits a very similar structure to $(C,\partial)$, but $\bpartial$ is only linear in $\NR$ and it may have higher components with respect to the Morse grading: $\bpartial=\bpartial_1+\bpartial_3+\bpartial_5+\cdots$. By \cite[Proposition 34.4.1\ \&\ 33.3.8, Theorem 35.1.6]{Bible}, $(\overline{C}_*(Y,\s,c_b;\NR_\omega),\bpartial)$ is homotopic to such a standard complex $(C,\partial)$, preserving the topology induced by $\{(\overline{C}_*)_m\}$ and $\{C_m\}$ respectively.

To compute the homology of $(C,\partial)$, we exploit the spectral sequence induced from the Morse grading, which abuts to $H(C,\partial)$ and whose $E_1$-page is $H(C,\partial_1)$. Since the period class $c=-2\pi i[\omega]$ is not monotone, one verifies immediately that $H(C,\partial_1)=\{0\}$. Indeed, we may write $\T^b=S^1\times \T^{b-1}$ such that a integral multiple of $\{pt\}\times \T^{b-1}$ is dual to the map $g\in H^1(\T^b,\Z)=\Hom(H_1(\T^b, \Z), \Z)$. One may further decompose $\T^{b-1}$ into a product of circles and let $h$ be the sum of Morse functions on circles. It follows that the homology of $\T^{b-1}$ with local coefficients $\Gamma_\omega$ is trivial, since the holonomy of $\Gamma_\omega$ along some $S^1$-factor is non-trivial and lies in $\NR$. Finally, we deduce that $H(C,\partial_1)=\{0\}$ using the K\"{u}nneth formula. 
\smallskip

This finishes the proof when $\circ=*$. When $\circ=\diamond$ or $\bullet$, it suffices to replace the base ring $\NR[T,T^{-1}]$ in \eqref{E4.7} by 
\[
\NR[T,T^{-1}]_\diamond \text{ or } \NR[T,T^{-1}]]
\]
respectively. The ring $\NR[T,T^{-1}]_\diamond$ is obtained by taking the completion of $\NR[T^{-1}]$ with respect to the topology of $\NR$ and then inverting $T^{-1}$. Any element in $\NR[T,T^{-1}]_\diamond$ is then of the form $\sum_{n\geq m}^{\infty} a_n T^{-n}$ for some $m\in \Z$ and $a_n\in \NR$ such that $\lim_{n\to\infty} a_n=0$ in $\NR$; so
\[
\NR[T,T^{-1}]\subset \NR[T,T^{-1}]_\diamond\subset  \NR[T,T^{-1}]].
\]

One may now apply the universal coefficient theorem to conclude.
\end{proof}

\subsection{Proof  of Proposition \ref{P4.5}} To deal with the non-vanishing result, we first look at a non-torsion \spinc structure $\s$ on $Y$. 

If the period class $c$ is non-balanced with respect to $\s$, then the map $j_\bullet=y_\diamond\circ j_\diamond$ in diagram \eqref{E4.2} is an isomorphism, since the group $\bHM_\bullet(Y,\s,c_b;\NR_\omega)$ vanishes by \cite[Theorem 31.1.1]{Bible}. This shows that the vertical map
\[
y_\diamond: \tHM_\diamond(Y,\s,c_b; \NR_\omega)\to
\tHM_\bullet(Y,\s,c_b; \NR_\omega). 
\]
is always surjective. By \eqref{E4.8} and Proposition \ref{P4.3}, we conclude that 
\begin{equation}\label{E4.9}
\rank_\NR  \HM_*(Y,\s, c; \NR_\omega)\geq \rank_\NR\HM_*(Y,\s; \NR_\omega). 
\end{equation}

If the period class $c$ is balanced with respect to $\s$, then $\HM_*(Y,\s, c; \NR_\omega)\colonequals \im j^*$ in the digram \eqref{E4.2}. We apply the same argument to $y_\diamond\circ x_*$; so the inequality \eqref{E4.9} still holds.

When $c_1(\s)$ is torsion and $c$ is non-exact, the inequality \eqref{E4.9} follows from \cite[Theorem 31.1.3]{Bible}.

 It remains to verify that for the \textit{special} \spinc structure $\s$ given by Theorem \ref{T4.1} $(2)$, we have $
\HM_*(Y,\s; \NR_\omega)\neq \{0\}$. Combined with \eqref{E4.9}, this will complete the proof of Proposition \ref{P4.5}. 

\smallskip

The rest of the argument is completely formal and is inspired by \cite[Section 32.3]{Bible}. The group $\HM_*(Y,\s; \NR_\omega)$ can be understood using a smaller local coefficient system $\Pi_\omega$, whose fiber at any point in $\CB(Y,\s)$ is the group ring $R\colonequals \Z[\Z]$ and whose monodromy is again \eqref{E4.6} (note that the image of $f$ lies in a $\Z$-subgroup in $\R$). Let $(\widehat{C}_*,\hat{\partial})$ be the resulting chain complex defined over $R$. There are a few other base rings that we will look at: 
\[
\begin{tikzcd}
R=\Z[\Z] \arrow[r]\arrow[rd] & \R[\Z] \arrow[r,"g_\alpha"]&\R\\
& \BF_2[\Z]\arrow[r] & \K \arrow[r] & \NR.  
\end{tikzcd}
\]
where $\K$ is the rational field of $\BF_2[\Z]$ and for any $\alpha\in \R$, $g_\alpha$ is the ring homomorphism sending the generator $t\in\Z$ to $e^{\alpha}\in \R$. Note that both $\BF_2[\Z]$ and $\R[\Z]$ are principal ideal domains. Moreover, $\K\to \NR$ is a field extension. The goal is to show that $H(\widehat{C}_*\otimes_R \BF_2[\Z], \hat{\partial})$ contains at least one free summand, so that $\HM_*(Y,\s ;\NR_\omega)=H(\widehat{C}_*\otimes_R \NR, \hat{\partial})\neq \{0\}$.

Consider a finite presentation of $H(\widehat{C}_*)$ over the Noetherian ring $R$:
\begin{equation}\label{E4.10}
R^n\xrightarrow{A} R^k\to H(\widehat{C}_*)\to 0, 
\end{equation}
and let $I\subset R$ be the ideal generated by all $k\times k$ minors of $A$, called the Fitting ideal of $H(\widehat{C}_*)$. By \cite[Corollary 20.4]{E95}, $I$ is independent of the resolution that we choose.

Suppose on the contrary that the finitely generated $\BF_2[\Z]$-module $H(\widehat{C}_*\otimes_R \BF_2[\Z], \hat{\partial})$ is torsion. Then we claim that $I\neq \{0\}$; otherwise, the extension $I_{\BF_2}$ of $I$ in $\BF_2[\Z]$ is trivial. $I_{\BF_2}$ is also the Fitting ideal of $H(\widehat{C}_*)\otimes_\Z \BF_2$, which can be seen by taking the tensor product of \eqref{E4.10} with $\BF_2$ (over $\Z$). As a result, $H(\widehat{C}_*)\otimes_\Z \BF_2$ contains a free $\BF_2[\Z]$-summand. The universal coefficient theorem \cite[Theorem 3A.3]{Hatcher} then provides an injection:
\[
0\to H(\widehat{C}_*)\otimes_\Z \BF_2\to H(\widehat{C}_*\otimes_\Z\BF_2, \hat{\partial}),
\]
which contradicts the assumption that we start with. 

Now consider the tensor product of $\eqref{E4.10}$ with $\R$. The Fitting ideal $I_\R$ of $H(\hat{C}_*)\otimes_\Z \R$ is then the extension of $I$ in $\R[\Z]$. Since $I\neq \{0\}$, $I_\R\neq \{0\}$. Using the universal coefficient theorem, we conclude that 
\[
H(\hat{C}_*)\otimes_\Z \R\cong H(\hat{C}_*\otimes_\Z \R)
\]
is a torsion $\R[\Z]$-module and is therefore annihilated by some polynomial $u(t)\in \R[\Z]$.

On the one hand, one may pick $\alpha\in \R$ such that $g_\alpha(u(t))=u(e^{\alpha})\neq 0\in \R$; as a result, the group $H(\widehat{C}_*\otimes_{g_\alpha} \R, \hat{\partial})$ is trivial for this particular ring homomorphism $g_\alpha: \R[\Z]\to \R$.  

On the other hand, the local coefficient system $\Pi_\omega\otimes_{g_\alpha} \R$ agrees with $\Gamma_\xi$ for some 1-cycle $\xi$ with $[\xi]\neq 0\in H_1(Y;\Z)$, since $[\omega]$ is non-balanced. Theorem \ref{T4.1} (2) then asserts that $\HM_*(Y,\s; \Gamma_\xi)$ is non-trivial. A contradiction.\qed

\subsection{Computation for Product 3-manifolds} For future reference, we recall a classical result for the 3-manifold $\Sigma_g\times S^1$, where $\Sigma_g$ is a closed surface of genus $g\geq 2$.


\begin{lemma}\label{L6.9} Following the shorthands from Definition \ref{D6.4}, for any closed 2-form $\omega\in \Omega^2(\Sigma_g\times S^1, i\R)$ with
	\begin{equation}\label{E4.11}
	i\langle [\omega], [\Sigma_g]\rangle<2\pi (g-1),
	\end{equation}
	we have 
	$
\HM(\Sigma_g\times S^1,[\omega]|\Sigma_g )\cong \NR.
	$
\end{lemma}
\begin{proof}[Proof of Lemma \ref{L6.9}] Let $\s$ be any \spinc structures contributing to the group $\HM(\Sigma_g\times S^1,[\omega]|\Sigma_g )$. Under the assumption \eqref{E4.11}, the period class $c=-2\pi i[\omega]$ is neither negatively monotone nor balanced with respect to $\s$; indeed,
	\[
	\langle 2\pi^2 c_1(\s)+c, [\Sigma_g]\rangle=4\pi^2(g-1)-2\pi i\langle [\omega],[\Sigma_g]\rangle >0.
	\] 
	Using Proposition \ref{P4.2}, we reduce the computation to the case when the perturbation is exact, but the coefficient system is still $\NR_\omega$. Now we use \cite[Lemma 2.2]{KS} to conclude the proof.
\end{proof}

\section{Thurston Norm Detection}\label{Sec7}

From now on, we will always take $Y$ to be an oriented 3-manifold with toroidal boundary.

\begin{definition}\label{D7.1}For any $T$-surface $\TSigma=(\Sigma, g_\Sigma,\lambda,\mu)$ and any 1-cell $\y\in \AT(\emptyset,\TSigma)$, consider the set of  monopole classes:
	\[
	\sw(\y)\colonequals \{c_1(\bs): \HM_*(\y,\bs)\neq\{0\}\} \subset H^2(Y,\partial Y; \Z),
	\]
	and for any $\kappa\in H_2(Y,\partial Y;\Z)$, define 
	\[
	\varphi_{\y}(\kappa)\colonequals \max_{c_1(\bs)\in \sw(\y)} \langle c_1(\bs), \kappa\rangle,\ \kappa\in H_2(Y,\partial Y;\Z). 
	\]
	Our convention here is that $\varphi_{\y}\equiv -\infty$, if $\sw(\hy)=\emptyset$. 
\end{definition}

When $Y$ is connected and irreducible, it is tempting to generalize Theorem \ref{T4.1} and relate $\sw(\y)$ with the Thurston norm on $H_2(Y,\partial Y;\R)$. However, the author was not able to show that $\sw(\y)$ is symmetric about the origin; so only a weaker statement is obtained in this paper.

\begin{theorem}\label{T7.1} For any $T$-surface $\TSigma=(\Sigma, g_\Sigma,\lambda,\mu)$, let $\y\in \AT(\emptyset,\TSigma)$ be any 1-cell with $Y$ connected and irreducible. Then the set of monopole classes $\sw(\y)$ is non-empty and determines the Thurston norm on $H_2(Y,\partial Y;\R)$ in the following sense:
	\[
	x(\kappa)=\half ( \varphi_{\y}(\kappa)+\varphi_{\y}(-\kappa)),\  \forall \kappa\in H_2(Y,\partial Y; \Z). 
	\]
	In general, we only have an inequality $\varphi_{\y}(\kappa)+\varphi_{\y}(-\kappa)\leq 2x(\kappa)$.
\end{theorem}

One may approach the more desirable statement
\[
x(\kappa)=\varphi_{\y}(\kappa)=\varphi_{\y}(-\kappa).
\]
by either proving the symmetry of $\sw(\y)$ or the adjunction inequality:
\[
\varphi_{\y}(\kappa)\leq x(\kappa).
\]
But the author was unable to verify either of them directly. This Thurston norm detection result is accompanied with a fiberness detection result: 
\begin{theorem}\label{T7.2} For any $T$-surface $\TSigma=(\Sigma, g_\Sigma,\lambda,\mu)$, let $\y\in \AT(\emptyset,\TSigma)$ be any 1-cell with $Y$ connected and irreducible, and $\kappa\in H_2(Y,\partial Y;\Z)$ be any integral class. Consider the subgroup
	\[
	\HM_*(\y|\kappa)\colonequals \bigoplus_{\langle c_1(\bs), \kappa\rangle=\varphi_{\y}(\kappa)} \HM_*(\y,\bs).
	\]
	If $\rank_\NR 	\HM_*(\y|\kappa)=\rank_\NR 	\HM_*(\y|-\kappa)=1$, then $\kappa$ can be represented by a Thurston norm minimizing surface $F$ and $Y$ fibers over $S^1$ with $F$ as fiber. 
\end{theorem}

The rest of this section is devoted to the proof of Theorem \ref{T7.1}, while the proof of Theorem \ref{T7.2} is deferred to Section \ref{Sec9}, after a digression into sutured monopole Floer homology in Section \ref{Sec8}.

\begin{proof}[Proof of Theorem \ref{T7.1}] We focus on the case when $Y$ is irreducible. The strategy is to apply the Gluing Theorem \ref{T2.5} and reduce the problem to the double of $Y$:
	\[
	\tilde{Y}\colonequals Y\ \bigcup_{\Sigma}\ (-Y).
	\]
	Since $Y$ is irreducible, so is $\tilde{Y}$. Then we can conclude using Proposition \ref{P4.2} and \ref{P4.5}.
	
	For any 1-cell $\y=(Y,g_Y,\omega,\cdots)$, consider its orientation reversal:
	\[
	(-\y)=(-Y,g_Y,\omega, \cdots)\in \AT(\TSigma',\emptyset).
	\]
	The problem here is that the $T$-surface $\TSigma'=(\Sigma, g_\Sigma, -\lambda,\mu)$ is different from $\TSigma$ in that the sign of $\lambda$ is reversed. We can not form the horizontal composition $\y\circ_h (-\y)$. 
	
This problem is circumvented by changing the 2-form $\omega$ on $(-Y)$. Since $\HM_*(\y,\bs)$ is independent of the metric $g_Y$, we assume that it is cylindrical on $(-2,1]_s\times\Sigma\subset Y$. Pick a cut-off function $\chi_1: (-2,1]_s\to \R$ such that $\chi_1\equiv 0$ if $s\leq -3/2$ and $\equiv 1$ if $s\geq-1$. Write $
\omega=\bomega+\chi_1(s)ds\wedge\lambda$ and set 
\[
\y'\colonequals (Y,g_Y,\omega',\cdots )\in \AT(\emptyset,\TSigma') \text{ with }\omega'\colonequals\bomega-\chi_1(s)ds\wedge\lambda. 
\]
We shall work instead with the orientation reversal $(-\y')\in \AT(\TSigma, \emptyset)$. At this point, we need a lemma relating the Floer homology of $\y'$ with that of $\y$:
\begin{lemma}\label{L7.3} For any $T$-surface $\TSigma$, let $\y\in \AT(\emptyset,\TSigma)$ be any 1-cell and $\bs\in \Spincr(Y)$ be any relative \spinc structure. Then we have the following isomorphisms:
\begin{enumerate}[label=(\arabic*)]
\item\label{7.2.1} (Poincar\'{e} Duality) $ \HM_*((-\y),\bs)\cong \HM^*(\y,\bs)$.
\item\label{7.2.2} (Reversing the sign of $\lambda$) $\HM_*(\y,\bs)\cong \HM_*(\y',\bs)$.
\end{enumerate}
\end{lemma}

Let us finish the proof of Theorem \ref{T7.1} assuming Lemma \ref{L7.3}. Consider the horizontal composition of $\y$ and $-\y'$:
\[
(\tilde{Y},\tilde{\omega},\cdots)\colonequals \y\circ_h (-\y').
\]
We verify that the closed 2-form $\tilde{\omega}$ is neither negatively monotone nor balanced with respect to any \spinc structures on the double $\tilde{Y}$. Indeed, by \ref{T3} and \ref{T1}, $|\langle \tilde{\omega},[\Sigma^{(i)}]\rangle|<2\pi$ and $\neq 0$ for any component $\Sigma^{(i)}\subset \Sigma$.

Any integral class $\kappa\in H_2(Y, \partial Y;\Z)$ can be represented by a properly embedded oriented surface $F\subset Y$ minimizing the Thurston norm. Combined with its orientation reversal $(-F)\subset (-Y)$, they form a closed surface 
\[
\tilde{F}\colonequals F\cup (-F)\subset \tilde{Y}.
\]
Since $Y$ is irreducible and $Y\neq S^1\times D^2$, the surface $F$ has no sphere or disk components. By \cite[Lemma 6.15]{Gabai83}, the double $\tilde{F}\subset \tilde{Y}$ is also norm-minimizing.
 By Theorem \ref{T3.5}, we have 
 \begin{equation}\label{E7.1}
 \HM_*(\tilde{Y},\s,\tilde{c}; \NR_{\tilde{\omega}})\cong \bigoplus_{ \bs_1\circ_h \bs_2=\s} \HM_*(\y,\bs_1)\otimes_\NR \HM_*(-\y',\bs_2). 
 \end{equation}
 where the sum is over all pairs $(\bs_1,\bs_2)\in \Spincr(Y)\times \Spincr(-Y)$ with $\bs_1\circ_h \bs_2=\s$. Note that 
  \begin{equation}\label{E7.3}
 \langle c_1(\bs_1\circ_h\bs_2), [\tilde{F}]\rangle =\langle c_1(\bs_1), [F]\rangle +\langle c_1(\bs_2), [-F]\rangle,
 \end{equation}
 By the adjunction inequality from Proposition \ref{P4.2}, the left hand side of \eqref{E7.1} vanishes whenever
 \[
 \langle c_1(\s), [\tilde{F}]\rangle >x(\tilde{F})=2x(F). 
 \]
 Combined with Corollary \ref{C6.5} and \eqref{E7.1}\eqref{E7.3}, this implies that the group
\begin{equation}\label{E7.2}
\HM(\tilde{Y},[\tilde{\omega}]|[\tilde{F}])\cong \HM_*(\y|\kappa)\otimes_\NR \HM_*(-\y'|-\kappa),
\end{equation}
 is non-vanishing. By Lemma \ref{L7.3}, we have 
 \begin{equation}
\rank_\NR \HM_*(-\y'|-\kappa)= \rank_\NR\HM_*(\y|-\kappa),
 \end{equation}
so
 \[
2x(\kappa)=x(\tilde{F})=\varphi_{\y}(\kappa)+\varphi_{(-\y')}(-\kappa)=\varphi_{\y}(\kappa)+\varphi_{\y}(-\kappa).
 \]
This completes the proof of Theorem \ref{T7.1} when $Y$ is irreducible. In the general case, the inequality 
\[
\varphi_{\y}(\kappa)+\varphi_{\y}(-\kappa)\leq 2x(\kappa),
\]
follows from the vanishing result that we used earlier. 
\end{proof}

\begin{proof}[Proof of Lemma \ref{L7.3}] The first isomorphism \ref{7.2.1} is due to Poincar\'{e} duality. Changing the orientation of $\hy$ has the same effect as changing the sign of the perturbed Chern-Simons-Dirac functional $\CL_\omega$ in \eqref{E3.3}, while keeping the orientation fixed; see \cite[Section 22.5]{Bible} for more details. Since we have worked with a ring $\NR$ over $\BF_2$, there is no need to deal with orientations. Our Floer homology $\HM_*(\y)$ is defined using a local coefficient system with fibers $\NR$. Thus the most relevant analogue for closed 3-manifolds is \cite[P.624 (32.2)]{Bible}.
	
	\smallskip
	
The second isomorphism \ref{7.2.2} is induced from a concrete strict cobordism
\[
\x:\y\coprod e_{\TSigma}\to \y',
\]
as we describe now. Similar to the construction in Section \ref{Sec3}, we start with a hexagon $\Omega_1$ with boundary consisting of geodesic segments of length $2$ and whose internal angles are always $\pi/2$. Also, the metric of $\Omega_1$ is flat near the boundary.
	\begin{figure}[H]
	\centering
	\begin{overpic}[scale=.09]{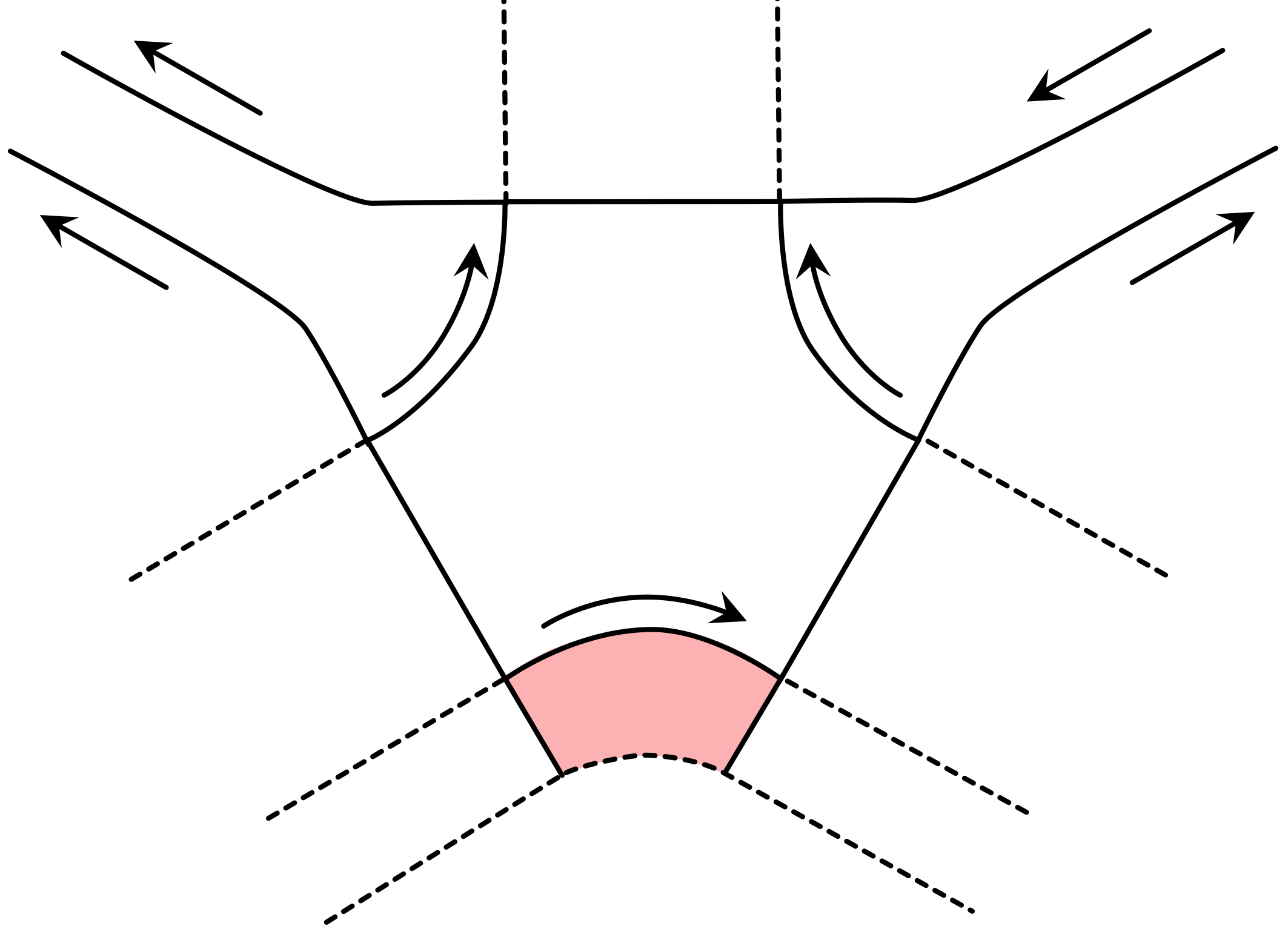}
\put(37,41){\small$\gamma_{12}$}
\put(56,41){\small$\gamma_{23}$}
\put(53,28){\small$t_{13}$}
\put(27,21){\small$\hy$}
\put(68,21){\small$\hy$}
\put(41,58){\small$\R_{s_2}\times\Sigma$}
\put(5,46){\small$s_1$}
\put(94,47){\small$s_3$}
\put(15,69){\small$s_2$}
\put(80,69){\small$s_2$}
\put(46,18){\small $\gamma_{13}$}
\put(27,46){\small$t_{12}$}
\put(67,46){\small$t_{32}$}
\put(48,46){$\Omega_1$}
	\end{overpic}	
	\medskip
	\caption{}
	\label{Pic12} 
\end{figure}

The desired  cobordism $X$ is then obtained from $\Omega_1\times\Sigma$ by attaching $[-1,1]_t\times Y_-$ to $\gamma_{13}\times \Sigma$ in Figure \ref{Pic12}, where $Y_-\colonequals \{s\leq -1\}\subset \hy$. In this case,  we use $s_1,s_2,s_3$ for spatial coordinates and $t_{12}, t_{13},t_{32}$ for time coordinates; their relations are indicated as in Figure \ref{Pic12}.

We have to specify the closed 2-form $\omega_X$ on $X$ in order to construct the cobordism map. Recall that the 2-form $\omega$ on $Y$ is decomposed as 
\[
\omega=\bomega+\chi_1(s)ds\wedge\lambda. 
\] 
First, pull back $\bomega$ to $[-1,1]_t\times Y_-$ and extend it on $\Omega_1\times\Sigma$ constantly as the harmonic 2-form $\mu\in \Omega^1_h(\Sigma, i\R)$. Second, consider a function $f: (-2,\infty)_s\to \R$ such that $f(s)\equiv s$ when $s\geq -1$ and $f'(s)=\chi_1(s)$ when $s\in (-2,1]$. We shall regard $f$ as a function on $\hy$ by extending $f$ constantly when $s\leq -3/2$. Now consider a function $h: \hx\to \R$ with the following properties:
\begin{itemize}
\item $h\equiv f$ on a neighborhood $[-1,-1+\epsilon)_{t_{13}}\times \hy$ and $\equiv -f$ on $(1-\epsilon, 1]_{t_{13}}\times \hy$;
\item $h\equiv s_1\equiv s_2$ on $\gamma_{12}\times [1,\infty)_{s_1}\times\Sigma$ and $\equiv -s_3\equiv s_2$ on $\gamma_{23}\times [1,\infty)_{s_3}\times \Sigma$;
\item $h\equiv s_2$ on a neighborhood $(1-\epsilon,1]_{t_{12}}\times \R_{s_2}\times \Sigma$. 
\end{itemize}

The extension of $h$ over the interior of $X$ can be arbitrary. Finally, set
\[
\omega_X=\bomega+dh\wedge\lambda \text{ on } \hx. 
\]
The strict cobordism $\x=(X,\omega_X)$ then induces a map:
\[
\alpha_1:\HM_*(\y,\bs)\xrightarrow{\Id\otimes 1}\HM_*(\y,\bs)\otimes \HM_*(e_\TSigma, \bs_{std})\xrightarrow{\HM_*(\x)} \HM_*(\y',\bs). 
\]
Switching the roles of $\y$ and $\y'$ produces a strict cobordism $\x': \y'\coprod e_{\TSigma}\to \y$ and a map in the opposite direction:
\[
\alpha_1':\HM_*(\y',\bs)\xrightarrow{\Id\otimes 1}\HM_*(\y',\bs)\otimes \HM_*(e_{\TSigma'}, \bs_{std})\xrightarrow{\HM_*(\x')} \HM_*(\y,\bs).
\]
It remains to verify that $\alpha_1$ and $\alpha_1'$ are mutual inverses to each other (up to a non-zero scalar). To see this, compose $\x$ with $\x'$ along the common boundary $\y'$. The resulting cobordism
\[
\y\ \coprod\  e_{\TSigma}\ \coprod\ e_{\TSigma'}\to \y
\]
induces the gluing map $\alpha$ in Subsection \ref{Subsec3.3}, which is an isomorphism by Theorem \ref{T3.3}. Indeed, one can change the 2-form $\omega_X'\circ \omega_X$ into the standard one in \eqref{E3.1} by adding an exact 2-form. Thus the map $\alpha_1'\circ\alpha_1: \HM_*(\y,\bs)\to \HM_*(\y,\bs)$ is invertible. The other composition $\alpha_1\circ \alpha_1'$ is dealt with in a similar manner. This completes the proof of Lemma \ref{L7.3}.
\end{proof}

\section{Relations with Sutured Floer Homology}\label{Sec8}

Before we prove the fiberness detection result, Theorem \ref{T7.2}, we add a digression to explain the relations of $\HM_*(\y)$ with the sutured Floer homology $\SHM$ \cite{KS} and $\SFH$ \cite{J06,J08}. In the original definition of Gabai \cite[Definition 2.6]{Gabai83}, any 3-manifolds with toroidal boundary are sutured manifolds with only toral components. However, they are not examples of balanced sutured manifolds in the sense of Juh\'{a}sz \cite[Definition 2.2]{J06}; thus the sutured Floer homology, either $\SHM$ or $\SFH$, is never defined for this class of sutured manifolds. 

One may regard our construction as a natural extension of the existing sutured Floer theory, and ask if the suture manifold decomposition theorem, e.g.  \cite[Theorem 1.3]{J08} and \cite[Proposition 6.9]{KS}, continue to hold in our case. In this section, we will only prove a preliminary result towards this direction. Since our Floer homology $\HM_*(\y)$ also relies on the closed 2-form $\omega$, it is not clear to the author whether a general result is available. 

It is worth mentioning that the sutured monopole Floer homology $\SHM$ introduced by Kronheimer-Mrowka \cite{KS} is defined over $\Z$, but the construction extends naturally to the mod 2 Novikov ring $\NR$, as explained in \cite[Section 2.2]{Sivek12}. We shall work with the latter case.

\subsection{Cutting along a surface}
For any $T$-surface $\TSigma=(\Sigma, g_{\Sigma},\lambda,\mu)$ and any 1-cell $\y\in \AT(\emptyset, \TSigma)$, we listed a few cohomological conditions on $\omega$ in \ref{P2} \ref{P3}. In this section, we shall think of them geometrically and work with a more restrictive setup:
\begin{enumerate}[label=(P\arabic*)]
	\setcounter{enumi}{4}
\item\label{P5} there exists a properly embedded oriented surface $F\subset Y$ such that $\partial F$ intersects each component of $\Sigma$ in parallel circles, and $i[*_2\lambda]$ is dual to $z[\partial F]\in H_1(\Sigma, \Z)$ for some $z\in \R$; moreover, $F$ has no closed component or disk components. In particular, $\chi(F)\leq 0$.
\item\label{P6} the Poincar\'{e} dual of $-2\pi i[\omega]\in H^2(Y; \R)$ can be represented by a real 1-cycle $\eta$ that lies on the surface $F$. 
\end{enumerate} 

If $\y$ satisfies the additional properties \ref{P5}\ref{P6}, then we can cut $Y$ along the surface $F$ to obtain a balanced sutured manifold, denoted by $M(Y,F)$. Let
\[
\SHM(M(Y,F))
\]
be the sutured monopole Floer homology of $M(Y, F)$ defined over $\NR$ with trivial coefficient system; cf. \cite[Section 2.2]{Sivek12}.
\begin{theorem}\label{T8.1} For any $T$-surface $\TSigma$ and any 1-cell $\y\in \AT(\emptyset, \TSigma)$ satisfying \ref{P5}\ref{P6}, we have 
	\[
	\SHM(M(Y,F))\cong \bigoplus_{\langle c_1(\bs),[F]\rangle =x(F)} \HM_*(\y,\bs).
	\]
	Moreover, $\varphi_{\y}([F])\leq x(F)$ with $\varphi_{\y}$ defined as in Definition \ref{D7.1}.
\end{theorem}

For the proof of Theorem \ref{T8.1}, we have to first understand the special case when $F$ is connected and $Y$ fibers over $S^1$ with $F$ as a fiber. Thus $F$ is a genus $g$ surface with $n$ boundary circles, and $Y$ is the mapping torus of a self-diffeomorphism $\phi: F\to F$ such that $\phi|_{\pi_0(\partial F)}$ has at least two orbits.

\begin{lemma}\label{L8.2} For any $T$-surface $\TSigma$ and any 1-cell $\y\in \AT(\emptyset, \TSigma)$, if $F$ is connected and $Y$ is a mapping torus over $F$, then for $\kappa=\pm [F]$, $\varphi_{\y}(\kappa)=-\chi(F)$ and 
	\[
\HM(\y|\kappa)\cong \NR.
	\]
\end{lemma}
\begin{proof}[Proof of Lemma \ref{L8.2}] If $\chi(F)=0$, then $F$ is an annulus and $Y=[-3,3]_s\times\T^2$ is a product. This case is addressed already in Lemma \ref{L5.2}. We focus on the case when $\chi(F)<0$.

With loss of generality, let $\kappa=[F]$. Note that $Y$ is irreducible and $F$ minimizes the Thurston norm. We follow the notations and arguments in the proof of Theorem \ref{T7.1}. Let $(\tilde{Y},\tilde{F})$ denote the double of $(Y,F)$, then $\tilde{Y}$ is a mapping torus over $\tilde{F}$ with $g(\tilde{F})\geq 2$. By \eqref{E7.2}, we have 
	\begin{equation}\label{E8.1}
\HM(\tilde{Y},[\tilde{\omega}]|\tilde{F})=\HM(\y|\kappa)\otimes_\NR \HM(-\y'| -\kappa). 
\end{equation}

The rest of the proof is divided into four steps.

\medskip

\Step 1. We can arrange the 2-form $\tilde{\omega}$ so that $\langle [\tilde{\omega}], [\tilde{F}]\rangle=0$.

Instead of $\y\circ_h(-\y')$, we consider the horizontal composition
\[
(\tilde{Y},\tilde{\omega},\cdots)=\y\circ_h \y_1\circ_h (-\y'),
\]
where $\y_1=([-3,3]_s\times \Sigma, \omega_1,\cdots)\in \AT(\TSigma,\TSigma)$. Here $\omega_1=\mu+ds\wedge\lambda+\omega_2$ and $\omega_2\in \Omega^2_c(Y_1, i\R)$ is compactly supported. By Lemma \ref{L5.2}, inserting this extra 1-cell $\y_1$ does not affect the identity \eqref{E8.1}, but changing the class $[\omega_2]\in H^2(Y_1,\partial Y_1; i\R)$ can effectively alter the pairing $\langle [\tilde{\omega}], [\tilde{F}]\rangle$. From now on, we shall always assume that $\langle [\tilde{\omega}], [\tilde{F}]\rangle=0$.

\medskip

\Step 2. The group $\HM(\tilde{Y},[\tilde{\omega}]|\tilde{F})$ in \eqref{E8.1} has rank $1$. 

Since $g(\tilde{F})\geq 2$, this computation for mapping tori is \cite[Lemma 4.7]{KS}, if $\tilde{\omega}=0$. The general case is not really different. Since $\langle [\tilde{\omega}], [\tilde{F}]\rangle=0$, we can still apply Floer's excision theorem \cite[Theorem 3.1]{KS} to reduce the problem to the case when $\tilde{Y}=\tilde{F}\times S^1$ is a product. The same trick is used also in the proof of Theorem \ref{T8.1} below. Now the statement follows from Lemma \ref{L6.9}. 

\medskip

\Step 3. We conclude from \eqref{E8.1}, Lemma \ref{L7.3} and \Step 2 that
\[
\rank_{\NR} \HM(\y|\kappa)=\rank_{\NR} \HM(\y|-\kappa)=1. 
\]

\medskip

\Step 4. By Theorem \ref{T7.1}, $\varphi_{\y}(\kappa)+\varphi_{\y}(-\kappa)=2x(\kappa)$. We have to verify that $\varphi_{\y}(\kappa)=\varphi_{\y}(-\kappa)$. This equality now follows from the symmetry of the graded Euler characteristics:
\[
\SW(Y):\bs\mapsto \chi(\HM_*(\y,\bs));
\]
see Theorem \ref{T3.4}. This function is invariant under the conjugacy of relative \spinc structures: $\bs\leftrightarrow \bs^*$ by  \cite{MT96,Taubes01}. The computation of $\HM(\y|\kappa)$ shows that $\chi(\HM_*(\y,\bs))\neq 0$ for exactly one $\bs$ with $\langle c_1(\bs),\kappa\rangle\geq \varphi_{\y}(\kappa)$. This proves the equality $\varphi_{\y}(\kappa)=\varphi_{\y}(-\kappa)$ and completes the proof of Lemma \ref{L8.2}.
\end{proof}

\begin{proof}[Proof of Theorem \ref{T8.1}] In \cite{KS}, the group $\SHM(M(Y,F))$ is defined as the monopole Floer homology of a suitable closure of $M(Y,F)$, which can be described as follows. 
	
Consider an 1-cell $\y_1\in \AT(\TSigma, \emptyset)$ such that $Y_1$ is a mapping torus over a connected surface $F_1$. We require that 
\begin{itemize}
\item $g(F_1)\geq 2$;
\item $[\partial F_1]=-[\partial F]\in H_1(\Sigma, \Z)$, and 
\item the 1-cell $\y_1$ satisfies property \ref{P5}\ref{P6} for the surface $F_1$. 
\end{itemize}

Consider the closed 3-manifold obtained by gluing $\y$ and $\y_1$:
\[
(Y_2, \omega_2,\cdots)\colonequals \y\circ_h\y_1,
\]
We can arrange so that $\partial F$ is identical to $\partial F_1$ on $\Sigma$; so $Y_2$ contains a closed oriented surface $F_2\colonequals F\cup F_1$ with $g(F_2)\geq 2$. Moreover, $F_2$ is connected. Following the shorthands from Definition \ref{D6.4}, the sutured monopole Floer homology of $M(Y,K)$ is then defined as 
\[
\SHM(M(Y, K))\colonequals  \HM_*(Y_2|F_2);
\]
see \cite[Definition 4.3]{KS} and \cite[Section 2.2]{Sivek12}. For the latter group, we have 
\[
\HM_*(Y_2|F_2)\cong \HM_*(Y_2,[0]|F_2),
\]
 by Remark \ref{R4.5}. Let $c_2=-2\pi i[\omega_2]$ be the period class of $\omega_2\in \Omega^2(Y_2, i\R)$. Using Property \ref{P6}, we can arrange so that the Poincar\'{e} dual of $c_2$ is represented by a real 1-cycle $\eta_2$ lying over $F_2\subset Y_2$. Now consider the subgroup
\[
G_y\colonequals \bigoplus_{\langle c_1(\bs), [F]\rangle = y} \HM_*(\y,\bs)\subset \HM_*(\y).
\]
Theorem \ref{T3.5} and Lemma \ref{L8.2} then imply that 
\begin{equation}\label{E8.2}
\bigoplus_{\langle c_1(\s), [F_2]\rangle =y+x(F_1)} \HM_*(Y_2, \s,c_2; \NR_{\omega_2})= G_b\otimes_\NR \HM(\y_1|[F_1])\cong G_b,
\end{equation}
for any $y\geq \varphi_{\y}([F])$. By the adjunction inequality from Proposition \ref{P4.2}, the left hand side of \eqref{E8.2} vanishes whenever $b+x(F_1)>x(F_2)$. As a result,
\[
\varphi_{\y}([F])\leq x(F)=-\chi(F).
\]
Let $y=x(F)$ in \eqref{E8.2}, then
\[
\HM(Y_2,[\omega_2]|F_2)\cong G_{x(F)}.
\]

 To complete the proof of Theorem \ref{T8.1}, it remains to verify that
 \begin{equation}\label{E8.3}
\HM(Y_2,[0]| F_2)\cong \HM(Y_2, [\omega_2]|F_2).
 \end{equation}
This isomorphism, which involves only the closed 3-manifold $Y_2$, is similar to the one in \cite[Corollary 3.4]{KS}, except that $\omega_2$ is used for non-exact perturbations here. The proof of \eqref{E8.3} relies on the property that the real 1-cycle $\eta_2$ that represents $c_2$ lies on the surface $F_2$, so we can pick $\omega_2'\in [\omega_2]$ such that $\omega_2'$ is supported on a tubular neighborhood of $F_2$:
\[
[-1,1]\times F_2\subset Y_2. 
\]
By identifying $\{\pm 1\}\times F_2$, $\omega_2'$ becomes a closed 2-form on $F_2\times S^1$.  The same process applied to $Y_2\setminus [-1,1]\times F_2$ yields another copy of $Y_2$, but the 2-form is zero now. 

\smallskip

 As in the proof of \cite[Corollary 3.4]{KS}, we apply Floer's excision theorem \cite[Theorem 3.1]{KS} to obtain that 
\[
\HM(Y_2, [\omega_2]|F_2)=\HM(Y_2,[0]|F_2)\otimes_\NR \HM(F_2\times S^1,[\omega_2']|F_2). 
\]
By Lemma \ref{L6.9}, $\HM( F_2\times S^1,[\omega_2']|F_2)\cong \NR$. This completes the proof of Theorem \ref{T8.1}.
\end{proof}

The proof of Theorem \ref{T8.1} has an immediate corollary:
\begin{corollary}\label{C8.4} For any $T$-surface $\TSigma$ and any 1-cell $\y\in \AT(\emptyset, \TSigma)$ satisfying \ref{P5}, $\varphi_{\y}([F])\leq x(F)$.
\end{corollary}

\begin{remark} The property \ref{P5} is crucial to the proof of Theorem \ref{T8.1} for the following reason: for the gluing argument to work, we have to make sure that 
\begin{enumerate}[label=(\roman*)]
\item\label{i} $[\partial F]=-[\partial F_1]\in H_1(\Sigma;\Z)$;
\item\label{ii} there exists classes $a\in H_2(Y,\partial Y;\R)$ and $a_1\in H_2(Y_1, \partial Y_1;\R)$ such that 
\[
[\partial a]=-[\partial a_1]=\text{the Poincar\'{e} dual of }i[*_\Sigma\lambda]\in H_1(\Sigma;\R);
\]
\end{enumerate}
In general, it is not clear to the author whether $Y$ and $Y_1$ can be always glued. However, Property \ref{P5} reduces \ref{i}\ref{ii} into a single condition, which is easier to verify in practice.
\end{remark}

\subsection{Relations with Link Floer Homology} As a special case of Theorem \ref{T8.1}, our construction recovers the monopole knot Floer homology $\KHM$ introduced by Kronheimer and Mrowka \cite{KS}. This statement is also true for the link Floer homology; let us expand on its meaning.

For any link $L=\{L_i\}_{i=1}^n$ inside a closed oriented 3-manifold $Z_0$, consider the balanced sutured manifold 
\[
Z_0(L)\colonequals (Z_0\setminus N(L), \gamma)
\]
where $s(\gamma)\cap N(L_i)$ are two meridional sutures on $\partial N(L_i)$ oriented in opposite ways. The link Floer homology of $(Z_0, L)$ is defined as the sutured monopole Floer homology of $Z_0(L)$:
\[
\LHM(Z_0, L)\colonequals \SHM( Z_0(L)),
\]
and we shall work with the mod 2 Novikov ring $\NR$. 

On the other hand, pick a meridian $m_i$ for each link component $L_i$ and consider the link complement
\[
Y(Z_0, L)\colonequals Z_0\setminus N(L\cup m_1\cup m_2\cup\cdots \cup m_n). 
\]
Each $m_i$ bounds a disk in $Z_0$, which becomes an annulus $A_i$ in $Y(Z_0, L)$. Let $F$ be the union $\bigcup_{i=1}^n A_i$ (with any fixed orientation). The balanced sutured manifold $Z_0(L)$ is then obtained from $Y(Z_0, L)$ by cutting along $F$. 

To apply Theorem \ref{T8.1}, we have to specify the choice of $\omega$:
\begin{itemize}
\item let $\Sigma=\partial Y(Z_0, L)$. Pick a flat metric $g_\Sigma$ and $\lambda\in \Omega^1_h(\Sigma, i\R)$ such that $i[*_\Sigma\lambda]$ is dual to $z[\partial F]\in H_1(\Sigma, \R)$ for some $z\neq 0\in \R$;
\item $\eta$ is a real 1-cycle on $F$ such that $\eta\cap A_i$ is a segment joining two components of $\partial A_i$. Let $-2\pi i[\omega]\in H^2(Y(Z_0, L), \R)$ be the dual of $\eta$. 
\end{itemize}

Finally, let $\y(Z_0, L)=(Y(Z_0, L), \omega,\cdots )\in \AT(\emptyset, \TSigma)$ be any 1-cell satisfying these properties. As a corollary of Theorem \ref{T8.1}, we have

\begin{corollary}\label{C8.3} For any 1-cell $\y(Z_0,L)$ constructed above, we have an isomorphism
	\[
	\HM_*(\y(Z_0, L))\cong \SHM(Z_0(L))=\LHM(Z_0, L). 	
	\]
	Moreover, if $	\HM_*(\y(Z_0, L),\bs)\neq \{0\}$, then $\langle c_1(\bs), [A_i]\rangle=0$ for any $1\leq i\leq n$. 
\end{corollary}
\begin{proof}[Proof of Corollary \ref{C8.3}] If $\HM_*(\y(Z_0, L),\bs)\neq \{0\}$, then Theorem \ref{T8.1} implies that 
	\[
	\langle c_1(\bs), [F] \rangle\leq \varphi_{\y}([F])\leq x(F)=0. 
	\]
	Applying the same argument for $[-F]$, we conclude that $	\langle c_1(\bs), [F] \rangle=0$. The desired isomorphism then follows from the first part of Theorem \ref{T8.1}.
	
	It remains to verify the stronger statement: $\langle c_1(\bs), [A_i] \rangle=0$. In the proof of Theorem \ref{T8.1}, we composed $\y$ with a mapping torus over a connected surface $F_1$. We now make $F_1$ disconnected: take $F_1=(-F)$ and $Y_1=(-F)\times S^1$. When $Y(Z_0, L)$ are glued with $Y_1$, we require that 
	\begin{itemize}
\item $\partial N(L_i\cup m_i)\subset \partial Y(Z_0, L)$ is identified with $\partial (-A_i)\times S^1\subset \partial Y_1$;
\item $\partial A_i\subset\partial Y(Z_0, L)$ is identified with $\partial(-A_i)\times\{pt\}\subset \partial Y_1$.
	\end{itemize}
 The closed 3-manifold $\y\circ_h\y_1$ contains a collection of 2-tori $\T^2_i$, which are doubles of $A_i$'s. Now we use the adjunction inequality in Proposition \ref{P4.2} and Lemma \ref{L5.2} to conclude. 
\end{proof}

\section{Fiberness Detection}\label{Sec9}

\subsection{Some Preparations} In this section, we complete the proof of Theorem \ref{T7.2}. We start with a preliminary result:
\begin{lemma}\label{L9.1} Under the assumption of Theorem \ref{T7.2}, if  $\kappa\in H_2(Y,\partial Y;\Z)$ is primitive, then $\kappa$ can be represented by a connected Thurston norm minimizing surface $F$ that intersects each component of $\Sigma$ non-trivially. Moreover, $\varphi_{\y}(\kappa)=\varphi_{\y}(-\kappa)$.
\end{lemma}
\begin{proof} This lemma follows from \cite[Theorem 4.1 \& Proposition 6.1]{M02}. Let 
	\[
	\phi\in H^1(Y;\Z)=\Hom(\pi_1(Y), \Z)
	\]
	be the Poincar\'{e} dual of $\kappa$ and $b_1(\ker\phi)$ be the first Betti number of the subgroup $\ker \phi\subset \pi_1(Y)$. Then \cite[Proposition 6.1]{M02} states that $\kappa$ can be represented by such a norm-minimizing surface $F$ if 
\begin{equation}\label{E9.1}
b_1(\ker \phi)<\infty. 
\end{equation}
To verify that $F$ intersects each component of $\Sigma$ non-trivially, one has to go through the proof of \cite[Proposition 6.1]{M02}. The condition \eqref{E9.1} will follow from \cite[Theorem 4.1]{M02}, if we can verify its assumptions. Consider the set of Alexander classes
\[
\Delta(Y)\colonequals \{ c_1(\bs): \chi(\HM_*(\y,\bs))\neq 0\}\subset H^2(Y,\partial Y;\Z).
\]
By Theorem \ref{T3.4}, $\Delta(Y)$ is precisely the support of the Alexander polynomial of $Y$ and is symmetric about the origin. Since
\[
\rank_\NR \HM_*(\y|\kappa)=\rank_\NR \HM_*(\y|-\kappa)=1, 
\]
we conclude that $\Delta(Y)\neq \emptyset$ and $\varphi_{\y}(\kappa)=\varphi_{\y}(-\kappa)$. Moreover, the maximum
\[
\max_{a,b\in \Delta(Y)} \phi(a-b)
\]
is achieved for exactly one pair of elements $(a,b)$. The other assumption of \cite[Theorem 4.1]{M02} is certified by \cite[Theorem 5.1]{M02}. This completes the proof of Lemma \ref{L9.1}.
\end{proof}

\subsection{Proof of Theorem \ref{T7.2}} The proof of Theorem \ref{T7.2} is based on the fiberness detection result for balanced sutured manifolds, adapted to the case of mod 2 Novikov ring $\NR$:
\begin{theorem}[{\cite[Theorem 6.1]{KS}}]\label{T9.2}Suppose that a balanced sutured manifold $(M,\gamma)$ is taut and a homology product. Then $(M,\gamma)$ is a product sutured manifold if and only if $\SHM(M,\gamma)\cong \NR$. 
\end{theorem}

For the proof of Theorem \ref{T7.2}, it suffices to deal with the case when $\kappa$ is primitive. Let $F$ be the surface given by Lemma \ref{L9.1}. We will address the cases when $\chi(F)<0$ and when $\chi(F)=0$ separately. 

\begin{proof}[Proof of Theorem \ref{T7.2} when $\chi(F)<0$] Since $Y$ is irreducible, by Theorem \ref{T7.2} and Lemma \ref{L9.1}, we have
	\[
	\varphi_{\y}(\kappa)=\varphi_{\y}(-\kappa)=x(F). 
	\]
	Let $(\tilde{Y},\tilde{F})$ be the double of $(Y, F)$, then $g(\tilde{F})\geq 2$. Following the notations in the proof of Theorem \ref{T7.2}, we have 
	\begin{equation}\label{E9.2}
	\HM(\tilde{Y},[\tilde{\omega}]|\tilde{F})=\HM_*(\y|\kappa)\otimes_\NR \HM_*(-\y'|-\kappa)\cong \NR. 
	\end{equation}
	
	The rest of the proof is divided into six steps.
	
	\medskip
	
	\Step 1. We can effectively change the 2-form $\tilde{\omega}$ so that $\langle [\tilde{\omega}], [\tilde{F}]\rangle=0$. This is \Step 1 in the proof of Lemma \ref{L8.2}.

\medskip

\Step 2. Let $\tilde{M}$ be the 3-manifold with boundary obtained by cutting $\tilde{Y}$ along $\tilde{F}$. Write $\partial \tilde{M}=\tilde{F}_+\cup \tilde{F}_-$. Then $(\tilde{M},\tilde{F}_+)$ is a homology product, i.e. $H_*(\tF_+;\Z)\to H_*(\tilde{M};\Z)$ is an isomorphism. This follows from the fact \cite[Theorem 1]{T98} that for the graded Euler characteristic of $\HM_*(\tilde{Y},[\tilde{\omega}])$ recovers the Minor-Turaev torsion invariant $T(\tilde{Y})$. Then we apply \cite[Proposition 3.1]{Ni09b} and \eqref{E9.2}.

\medskip

\Step 3. $\HM(\tilde{Y}| \tilde{F})\cong \NR$. 

Let $N\colonequals [-1,1]\times \tF\subset \tilde{Y}$ be a tubular neighborhood of $\tF$. Then $\tilde{Y}\setminus N\cong \tilde{M}$.  We claim that $[\tilde{\omega}]$ is represented by a 2-form $\tilde{\omega}_1$ supported in $[-1,1]\times \tF$. This follows from \Step 1, \Step 2 and a diagram-chasing:
\[
\begin{tikzcd}
 H^2(\tilde{Y},\{1\}\times \tilde{F}) \arrow[r] & H^2(\tilde{Y}) \arrow[r,"{[\tilde{\omega}]\mapsto 0}"]& H^2(\tilde{F})\\
 H^2(\tilde{Y},\tilde{M}) \arrow[u,"\cong"]\arrow[r,"\cong"]&  H^2(N,\partial N). 
\end{tikzcd}
\] 
Given such a 2-form $\tilde{\omega}_1$, we use Floer's excision theorem as in the proof of Theorem \ref{T8.1} to deduce that $\HM(\tilde{Y}| \tilde{F})\cong \HM(\tilde{Y},[\tilde{\omega}]|\tilde{F})\cong \NR$. 

\medskip

\Step 4. Let $(M(Y,F),\gamma)$ be the balanced sutured manifold obtained by cutting $Y$ along $F$. Then $M(Y,F)$ is a homology product.

Note that $\tilde{M}$ is the double of $M(Y,F)$ along the annuli $A(\gamma)$. Write $\tilde{M}=M_1\cup M_2$ and $F_i\colonequals M_i\cap \tilde{F}_+$ for $i=1,2$. Then the statement follows by examining the long exact sequences:
\[
\begin{tikzcd}
\cdots\arrow[r] &H_*(F_1\cap F_2)\arrow[d,"\cong"] \arrow[r] & H_*(F_1)\oplus H_*(F_2) \arrow[r]\arrow[d] & H_*(\tilde{F}_+)\arrow[r] \arrow[d,"\cong"] &\cdots\\
\cdots\arrow[r] &H_*(M_1\cap M_2) \arrow[r] & H_*(M_1)\oplus H_*(M_2) \arrow[r] & H_*(\tilde{M}) \arrow[r] &\cdots 
\end{tikzcd}
\]

By \Step 3 and the five lemma, the middle vertical map is also an isomorphism. 
\medskip

\Step 5. If properties \ref{P5}\ref{P6} hold for $(\TSigma, \y, F)$, then Theorem \ref{T7.2} holds.  

In this special case, we can use  Theorem \ref{T8.1} and Lemma \ref{L9.1} to obtain that
\[
\SHM(M(Y,F),\gamma)\cong \HM(\y|[F])\cong \NR. 
\]
Since $Y$ is irreducible and $F$ is norm-minimizing, $(M(Y,F),\gamma)$ is taut. By \Step 4 and Theorem \ref{T9.2}, $M(Y,F)$ is a product sutured manifold. 

\medskip

\Step 6. Reduce the general case to \Step 5. 

Let $\TSigma_2$ be another $T$-surface with the same underlying oriented surface as $\TSigma$ and $\y_2\in \AT(\emptyset, \TSigma_2)$ has the same underlying 3-manifold as $\y$. We require that properties \ref{P5}\ref{P6} hold for $(\TSigma_2,\y_2, F)$. The goal is to show that $\HM(\y_2|[F])\cong \NR$. Let 
\[
(\tilde{Y},\tilde{\omega}_2,\cdots)=\y_2\circ_h (-\y_2').
\] 
By \Step 1, we may assume that $\langle [\tilde{\omega}_2], [F]\rangle=0$. By \Step 3, we have 
\[
\NR\cong \HM(\tilde{Y}, [\tilde{\omega}]|\tilde{F})\cong \HM(\tilde{Y}|\tilde{F})\cong \HM(\tilde{Y}, [\tilde{\omega}_2]|\tilde{F}). 
\]
By \eqref{E9.2}, $\HM(\y_2|[F])\cong \NR$. Now we use \Step 5 to complete the proof of Theorem \ref{T7.2} when $\chi(F)<0$. 
\end{proof} 

\begin{proof}[Proof of Theorem \ref{T7.2} when $\chi(F)=0$] In this case, $F$ is an annulus, $\Sigma$ has 2-components and the double $\tilde{F}$ is a 2-torus. By Lemma \ref{L9.1}, $\varphi_{\y}(\kappa)=\varphi_{\y}(-\kappa)=0$. Our assumptions then imply that $\HM_*(\y)\cong \NR$. 
	
The proof when $\chi(F)<0$ carries over with no essential changes. Let us explain where the differences arise:
\begin{itemize}
\item In \Step 1, we require instead that $\langle i[\tilde{\omega}], [F]\rangle =a$ for some fixed $a\in\R$ with $a\neq 0$ and $|a|<2\pi$;

\item In \Step 2, Ni's result \cite[Proposition 3.1b]{Ni09b} is stated for a closed surface with genus $\geq 2$; but its proof in \cite[Section 3.3]{Ni09b} relies only on the property of the Milnor-Turaev torsion invariant  $T(\tilde{Y})$ (note also that $b_1(\tilde{Y})\geq 3$). Thus we can still conclude from \eqref{E9.2} that $(\tilde{M},\tilde{F}_+)$ is a homology product;

\item \Step 3 is replaced by the isomorphism
\[
\HM(\tilde{Y},[\tilde{\omega}]|\tilde{F})\cong \HM(\tilde{Y},[\tilde{\omega}_2]|\tilde{F})
\]
for any classes $[\tilde{\omega}]=[\tilde{\omega}_2]\in H^2(\tilde{Y}, i\R)$ with $\langle i[\tilde{\omega}], [F]\rangle =\langle i[\tilde{\omega}_2], [F]\rangle=a$.  Then the difference $[\tilde{\omega}]-[\tilde{\omega}_2]$ is represented by a 2-form $\tilde{\omega}_1$ supported in the neighborhood $N=[-1,1]\times \tilde{F}$. As in the proof of Gluing Theorem \ref{T3.3}, one may adapt Floer's excision theorem \cite[Theorem 3.1]{KS} to the case of a 2-torus using non-exact perturbations, 
\end{itemize}

The rest of proof can now proceed with no difficulty. The conclusion says that $Y$ is a mapping torus over an annulus; so $Y=[-3,3]_s\times \T^2$ is in fact a cylinder. 
\end{proof}

As an immediate corollary of Theorem \ref{T7.2}, we have 
\begin{corollary}\label{C9.3} For any $T$-surface $\TSigma=(\Sigma, g_\Sigma,\lambda,\mu)$ and any 1-cell $\y\in \AT(\emptyset,\TSigma)$, if $Y$ is connected, irreducible and $\HM_*(\y,\bs)\cong \NR$, then $Y=[-1,1]_s\times \T^2$. 
\end{corollary}
\begin{proof}[Proof of Corollary \ref{C9.3}] Let $\kappa\in H_2(Y,\partial Y;\Z)$ be any primitive class, i.e., $\kappa$ is not divisible by any other integral classes non-trivially. By the symmetry of the graded Euler characteristic $\SW(Y)$ in Theorem \ref{T3.4}, the conditions of Theorem \ref{T7.2} are verified with $\varphi_{\y}(\kappa)=\varphi_{\y}(-\kappa)=0$. By Theorem \ref{T7.2}, $Y$ is a mapping torus over an annulus; so $Y=[-1,1]_s\times \T^2$ is a product. 
\end{proof}

\section{Connected Sum Formulae}\label{Sec10}

Having discussed irreducible 3-manifolds with toroidal boundary, we focus on reducible 3-manifolds in this section and derive the connected sum formulae.
\subsection{Connected Sums with 3-Manifolds with Toroidal Boundary}
For $i=1,2$, let $\TSigma_i$ be any $T$-surface and $\y_i\in \AT(\emptyset, \TSigma_i)$ be any 1-cell. Then we can form an 1-cell 
\[
\y_3=(Y_3,\omega_3,\cdots)\in \AT(\emptyset,\TSigma_1\cup \TSigma_2)
\] with the following properties:
\begin{enumerate}[label=(C\arabic*)]
\item\label{C1} the underlying 3-manifold $Y_3$ is $Y_1\# Y_2$; let $S^2\subset Y_3$ be the 2-sphere separating $Y_1$ and $Y_2$;
\item\label{C2} $[\omega_3]\in H^2(Y_3; i\R)$ is determined by the relations: $k_i([\omega_3])=l_i([\omega_i])$ for $i=1,2$ in the digram below. As a result, $\langle [\omega_3], [S^2]\rangle=0$. 
\[
\begin{tikzcd}
0 \arrow[r] & H^2(Y_1\# Y_2;i\R) \arrow[r,"{(k_1,k_2)}"] & H^2(Y_1\setminus B^3_1;i\R)\oplus H^2(Y_2\setminus B^3_2;i\R)\arrow[r] & H^2(S^2;i\R)\\
& &([\omega_1],[\omega_2])\in  H^2(Y_1;i\R)\oplus H^2(Y_2;i\R)\arrow[u,hook,"{l_1\oplus l_2}"]. & 
\end{tikzcd}
\]
\end{enumerate}

Although $\y_3$ is not uniquely determined by these properties, we say that $\y_3$ is a connected sum of $\y_1$ and $\y_2$. By Theorem \ref{T5.1}, the isomorphism type of $\HM_*(\y_3)$ is determined by \ref{C1}\ref{C2}. 
\begin{proposition}\label{P10.1} The monopole Floer homology of $\y_3$ can be computed as follows:
	\[
	\HM_*(\y_3)\cong\HM_*(\y_1)\otimes_\NR\HM_*(\y_2)\otimes_\NR V,
	\]
	where $V$ is a 2-dimensional vector space over $\NR$. 
\end{proposition}
\begin{proof} By Theorem \ref{T5.1}, we can compute the group $\HM_*(\y_3)$ using a convenient metric on $Y_1\# Y_2$ and a 2-form $\omega_3$. Take a component $\Sigma_i'\subset \Sigma_i$ for each $i=1,2$. Consider the 3-manifold
	\[
	([-3,3]_s\times \Sigma_i',\ \mu_i|_{\Sigma_i'}+ds\wedge \lambda_i|_{\Sigma_i'}), i=1,2.
	\]
Let $\y_4=(Y_4,\omega_4,\cdots)\in \AT(\TSigma_1'\cup \TSigma_2', \TSigma_1'\cup \TSigma_2')$ be their connected sum. The 1-cell $\y_3$ can be then obtained as the horizontal composition 
\[
(\y_1\coprod\y_2)\circ_h \y_4. 
\]
Using the Gluing Theorem \ref{T3.3}, we obtain that 
\[
\HM_*(\y_3)\cong \HM_*(\y_1)\otimes_\NR \HM_*(\y_2)\otimes_\NR \HM_*(\y_4). 
\]

It remains to verify that $\rank_\NR \HM_*(\y_4)=2$.  Regard $\y_4$ as a 1-cell in $\AT(\emptyset, (-\TSigma_1')\cup (-\TSigma_2')\cup \TSigma_1'\cup \TSigma_2')$ and consider the horizontal composition with 
\[
e_{\TSigma_1'}\in \AT((-\TSigma_1')\cup \TSigma_1',\emptyset) \text{ and }e_{\TSigma_2'}\in \AT((-\TSigma_2')\cup \TSigma_2',\emptyset).
\]
Another application of Theorem \ref{T3.3} implies that 
\[
\HM_*(\y_4)\cong \HM_*((\Sigma'_1\times S^1)\# (\Sigma'_2\times S^1),[\omega_1']\#[\omega_2']),
\]
where $\omega_i'=\mu_i|_{\Sigma_i'}+d\theta\wedge \lambda_i|_{\Sigma_i'}$ and $\theta$ denotes the coordinate of $S^1$. $[\omega_1']\# [\omega_2']$ is obtained from $[\omega_1']$ and $[\omega_2']$ using \ref{C2}. Here we have used the shorthands from Definition \ref{D6.4}. By Lemma \ref{L5.4}, we have already known that 
\[
\HM_*(\Sigma_i'\times S^1, [\omega_i'])\cong \NR \text{ for }i=1,2.
\]
To conclude, we apply the connected sum formula \cite[Theorem 5]{Lin17} for closed 3-manifolds. As a result, the group $\HM_*(\y_4)$ is concentrated in a single \spinc grading and has rank $2$. 
\end{proof}

\subsection{Connected Sums with Closed 3-Manifolds} Let us first review the definition of $\THM_*(M)$ for any closed 3-manifold $M$.

\begin{definition}\label{D10.2} For any closed 3-manifold $Z$, we obtain a balanced sutured manifold $(Z(1), \delta)$ by taking $Z(1)=Z\setminus B^3$ and the suture $s(\delta)\subset \partial Z(1)$ to be the equator. Define 
	\[
\THM(Z)\colonequals \SHM(Z(1), \delta).\qedhere
	\] 
\end{definition}

Let $\TSigma$ be any $T$-surface, $\y\in \AT(\emptyset, \TSigma)$ be any 1-cell and $Z$ be any closed 3-manifold. Consider an 1-cell $\y_5=(Y_5, \omega_5,\cdots)\in \AT(\emptyset,\TSigma)$ that satisfies the following properties: 
\begin{itemize}
\item the underlying 3-manifold $Y_5$ is $Y\# Z$;
\item $[\omega_5]\in H^2(Y_5;i\R)$ is obtained from $[\omega]$ using \ref{C2} and the zero form on $Z$.
\end{itemize}
\begin{proposition}\label{P10.3}The monopole Floer homology of $\y_5$ can be computed as follows:
	\[
	\HM_*(\y_5)\cong\HM_*(\y)\otimes_\NR\THM(Z).
	\]
\end{proposition}
\begin{proof} Following the proof of Proposition \ref{P10.1}, it suffices to verify that 
	\begin{equation}\label{E10.1}
	\THM(M)\cong \HM_*(M\#(\Sigma'\times S^1), [0]\#[\omega'])
	\end{equation}
	where $\Sigma'$ is a connected component of $\Sigma$ and $\omega'\colonequals \mu|_{\Sigma'}+d\theta\wedge \lambda|_{\Sigma'}$. One can argue as in \Step 1 of the proof of Lemma \ref{L8.2} and work instead with the 2-form 
	\[
	\omega''\equiv \mu|_{\Sigma'} \text{ on }S^1\times\Sigma'.
	\]
	Then the Poincar\'{e} dual of $[\omega'']$ is proportional to the 1-cycle $\{pt\}\times S^1$. 	 Since $[0]\#[\omega'']$ is neither balanced nor negatively monotone for any \spinc structure on $M\#(S^1\times \Sigma')$, one may apply Proposition \ref{P4.2} to compute the right hand side of \eqref{E10.1} using exact perturbations. Now the isomorphism \eqref{E10.1} follows from \cite[Proposition 4.6]{KS}.
\end{proof}


\bibliographystyle{alpha}
\bibliography{sample}

\end{document}